\documentclass[leqno,11pt]{article}
\usepackage[utf8]{inputenc}
\usepackage[T1]{fontenc}
\usepackage{microtype}

\usepackage[dvipsnames]{xcolor}
\usepackage{xcolor-solarized}

\usepackage{calc}
\usepackage[a4paper]{geometry}

\usepackage{amsmath}
\usepackage{amsthm}
\usepackage{tikz-cd}
\usetikzlibrary{arrows} 
\tikzset{
  commutative diagrams/.cd, 
  arrow style=tikz, 
  diagrams={>=stealth}
}
\usetikzlibrary{matrix,decorations.pathreplacing,calc}

\usepackage{textcomp}
\usepackage[sb]{libertine}
\usepackage[varqu,varl]{zi4}%
\usepackage[libertine,bigdelims,vvarbb]{newtxmath}

\usepackage[scr=boondox,cal=euler]{mathalfa}

\usepackage{marvosym}

\usepackage{slashed}
\usepackage{esint} 
\usepackage[english]{babel}

\usepackage{imakeidx}
\indexsetup{noclearpage}
\makeindex[intoc]

\usepackage{csquotes}
\usepackage[
  backend=biber,
  hyperref=true,
  backref=true,
  isbn=false,
  doi=true,
  natbib=true,
  eprint=true,
  useprefix=true,
  maxcitenames=99,
  maxbibnames=99,  
  maxalphanames=99, 
  minalphanames=99,
  safeinputenc,
  style=alphabetic,
  citestyle=alphabetic,
  block=space,
  datamodel=preamble/ext-eprint,
  sorting=nyt
]{biblatex}
\definecolor{darkblue}{HTML}{0000A6}

\usepackage[
  bookmarksnumbered = true,
  hypertexnames     = false,
  colorlinks        = true,
  citecolor         = darkblue,
  linkcolor         = darkblue,
  urlcolor          = darkblue,
  breaklinks
]{hyperref}

\DeclareFieldFormat{url}{%
  \href{#1}{\ComputerMouse}
}
\DeclareFieldFormat{doi}{%
  \mkbibacro{DOI}\addcolon\space\href{https://doi.org/#1}{#1}
}
\makeatletter
\DeclareFieldFormat{arxiv}{%
  arXiv\addcolon\space\href{http://arxiv.org/\abx@arxivpath/#1}{#1}
}
\makeatother
\DeclareFieldFormat{mr}{%
  MR\addcolon\space\href{http://www.ams.org/mathscinet-getitem?mr=MR#1}{#1}
}
\DeclareFieldFormat{zbl}{%
  Zbl\addcolon\space\href{http://zbmath.org/?q=an:#1}{#1}
}
\renewbibmacro*{eprint}{%
  \printfield{arxiv}%
  \newunit\newblock
  \printfield{mr}%
  \newunit\newblock
  \printfield{zbl}%
  \newunit\newblock
  \iffieldundef{eprinttype}
  {\printfield{eprint}}
  {\printfield[eprint:\strfield{eprinttype}]{eprint}}
}

\AtEveryBibitem{%
  \clearlist{address}%
}
\DeclareFieldFormat[article,inproceedings,inbook,incollection,thesis]{title}{\textit{#1}}
\renewbibmacro{in:}{}
\newcommand{\printreferences}{\raggedright\printbibliography[heading=bibintoc]}
\usepackage[inline,shortlabels]{enumitem}

\usepackage{subcaption}

\usepackage{etoolbox}
\ifundef{\abstract}{}{\patchcmd{\abstract}%
    {\quotation}{\quotation\noindent\ignorespaces}{}{}}

\usepackage[super]{nth}

\usepackage{thmtools}

\numberwithin{equation}{section}
\renewcommand{\qedsymbol}{$\blacksquare$}

\newcommand{\CorollaryQED}{\qedsymbol}

\newcommand{\ConjectureQED}{$\square$}
\newcommand{\SituationQED}{$\times$}
\newcommand{\DefinitionQED}{$\spadesuit$}
\newcommand{\NotationQED}{$\blacktriangleright$}
\newcommand{\ExampleQED}{$\bullet$}
\newcommand{\RemarkQED}{$\clubsuit$}

\declaretheorem[numberlike=equation,]{theorem}
\declaretheorem[numbered=no,name=Theorem]{theorem*}
\declaretheorem[numberlike=equation,name=Lemma]{lemma}
\declaretheorem[numberlike=equation,name=Proposition]{prop}
\declaretheorem[numberlike=equation,name=Corollary,qed=\CorollaryQED]{cor}

\declaretheorem[numberlike=equation,name=Definition,style=definition,qed=\DefinitionQED]{definition}
\declaretheorem[numbered=no,name=Definition,style=definition,qed=\DefinitionQED]{definition*}
\declaretheorem[numberlike=equation,name=Notation,style=definition,qed=\NotationQED]{notation}

\declaretheorem[numberlike=equation,style=definition,qed=\ExampleQED]{example}

\declaretheorem[numberlike=equation,style=remark,qed=\RemarkQED]{remark}
\declaretheorem[numbered=no,style=remark,name=Remark,qed=\RemarkQED]{remark*}

\declaretheorem[numberlike=equation,style=definition]{question}

\def\makeautorefname#1#2{\AtBeginDocument{\expandafter\def\csname#1autorefname\endcsname{#2}}}
\makeautorefname{table}{Table}        
\makeautorefname{chapter}{Chapter}
\makeautorefname{section}{Section}
\makeautorefname{subsection}{Section}
\makeautorefname{subsubsection}{Section}
\makeautorefname{footnote}{Footnote}
\AtBeginDocument{\def\itemautorefname~#1\null{(#1)\null}}
\AtBeginDocument{\def\equationautorefname~#1\null{(#1)\null}}

\numberwithin{substep}{step}
\makeautorefname{step}{Step}
\makeautorefname{substep}{Step}

\makeautorefname{case}{Case}
\makeautorefname{substep}{Step}

\setlist[description]{leftmargin=!,labelindent=1em}
\setlist[enumerate]{label={\rm (\arabic*)},ref=\arabic*}
\setlist[enumerate,2]{label={\rm (\alph*)},ref=\theenumi.\alph*}
\setlist[enumerate,3]{label={\rm (\roman*)},ref=\theenumii.\roman*}

\let\C\undefined

\usepackage{bm}
\usepackage{mathtools} 
\usepackage{stmaryrd} 

\DeclareFontFamily{U}{mathx}{\hyphenchar\font45}
\DeclareFontShape{U}{mathx}{m}{n}{
      <5> <6> <7> <8> <9> <10>
      <10.95> <12> <14.4> <17.28> <20.74> <24.88>
      mathx10
      }{}
\DeclareSymbolFont{mathx}{U}{mathx}{m}{n}
\DeclareFontSubstitution{U}{mathx}{m}{n}
\DeclareMathAccent{\widecheck}{0}{mathx}{"71}
\DeclareMathAccent{\wideparen}{0}{mathx}{"75}

\DeclareMathOperator{\Aut}{Aut}

\DeclareMathOperator{\Diff}{Diff}
\DeclareMathOperator{\End}{End}
\DeclareMathOperator{\Ext}{Ext}

\DeclareMathOperator{\GL}{GL}

\DeclareMathOperator{\graph}{graph}
\DeclareMathOperator{\HF}{\HF}

\DeclareMathOperator{\Hol}{Hol}
\DeclareMathOperator{\Hom}{Hom}

\DeclareMathOperator{\coker}{coker}

\DeclareMathOperator{\diag}{diag}

\DeclareMathOperator{\im}{im}
\DeclareMathOperator{\ind}{index}

\DeclareMathOperator{\spec}{spec}

\DeclarePairedDelimiter{\norm}{\|}{\|}

\DeclarePairedDelimiterX{\inp}[2]{\langle}{\rangle}{#1, #2}

\DeclarePairedDelimiter{\abs}{\lvert}{\rvert}

\def\({\left(}
\def\){\right)}
\def\<{\left\langle}
\def\>{\right\rangle}

\newcommand{\C}{{\mathbf{C}}}

\newcommand{\N}{{\mathbf{N}}}

\newcommand{\Q}{\mathbf{Q}}

\newcommand{\R}{\mathbf{R}}

\newcommand{\SO}{\mathrm{SO}}
\newcommand{\SU}{\mathrm{SU}}

\newcommand{\Vect}{\mathrm{Vect}}

\newcommand{\Z}{\mathbf{Z}}

\newcommand{\id}{\mathbf 1}

\newcommand{\ob}{\mathrm{ob}}

\newcommand{\qandq}{\quad\text{and}\quad}

\newcommand{\reg}{\mathrm{reg}}

\newcommand{\sing}{\mathrm{sing}}

\renewcommand{\epsilon}{\varepsilon}

\newcommand{\vol}{\mathrm{vol}}

\renewcommand{\Im}{\operatorname{Im}}
\renewcommand{\O}{\mathrm{O}}

\renewcommand{\Re}{\operatorname{Re}}

\renewcommand{\emptyset}{\varnothing}
\renewcommand{\setminus}{{\backslash}}

\renewcommand{\leq}{\leqslant}
\renewcommand{\geq}{\geqslant}

\newcommand{\w}{\wedge}
\newcommand{\tn}{\otimes}

\makeatletter
\renewcommand*\env@matrix[1][*\c@MaxMatrixCols c]{%
  \hskip -\arraycolsep
  \let\@ifnextchar\new@ifnextchar
  \array{#1}}

\renewcommand\xleftrightarrow[2][]{%
  \ext@arrow 9999{\longleftrightarrowfill@}{#1}{#2}}
\newcommand\longleftrightarrowfill@{%
  \arrowfill@\leftarrow\relbar\rightarrow}
\makeatother

\newcommand{\bn}{{\mathbf{n}}}

\newcommand{\bD}{{\mathbf{D}}}

\newcommand{\bR}{{\mathbf{R}}}

\newcommand{\cD}{\mathcal{D}}
\newcommand{\cE}{\mathcal{E}}

\newcommand{\cI}{\mathcal{I}}

\newcommand{\cK}{\mathcal{K}}

\newcommand{\cM}{\mathcal{M}}

\newcommand{\cS}{\mathcal{S}}

\newcommand{\cX}{\mathcal{X}}

\newcommand{\cZ}{\mathcal{Z}}

\newcommand{\sP}{\mathscr{P}}

\newcommand{\fc}{{\mathfrak c}}

\newcommand{\fm}{{\mathfrak m}}

\newcommand{\fF}{{\mathfrak F}}

\newcommand{\fL}{{\mathfrak L}}

\newcommand{\ocs}{{\operatorname{cs}}}

\newcommand{\bsP}{{\bm{\sP}}}

\newcommand{\bcM}{{\bm{\cM}}}

\newcommand{\bphi}{{\bm{\phi}}}

\author{Gorapada Bera}
\title{Deformations and desingularizations of conically singular associative submanifolds}
\date{\vspace{-5ex}}

\begin{document}
\maketitle
\begin{abstract}The proposals of \citet{Joyce2016} and Doan--Walpuski \cite{Doan2017d} on counting closed associative submanifolds of $G_2$-manifolds depend on various conjectural transitions. This article contributes to the study of transitions arising from the degenerations of associative submanifolds into conically singular (CS) associative submanifolds. First, we study the moduli space of CS associative submanifolds with isolated singularities modeled on associative cones in $\R^7$, establishing transversality results in both fixed and one-parameter families of co-closed $G_2 $-structures. We prove that for a generic co-closed $G_2$-structure (or a generic path thereof) there are no CS associative submanifolds having singularities modeled on cones with stability-index greater than $0$ (or $1$, respectively). We establish that associative cones whose links are null-torsion holomorphic curves in $S^6$ have stability-index greater than $4$, and all special Lagrangian cones in $\C^3$ have stability-index greater than or equal to $1$ with equality only for the Harvey--Lawson $T^2$-cone and a transverse pair of planes. Next, we study the desingularizations of CS associative submanifolds in a one-parameter family of co-closed $G_2$-structures. Consequently, we derive desingularization results relating the above transitions for CS associative submanifolds with a Harvey--Lawson $T^2$-cone singularity and for associative submanifolds with a transverse self-intersection.  
\end{abstract}

\section{Introduction: Main results}
\citet{Joyce2016} and Doan--Walpuski \cite{Doan2017d} have made proposals of constructing enumerative invariants of $G_2$-manifolds by counting closed associative submanifolds. It has been observed that the  counting (possibly with sign) of associative submanifolds does not lead to an invariant due to various transitions that may occur along a generic path of $G_2$-structures. One type of such transition is the degeneration of associative submanifolds into conically singular associative submanifolds. In particular, it has been conjectured (see \citet[Conjectures 4.4 and 5.3]{Joyce2016}) that at least
the following transitions can occur along a generic path of $G_2$-structures $\phi_t$ on a $7$-dimensional manifold $Y$. 
\begin{enumerate}[1.]	
\item Three families of embedded closed associative submanifolds (see \autoref{Fig_T2Singularities}), $P^1_t$ with $-T<t<0$ and $P^2_t, P^3_t$ with $0<t<T$ in $(Y,\phi_t)$, converge in the sense of currents to an associative submanifold $P$ with a Harvey--Lawson (HL) $T^2$-cone singularity at $x$ in $(Y,\phi_{0})$, as $t\to 0$. The $P_t^i$ are diffeomorphic to the Dehn fillings of $P^o:=P\setminus B(x)$ along  simple closed curves $\mu_i\subset \partial P^o\cong T^2$ which satisfy $\mu_1.\mu_2=\mu_2.\mu_3=\mu_3.\mu_1=-1$.
\begin{figure}[h]
  \centering
  \begin{tikzpicture}
    \draw[thick,magenta] (-3,1.5) node [left] {$P_t^1$}  -- (0,1.5);
    \draw[thick,violet] (0,1.5) .. controls (0,2) and (2.75,2.5).. (3,2.5) node [right] {$P_t^2$} ;
    \draw[thick,cyan] (0,1.5) .. controls (0,1) and (2.75,.5).. (3,.5) node [right] {$P_t^3$};
    \filldraw[red] (0,1.5) node [above left] {$P$} circle (0.05);
    \draw[gray] (-3,0) -- (0,0);
    \draw[|-stealth,gray] (0,0) -- (3,0) node [right] {$\phi_t$};
  \end{tikzpicture}  
  \caption{Three associatives arising from a singular associative.}
  \label{Fig_T2Singularities}
\end{figure}
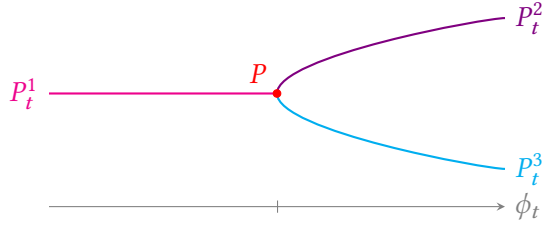

\

 \item Two families of embedded closed associative submanifolds (see \autoref{Fig_IntersectingAssociatives}), $P_t$ with $0\neq t\in (-T,T)$ and $P^\#_t$ with $0<t<T$ in $(Y,\phi_t)$, converge in the sense of currents to an associative submanifold $P$ with a self intersection in $(Y,\phi_{0})$, as $t\to 0$. The $P^\#_t$ are diffeomorphic to the connected sums $P_t\#(S^1\times S^2)$ if $P_t$ are connected, and otherwise to $P_t^+\#P_t^-$, where $P_t=P_t^+\amalg P_t^-$.	
 	\begin{figure}[h]
  \centering
  \begin{tikzpicture}
    \draw[thick,cyan] (-3,.5) -- (3,.5) node[right] {$P_t$};
    \draw[thick,Blue] (0,.5) .. controls (0,1.5) and (2,1.5) .. (3,1.5) node[right] {$P^\#_t$}; 
    \filldraw[red] (0,.5) node [above left] {$P$} circle (0.05);
    \draw[gray] (-3,0) -- (0,0);
    \draw[|-stealth,gray] (0,0) -- (3,0) node [right] {$\phi_t$};
  \end{tikzpicture}  
  \caption{Birth of an associative out of an associative with self intersection.}
  \label{Fig_IntersectingAssociatives}
\end{figure}
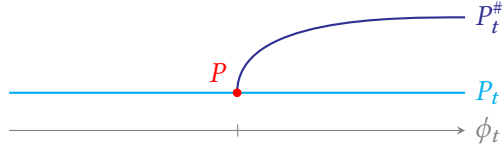

\end{enumerate}

One of the contributions of this article is to confirm the above two transitions, provided there is an associative submanifold with a Harvey--Lawson $T^2$-cone singularity (see \autoref{thm desing HL sing}) and an associative submanifold with a self intersection (see \autoref{thm desing intersection}), respectively. More generally, we prove a desingularization theorem (see \autoref{thm main desing}) for conically singular associative submanifolds under a certain hypothesis \autoref{hyp desing CS}. This is done in a $1$-parameter family of co-closed $G_2$-structures\footnote{These structures are weaker than torsion-free $G_2$-structures, but torsion-free structures form a special subclass; see \autoref{def G2 manifold}. We work with co-closed rather than closed structures, as the deformation operator for associatives is self-adjoint only in the co-closed case—a property used extensively in this article, and also being used in \cite{Joyce2016} to define canonical orientations. While closed structures are more suited to compactness issues, \citet[Section 2.5]{Joyce2016} addresses this by introducing the subclass of tamed structures. Since compactness is not relevant here, co-closed structures suffice.} by gluing rescaled asymptotically conical associative submanifolds of $\R^7$. Consequently we derive the first transition (see \autoref{thm desing HL sing}) and partially derive the second transition (see \autoref{thm desing intersection}). 

Desingularization results of a similar nature have been obtained previously for special Lagrangians in Calabi--Yau manifolds and for coassociatives in $G_2 $-manifolds \cites{Joyce2004c, Joyce2004d, Lotay2009a, Lotay2014}, but the associative case poses significantly greater difficulties. In particular, to account for the scaling freedom in the gluing construction, the $G_2 $-structure must vary in a one-parameter family; otherwise, one would produce a one-parameter family of associatives in a fixed $G_2$-manifold, contradicting the expected dimension zero. This necessity, captured in the hypothesis \autoref{hyp desing CS}, introduces further analytical complexity. Moreover, in contrast to the special Lagrangian and coassociative cases, associative deformations can be obstructed even in the smooth setting, making the problem substantially more delicate.

Another contribution of this article is the study of the deformation theory of conically singular (CS) associative submanifolds. This not only underpins the aforementioned desingularization results but also makes further progress toward the conjectural enumerative theories proposed by \citet{Joyce2016} and Doan--Walpuski \cite{Doan2017d} for $G_2$-manifolds. Specifically, one needs to understand all possible degenerations of closed associative submanifolds that may arise in a generic $d$-parameter family of co-closed $G_2$-structures for $d=0,1$. It is known in geometric measure theory (see \autoref{thm cptness current}) that the associative submanifolds can only degenerate into an associative integral current with a singular set of Hausdorff dimension at most $1$. However, the precise regularity of this singular set remains unknown.  If all the tangent cones are Jacobi integrable multiplicity $1$ associative cones with smooth link, then the associative integral current is conically singular with isolated singular points (see \autoref{def CS asso}). This naturally leads to the following question.
  \begin{question}\label{question all CS asso types}
  	What are all the possible conically singular (CS) associative submanifolds that may appear in a generic $d$-parameter family of co-closed $G_2$-structures with $d=0,1$?
  \end{question}
 To answer this question we study the deformation theory of conically singular associative submanifolds. We explain (see \autoref{thm moduli cs asso}) that the index of the deformation operator can be expressed completely in terms of a certain non-negative integer associated to the tangent cones, which we define to be the stability-index (see \autoref{def stability index}). In particular, CS associative submanifolds having one singularity modeled on cones with stability-index $0$ and $1$ can only appear in a generic $0$ and $1$-parameter family of co-closed $G_2$-structures, respectively (see \autoref{thm generic moduli cs asso}). Therefore we ask the following subsequent question. 
\begin{question}\label{question all asso cones stabilty index leq 1}
  	What are all the associative cones in $\R^7$ with stability-index equal to $0$ and $1$?
  \end{question}
 We establish that the Harvey--Lawson $T^2$-cone and a union of two transverse special Lagrangian planes have stability-index $1$ (see \autoref{thm stability index of cones}). Also we prove with the help of a result obtained by \citet{Haskins2004a} that all other special Lagrangian cones in $\C^3$ have stability-index strictly greater than $1$. Furthermore, building on a result of \citet{Madnick2021}, we prove that all associative cones in $\R^7$ whose links are null-torsion holomorphic curves in $S^6$ (see \citet{Bryant1982}) have stability-index strictly greater than $4$. In particular, this applies to all associative cones whose link in $S^6$ is of genus $0$ but not a totally geodesic sphere. We also establish that associative cones with link of genus $1$ always have stability-index greater than or equal to $1$. Although these results do not fully resolve \autoref{question all asso cones stabilty index leq 1}, they do rule out a large class of associative cones, narrowing down the possibilities and paving the way for future exploration.

We now provide a summary of the main results established in this article, encompassing the discussions above.
\subsection*{Associative cones}
Since our aim is to study the deformation theory of conically singular associative submanifolds-- where the associative cones play an important role-- we begin in \autoref{subsection Associative cones} with an investigation of associative cones in $\R^7$. These are the cones whose links in $(S^6,J)$ are holomorphic curves (see \autoref{subsection Associative cones}). For this purpose, we express the moduli space of holomorphic curves in $S^6$ locally as the zero set of a nonlinear map (see \autoref{def nonlinear holo in S6}), as described in the following theorem.

 \begin{theorem}\label{thm moduli holo}Let $\Sigma$ be a closed holomorphic curve in $(S^6,J)$. Then a neighbourhood of $\Sigma$ in the moduli space of holomorphic curves  is homeomorphic to the zero set of a smooth map (often called the obstruction map or Kuranishi map):
$$\ob_\Sigma:\mathcal I_\Sigma \to \coker(\bD_\Sigma+2J),$$
where the operator $\bD_\Sigma$ is defined in \autoref{eq dirac sigma} and $\mathcal I_\Sigma$ is an open neighbourhood of $0$ in $\ker(\bD_\Sigma+2J)$.
Moreover, the index of the deformation operator~ $\mathbf D_\Sigma+2J$ is zero. This deformation operator is related to another standard deformation operator, the normal Cauchy-Riemann operator $\bar\partial^N_{\nabla}$ (see \autoref{def normal Cauchy Riemann}), by a $J$-anti-linear isomorphism $\boldsymbol{\gamma}_\Sigma: \overline{\operatorname{Hom}}_\C(T\Sigma,N\Sigma)\to {N\Sigma}$ as follows:  $$\boldsymbol{\gamma}_\Sigma\circ \bar\partial^N_{\nabla}=\mathbf D_\Sigma+2J.$$
\end{theorem}

We denote by $\mathcal{M}^{\mathrm{hol}}$ the moduli space of closed holomorphic curves in $(S^6,J)$. Its subset $\mathcal{M}^{\mathrm{hol}}_{\bullet}$ consists of closed, connected holomorphic curves. This space naturally carries a structure of a real analytic space, essentially by \autoref{thm moduli holo}, but it is not necessarily a smooth manifold. In order to establish the results in \autoref{thm moduli cs asso} and \autoref{thm generic moduli cs asso} concerning the moduli space of conically singular associative submanifolds and their transversality properties, we equip $\mathcal{M}^{\mathrm{hol}}_{\bullet}$ with a \textbf{canonical minimal Whitney stratification} in the sense of \cite[Chapter I, Sections 1–2]{Gibson1976}, a decomposition
\begin{equation}\label{eq stratification}
\mathcal{M}^{\mathrm{hol}}_{\bullet} = \bigsqcup_{k \in I} \mathcal{Z}^{(k)},
\end{equation}
where
\begin{itemize}
    \item $I \subset \Z_{\geq 0}$ is a countable index set corresponding to the dimensions of the strata,
    \item each stratum $\mathcal{Z}^{(k)}$ is a smooth manifold of dimension $k$ and is preserved under the natural action of the group $G_2$,
    \item the strata are pairwise disjoint and collectively form a Whitney stratification,
      \item the stratification is minimal and canonical:  it consists of the fewest strata necessary to satisfy the above properties and is uniquely determined up to local equivalence of stratified spaces.
\end{itemize}
Such a canonical minimal Whitney stratification exists for any real analytic space, as established in classical results (see, for example, \cites{Hironaka1973,Hardt1975}). While such stratifications are not necessarily unique globally, any two are locally equivalent as stratified spaces.  This guarantees that many of the quantities considered below are independent of the particular choice of the above stratification.

Using any canonical minimal Whitney stratification as in \autoref{eq stratification}, we associate to each associative cone $C$ in $\R^7$, an integer invariant $\operatorname{s-ind}(C)$, referred to as the stability index. It turns out that the negative of this stability-index is essentially the virtual dimension of the moduli space of conically singular associative submanifolds modeled on $C$.

\begin{definition}[Stability-index]\label{def stability index}Let $C$ be an associative cone in $\R^7$. Denote the link by $\Sigma$, which is a closed $J$-holomorphic curve in $S^6$. Let $d_\lambda$ be the dimension of homogeneous kernels from \autoref{def homogeneous kernel}. Let $\Sigma=\amalg_{j=1}^l\Sigma_j$ be the decomposition into connected components.  Let $\cZ_j$, $j=1,\dots,l$, be the stratum in the decomposition \autoref{eq stratification} containing $\Sigma_j$. The \textbf{stability index} of $C$ is defined by
\[\operatorname{s-ind}(C):=\frac {d_{-1}}2+\displaystyle\sum_{-1<\lambda\leq 1}d_\lambda-7-\sum_{j=1}^l\dim \cZ_j.\qedhere\] 
\end{definition}
\begin{remark}Since canonical minimal Whitney stratifications are locally equivalent and $\dim \cZ_j$ coincides with the dimension of the tangent space of $\cZ_j$ at $\Sigma_j$,  $\operatorname{s-ind}(C)$ is independent of the particular choice of such a stratification. We will see in \autoref{rmk d0 geq 7} that if $C$ is not a $3$-plane then $\operatorname{s-ind}(C)\geq 0$. 
\end{remark}
In what follows, we introduce the notions of upper and lower stability indices, which provide effective upper and lower bounds on the stability index.
\begin{definition}[Upper and lower stability-indices]\label{def uplo stability index}
Let $C$ be an associative cone with link components $\Sigma_j$ as in \autoref{def stability index}. Denote the cone of $\Sigma_j$ by $C_j$. Let $H_j$ be the maximal subgroup of $G_2$ that fixes $\Sigma_j$ under the standard action of $G_2$ on $S^6$. The subgroup $H_j$ is called the \textbf{symmetry group} of ~$C_j$. We define the \textbf{upper stability-index} by
 \begin{equation*}
\operatorname{s-ind}_+(C):= \frac {d_{-1}}2+\displaystyle\sum_{-1<\lambda\leq 1}d_\lambda-7-\displaystyle\sum_{j=1}^l(\dim{G_2}-\dim{H_j}), 
\end{equation*}
and the \textbf{lower stability-index} by
 \begin{equation*}
\operatorname{s-ind}_-(C):= \frac {d_{-1}}2+\displaystyle\sum_{-1<\lambda< 1}d_\lambda-7. \qedhere
\end{equation*}
\end{definition}

\begin{remark}
Since $d_1$ represents the sum of the dimensions of the space of infinitesimal deformations of $\Sigma_j$, which definitely contains all the actual deformations induced by the $G_2$-action, it is larger than the sum of $\dim {G_2}/{H_j}$. Furthermore, since each stratum is $G_2$-invariant, $\cZ_j$ always contains $G_2 \cdot \Sigma_j$ and therefore $\dim \cZ_j \geq \dim {G_2}/{H_j}$. Hence,
\[
\operatorname{s-ind}_+(C) \geq \operatorname{s-ind}(C) \geq \operatorname{s-ind}_-(C). \qedhere
\]
\end{remark}

Finally, we introduce a natural class of cones that will be the primary focus of this article: the rigid cones.
\begin{definition}[Rigid associative cones]\label{def rigid cones}
An  associative cone $C$ is said to be \textbf{rigid} if $$\operatorname{s-ind}_+(C)=\operatorname{s-ind}_-(C),$$
 or equivalently all the infinitesimal deformations of each component of the link are induced by the $G_2$-action, that is, 
\begin{equation*}
 	d_1=\displaystyle\sum_{j=1}^l(\dim{G_2}-\dim{H_j}).\qedhere
 \end{equation*}
\end{definition}
Having introduced the necessary definitions, we now present the theorem proved in \autoref{subsection Associative cones}.
\begin{theorem}\label{thm stability index of cones}Let $C$ be an associative cone in $\R^7$ with link $\Sigma \subset S^6$.
\begin{enumerate}[(i)]
\item  If the genus of $\Sigma$ is $1$ then $\operatorname{s-ind}_-(C)\geq 1.$
\item If $\Sigma$ is a {null torsion} holomorphic curve in $S^6$ (see \autoref{eg Bryant null torsion}) then $$\operatorname{s-ind}_-(C)>4.$$ In particular this holds for any holomorphic curve of genus $0$ in $S^6$ which is not a totally geodesic sphere. 
\item If $C$ is the Harvey--Lawson $T^2$-cone (see \autoref{eg Harvey-Lawson $T^2$-cone}) or a union of two special Lagrangian planes with transverse intersection at the origin (see \autoref{eg Transverse pair of SL planes}) then it is rigid and $$\operatorname{s-ind}(C)=\operatorname{s-ind}_\pm(C)=1.$$	
\item If $C$ is a special Lagrangian cone in $\C^3$ that is not a plane then
\begin{equation*}
	\operatorname{s-ind}(C)\geq \operatorname{s-ind}_-(C)\geq \frac {b^1(\Sigma)}2+b^0(\Sigma)-1\geq 1
\end{equation*}
with equality if and only if $C$ is one of the cones in part (iii).
\end{enumerate}
  \end{theorem}
  
  \subsection*{Moduli space of conically singular associative submanifolds}
Let $(Y,\phi)$ be a co-closed $G_2$-manifold (see \autoref{def coclosed and tamed G2}). We denote by $\sP$ the space of all co-closed $G_2$-structures on $Y$ and by $\bsP$ the space of all smooth paths $[0,1]\to \sP$. We consider conically singular (CS) associative submanifolds with isolated singularities at a finite number of points in $(Y,\phi)$ (see \autoref{def CS asso}). These singularities are locally modeled on associative cones in $\R^7$. The moduli space of all CS associative submanifolds in $(Y,\phi)$ is denoted by $\cM_\ocs^\phi$ (see \autoref{def moduli CS asso}).  Given a path of co-closed $G_2$-structures $\bphi\in \bsP$, the $1$-parameter moduli space of all CS associative submanifolds is denoted by $\bcM_{\operatorname{cs}}^\bphi$ (see \autoref{def 1 para moduli CS}). More explicitly,
\begin{equation*}
	\bcM_\ocs^\bphi=\{(t,P): t\in [0,1], P\in \cM_\ocs^{\phi_t} \}.
\end{equation*} Both these moduli spaces are equipped with the weighted $C^\infty$ topology (see \autoref{def CS weighted topo}). We will now decompose the moduli spaces $\cM_{\operatorname{cs}}^\phi$ and $\bcM_{\operatorname{cs}}^\bphi$ as a countable union of sub-moduli spaces whose deformation theory will be studied.

\begin{definition}\label{def decomp CS asso moduli}
 Let $\cZ=\prod_{i=1}^m\cZ_i$ where $\cZ_i= \prod_{j=1}^l\cZ^j_i$ with $\cZ^j_i$ one of the strata in the decomposition \autoref{eq stratification}.
We denote by $\cM_{\operatorname{cs},\cZ}^\phi$ the set of all CS associative submanifolds with $m$ conical singularities, whose asymptotic cones $C_i$ have links $\Sigma_i=\sqcup_{j=1}^l \Sigma_i^j$ with $\Sigma_i^j\in \cZ_i^j$. We also define  \begin{equation*}
	\bcM_{\ocs,\cZ}^\bphi:=\{(t,P): t\in [0,1], P\in \cM_{\ocs,\cZ}^{\phi_t} \}.
\end{equation*}
 We can express the moduli spaces $\cM_{\operatorname{cs}}^\phi$ and $\bcM_{\operatorname{cs}}^\bphi$ as a countable union of sub-moduli spaces, namely
\[\cM_{\operatorname{cs}}^\phi=\bigcup_{\cZ}\cM_{\operatorname{cs},\cZ}^\phi\ \ \text{and}\ \  \bcM_{\operatorname{cs}}^\bphi=\bigcup_{\cZ}\bcM_{\operatorname{cs},\cZ}^\bphi.\]
Here the unions run over all possible $\cZ=\prod_{i=1}^m\cZ_i$ where $\cZ_i= \prod_{j=1}^l\cZ^j_i$ as above. \end{definition}
We prove the following thereom in \autoref{subsection moduli CS asso} about local Kuranishi models for the above moduli spaces.
\begin{theorem}\label{thm moduli cs asso} Let $\cZ$, $\cM_{\operatorname{cs},\cZ}^\phi$ and $\bcM_{\operatorname{cs},\cZ}^\bphi$ be as in \autoref{def decomp CS asso moduli}. Let $P\in \cM_{\operatorname{cs},\cZ}^\phi$ be a conically singular associative submanifold in a co-closed $G_2$-manifold $(Y,\phi)$ with $m$ singularities modeled on cones $C_i\in \cZ_i$, $i=1,\dots,m$ (see \autoref{def CS asso}). Let $\bphi:[0,1]\to \sP$ be a path of co-closed $G_2$ structures such that $\phi(t_0)=\phi$. Then there exist open neighbourhoods $\widetilde \cI_{P,\cZ}$ and $\widebar \cI_{P,\cZ}$ of $0$ in $\ker \widetilde \bD_{P,\mu,\cZ}$ and $\ker \widebar  \bD_{P,\mu,\cZ}$ respectively (where $\widetilde \bD_{P,\mu,\cZ}$  and $ \widebar \bD_{P,\mu,\cZ}$ are defined in \autoref{def lin L, D with Z} and \autoref{def 1 para moduli CS} respectively and $\mu$ is chosen as in \autoref{def CS weighted topo}) such that
\begin{enumerate}[(i)]
\item the moduli space $\cM_{\operatorname{cs},\cZ}^\phi$ near $P$ is homeomorphic to $\ob_{P,\cZ}^{-1}(0)$, the zero set of a smooth map
 $$\ob_{P,\cZ}:\widetilde \cI_{P,\cZ}\to \coker \widetilde \bD_{P,\mu,\cZ}.$$
 Moreover,
 $\ind \widetilde \bD_{P,\mu,\cZ}=-\displaystyle\sum_{i=1}^m\operatorname{s-ind}(C_i).$
\item  the $1$-parameter moduli space $\bcM_{\operatorname{cs},\cZ}^\bphi$ near $(t_0,P)$ is homeomorphic to $\ob_{t_0,P,\cZ}^{-1}(0)$, the zero set of a smooth map
 $$\ob_{t_0,P,\cZ}: \widebar \cI_{P,\cZ}\to \coker \widebar  \bD_{P,\mu,\cZ}.$$
 Moreover, $\ind \widebar  \bD_{P,\mu,\cZ}=\ind \widetilde \bD_{P,\mu,\cZ}+1=-\displaystyle\sum_{i=1}^m\operatorname{s-ind}(C_i)+1$.
\end{enumerate}
   \end{theorem}

 The moduli spaces $\cM_{\operatorname{cs},\cZ}^\phi$ and $\bcM_{\operatorname{cs},\cZ}^\bphi$ are not always smooth manifolds. The following definition will make these moduli spaces smooth in a generic situation.
  \begin{definition}\label{def regular G2 str}
Let $\cZ$, $\cM_{\operatorname{cs},\cZ}^\phi$ and $\bcM_{\operatorname{cs},\cZ}^\bphi$ be as in \autoref{def decomp CS asso moduli}. We define $\sP_{\operatorname{cs},\cZ}^{\operatorname{reg}}$ to be the subset consisting of all $\phi\in \sP$ with the property that for all CS associative $P \in \cM_{\operatorname{cs},\cZ}^\phi$ the linear operator 
$$\widetilde \bD_{P,\mu,\cZ} \ \ \text{is surjective}. $$
 Similarly we define $\bsP_{\operatorname{cs},\cZ}^{\operatorname{reg}}$ to be the subset consisting of all $\bphi\in \bsP$ with the property that for all $(t_0,P)\in \cM_{\operatorname{cs},\cZ}^{\bphi}$ the linear operator 
   \begin{equation*}
	\widebar  \bD_{P,\mu,\cZ} \ \ \text{is surjective}.  
	\end{equation*}
	The operators $\widetilde \bD_{P,\mu,\cZ}$  and $\widebar  \bD_{P,\mu,\cZ}$ are defined in \autoref{def lin L, D with Z} and \autoref{def 1 para moduli CS} respectively. 

In addition we define
\[\sP_{\operatorname{cs}}^{\operatorname{reg}}:=\bigcap_{\cZ}\sP_{\operatorname{cs},\cZ}^{\operatorname{reg}}\ \ \text{and}\ \ \ \bsP_{\operatorname{cs}}^{\operatorname{reg}}:=\bigcap_{\cZ}\bsP_{\operatorname{cs},\cZ}^{\operatorname{reg}}.\]
Here the intersections run over all possible $\cZ=\prod_{i=1}^m\cZ_i$ where $\cZ_i= \prod_{j=1}^l\cZ^j_i$ as in \autoref{def decomp CS asso moduli} and hence are countable intersections. 
	 \end{definition}
\begin{remark}
Since the definitions of the operators $\widetilde \bD_{P,\mu,\cZ}$ and $\widebar \bD_{P,\mu,\cZ}$ in \autoref{def lin L, D with Z} and \autoref{def 1 para moduli CS} depend only on the tangent space to the stratum $\cZ_i^j$ at $\Sigma_i^j$, and not on the full stratum itself, the spaces $\sP_{\operatorname{cs}}^{\operatorname{reg}}$ and $\bsP_{\operatorname{cs}}^{\operatorname{reg}}$ are independent of the choice of a canonical minimal Whitney stratification \autoref{eq stratification}.
\end{remark}	
We prove the following theorem about generic transversality of the above moduli spaces in \autoref{subsection generic moduli CS asso}. It also tells us about what type of singularity model cones appear in a generic co-closed $G_2$-structure, as well as in a generic path of co-closed $G_2$-structures. Before stating the theorem, let us clarify the notion of generic.

\begin{definition}
Let $X$ be a topological space and let $S \subset X$. The set $S$ is said to be \textbf{meager} if it is contained in a countable union of closed subsets with empty interior. The complement of a meager set is called \textbf{comeager}.
\end{definition}

\begin{remark}
In a completely metrizable space (such as $\sP$ or $\bsP$), the Baire category theorem asserts that meager sets have empty interior, and comeager sets are necessarily dense. Therefore, meager sets are often viewed as non-generic, while comeager sets are seen as generic.
\end{remark}

\begin{theorem}\label{thm generic moduli cs asso} Let $\cZ=\prod_{i=1}^m\cZ_i$ where $\cZ_i= \prod_{j=1}^l\cZ^j_i$ as in \autoref{def decomp CS asso moduli}. Then the subsets $\sP_{\operatorname{cs},\cZ}^{\operatorname{reg}}$, $\sP_{\operatorname{cs}}^{\operatorname{reg}}$ in $\sP$, and $\bsP_{\operatorname{cs},\cZ}^{\operatorname{reg}}$, $\bsP_{\operatorname{cs}}^{\operatorname{reg}}$ in $\bsP$  are comeager. 
 In particular, the following holds. Let $C_i$, $i=1,\dots,m$ be associative cones in $\R^7$ having links $\Sigma_i=\sqcup_{j=1}^l \Sigma_i^j$ with $\Sigma_i^j\in \cZ_i^j$. 
\begin{enumerate}[(i)]
	\item If $\sum_{i=1}^m\operatorname{s-ind}(C_i)>0$, then 
for any co-closed $G_2$-structure $\phi \in \sP_{\operatorname{cs}}^{\operatorname{reg}}$ the moduli space $\cM_{\operatorname{cs}}^\phi$ contains no conically singular associative submanifolds having singularities modeled on cones with links in a neighbourhood of  $\Sigma_i$ in $\cZ_i$. 
\item If $\sum_{i=1}^m\operatorname{s-ind}(C_i)>1$, then 
for any path of co-closed $G_2$-structures $\bphi \in \bsP_{\operatorname{cs}}^{\operatorname{reg}}$ the moduli space $\bcM_{\operatorname{cs}}^\bphi$ contains no conically singular associative submanifolds having singularities modeled on cones with links in a neighbourhood of $\Sigma_i$ in $\cZ_i$. 
\end{enumerate}
\end{theorem}
\begin{remark}We conclude from \autoref{thm generic moduli cs asso} that if $\phi \in \sP_{\operatorname{cs}}^{\operatorname{reg}}$ then the moduli space $\cM_{\operatorname{cs}}^\phi$ essentially can contain only conically singular associative submanifolds having singularities modeled on cones $C$ with $\operatorname{s-ind}(C)=0$. Similarly, if $\bphi \in \bsP_{\operatorname{cs}}^{\operatorname{reg}}$ then the moduli space $\bcM_{\operatorname{cs}}^\bphi$ essentially can contain only conically singular associative submanifolds having singularities modeled on cones $C$ with $\operatorname{s-ind}(C)=0$ or $1$. 
\end{remark}

\subsection*{Desingularizations of conically singular associative submanifolds}
Let $P$ be a conically singular (CS) associative submanifold in a co-closed $G_2$-manifold $(Y, \phi)$, with an isolated singularity modeled on a cone $C$ whose link is $\Sigma \subset S^6$ (see \autoref{def CS asso}). To desingularize $P$, we glue rescaled asymptotically conical (AC) associative submanifolds to obtain  approximate associative submanifolds, which we then aim to deform into genuine associative submanifolds, referred to as desingularizations. However, there is an obstruction--the deformation operator associated to this approximate one is not surjective due to the freedom of scaling. Due to the self-adjointness, this is equivalent  to the fact that its kernel—identified with the {matching kernel}, consisting of bounded kernel elements on the CS and AC sides that agree to leading order—never vanishes. To compensate for this obstruction, we perform the deformation within a one-parameter family of co-closed $G_2$-structures under the hypothesis that both the matching kernel and its extension over the family, called the {extended matching kernel}, are one-dimensional and generated by scaling. Under this assumption, we obtain the desired desingularizations in this one parameter family. 

More precisely, let $L$ be an AC associative submanifold in $\R^7$ with the same asymptotic cone $C$ (see \autoref{def AC associative}). The deformation operators $\bD_P$ and $\bD_L$ are defined in \autoref{def D_M}, and the kernels $\ker \bD_{P, \lambda}$ and $\ker \bD_{L, \lambda}$ can be informally described as follows:
\[\ker \bD_{P,\lambda}:=\{u\in C^\infty(NP):\bD_{P}u=0, u=O(r^\lambda)\ \text{as}\ r\to 0\},\]
and
\[\ker \bD_{L,\lambda}:=\{u\in C^\infty(NL):\bD_{L}u=0, u=O(r^\lambda)\ \text{as}\ r\to \infty\}.\]

\begin{definition}\label{def D matching kernel} We define the \textbf{matching kernel} $\mathcal K^{\mathfrak m}$ by
\begin{equation*}\mathcal K^{\mathfrak m}:=\{(u_L,u_{P})\in \ker \bD_{L,0}\oplus \ker \bD_{P,0}:i_\infty u_{L}=i_0u_{P}\},
	\end{equation*}
where the maps $i_\infty$ and $i_0$ are the asymptotic limit maps $i_{L,0}$ and $i_{P,0}$ respectively, defined in \autoref{def asymp limit map}.
\end{definition}

The dilation action of $\R^+$ on $\R^7$ induces a \textbf{canonical Fueter section} $\hat s_L \in C^\infty(NL)$, which satisfies
\begin{equation}\label{def canonical Fueter section}
	\bD_L \hat s_L = 0.
\end{equation}
Moreover, $\hat s_L$ vanishes at infinity, that is, $i_\infty \hat s_L = 0$. As a consequence, the matching kernel $\mathcal K^{\mathfrak m}$ defined in \autoref{def D matching kernel} always contains the element $(\hat s_L, 0)$. This observation will be useful in formulating the hypothesis of the desingularization theorem (\autoref{thm main desing}).

Let $\bphi\in \bsP$ be a path of co-closed $G_2$-structures on $Y$ and $t_0\in (0,1)$ such that $\bphi(t_0)=\phi$. We set $\phi_{t}:=\bphi(t)$. The $G_2$-structure $\phi_t$ induces the $4$-form $\psi_t$ (see \autoref{def almost G2}). 
\begin{definition}We define the \textbf{extended matching kernel} by
\begin{equation*}
\widetilde{\cK}^{\mathfrak m} := \{(u_L,u_P,t)\in \ker \bD_{L,0}\oplus C_{P,0}^{\infty}\oplus \R:\bD_Pu_P+t\hat f_P=0, i_{\infty}u_L=i_{0}u_P\},
\end{equation*}
where $\hat f_P$ is the linearization at $t_0$ of the nonlinear map (relevant for associative deformations) along the path $\bphi\in \bsP$ with $P$ fixed, defined in \autoref{def 1 para moduli CS}. Here, $C_{P,0}^{\infty}$ means bounded, i.e., of order $O(r^0)$, smooth normal vector fields on $P$.
\end{definition}

Note that the inclusion $\langle (\hat s_L,0) \rangle_\R\subset {\cK}^{\mathfrak m}\subset \widetilde{\cK}^{\mathfrak m}$ always holds. Equality in both of these inclusions implies that the relevant extended deformation operator for the family is surjective. This fact, together with other technical estimates, allows us to establish the following desingularization theorem, that is proved in \autoref{subsection main desing}.
\begin{theorem}\label{thm main desing} Let $P$ be a conically singular associative submanifold in a co-closed $G_2$-manifold $(Y,\phi)$ with singularity at a single point modeled on a cone $C$. Let $L$ be an asymptotically conical associative submanifold in $\R^7$ with the same asymptotic cone $C$ and rate $\nu<0$. Let $\bphi\in \bsP$ be a path of co-closed $G_2$-structures on $Y$ and $t_0\in (0,1)$ such that $\phi_{t_0}=\phi$. Assume that the matching kernel $\cK^\fm$ and the extended matching kernel $\widetilde \cK^\fm$ both are one dimensional, that is,
\begin{equation}\label{hyp desing CS}
	\widetilde \cK^\fm=\cK^\fm=\langle (\hat s_L,0) \rangle_\R,
\end{equation}
 where $\hat s_L$ is the canonical Fueter section from \autoref{def canonical Fueter section}. Then there exist $\tilde \epsilon_0>0$, a continuous function $t:[0,\tilde \epsilon_0)\to [0,1]$ with $t(0)=t_0$ and  smooth closed embedded associative submanifolds $\tilde P_{\epsilon, t(\epsilon)}$ in $(Y,\phi_{t(\epsilon)})$ for all $\epsilon\in (0,\tilde \epsilon_0)$ such that $\tilde P_{\epsilon, t(\epsilon)}\to P$ as $\epsilon\to 0$ in the sense of integral currents. 
\end{theorem}

As applications of the \autoref{thm main desing}, we proceed to desingularize CS associative submanifold with a Harvey--Lawson $T^2$-cone singularity and associative submanifold with a transverse unique self intersection. 

  \subsubsection*{Desingularizations for CS associatives with a Harvey--Lawson $T^2$-cone singularity:}
  Let $\bphi \in \bsP_{\operatorname{cs}}^{\operatorname{reg}}$ be a generic path of co-closed $G_2$-structures on $Y$ (see \autoref{def regular G2 str}). Consider a conically singular associative submanifold $P$ in  $(Y, \phi_{t_0})$ for some $t_0 \in (0,1)$, with a singularity modeled on the Harvey--Lawson $T^2$-cone at a single point (see \autoref{eg Harvey-Lawson $T^2$-cone}). Since this cone is a rigid special Lagrangian in $\C^3$ with a genus-one link, the genericity of the path implies that $P$ is an isolated point in the moduli space $\bcM_{\operatorname{cs}}^\bphi$ (see \autoref{lem generic matching kernel d=2}). Moreover, there exists normal vector field $\hat{v}_P$ of order $O(r^{-1})$, which spans $\ker \bD_{P,-1}$ and gives rise to the following non-zero quantity that will be useful in the following \autoref{thm desing HL sing}:
\begin{equation}\label{eq generic HL vP}
	a := \langle \hat{v}_P, \hat{f}_P \rangle_{L^2} \neq 0,
\end{equation}
where $\hat{f}_P$ is defined in \autoref{def 1 para moduli CS}. The quantity $a$ essentially represents the first order obstruction to deform $P$ as a CS associative along the path $\boldsymbol{\phi}$.
This also implies that the kernels contributing to both the matching kernel and the extended matching kernel originate entirely from the AC side and consist of elements decaying at the rate $O(r^{-1})$ at infinity. Notably, the AC special Lagrangians employed in the desingularization process always possess a one-dimensional kernel of this decay type, which corresponds precisely to the canonical Fueter section arising from scaling. In summary, the genericity assumption on the path $\bphi$ is sufficient to ensure that the hypothesis \autoref{hyp desing CS} required in the desingularization \autoref{thm main desing} is satisfied.

There are three AC special Lagrangians in $\C^3$, $L^i :=L^i_1$, for $i = 1,2,3$ from \autoref{eg Harvey-Lawson special Lagrangians} asymptotic to the Harvey--Lawson $T^2$-cone, which will be used to construct three one-parameter families of desingularizations. However, to realize the first transition discussed at the beginning of this article, we need to impose additional transversality conditions. These restrict the choice of the path of $G_2$-structures $\bphi$ to a subset $\bsP^\bullet$ of $\bsP_{\operatorname{cs}}^{\operatorname{reg}}$, which we expect remains comeager—so the choice of $\bphi$ is still generic, though a proof is not included in this article. More precisely, we assume that $\hat{v}_P$ appearing in \autoref{eq generic HL vP} is transverse, in leading order, to each $L^i$. The following definition formalizes this assumption.

\begin{definition}$\bsP^\bullet\subset \bsP_{\operatorname{cs}}^{\operatorname{reg}}$ consists of all $\bphi\in \bsP_{\operatorname{cs}}^{\operatorname{reg}}$ such 
	that for all CS associative submanifolds $P$ in $\bcM_{\operatorname{cs}}^\bphi$ with Harvey--Lawson $T^2$-cone singularity at a single point, the asymptotic limit of $\hat v_P$ appeared in  \autoref{eq generic HL vP} has the form  \begin{equation}\label{eq generic HL asymp limit}
	i_{P,-1}\hat v_P=b_1\xi_1+b_2\xi_2,\ b_1\neq 0,b_2\neq 0, b_1\neq b_2.   
\end{equation}
where $\xi_1$ and $\xi_2$ denote the leading order $O(r^{-1})$ terms in the expansions of the AC special Lagrangians $L_1^1$ and $L_1^2$, respectively, as given in \autoref{eq AC HL} of \autoref{eg Harvey-Lawson special Lagrangians}. The quantities $b_1$, $b_2$ and $b_1-b_2$ essentially represent the first order obstructions to desingularizations of $P$ by gluing $L_1^1$, $L_1^2$ and $L_1^3$ within the fixed $G_2$-structure $\phi_{t_0}$.
\end{definition}
\begin{remark}
The ratios $\frac {b_1}{a}, \frac {b_2}{a}$ do not depend on the particular choice of $\hat v_P$.  
\end{remark}
We now state the desingularization theorem for conically singular associatives with Harvey--Lawson $T^2$-cone singularity at a single point that is proved in \autoref{subsection desing HL}.
 \begin{theorem}\label{thm desing HL sing}Let $\bphi\in \bsP_{\operatorname{cs}}^{\operatorname{reg}} $ be a path of co-closed  $G_2$-structures on $Y$ and $t_0\in (0,1)$. Let $P$ be a conically singular associative submanifold of $(Y,\phi_{t_0})$ in $\bcM_{\operatorname{cs}}^\bphi$ with Harvey--Lawson $T^2$-cone singularity at a single point $x$. There exist $\tilde \epsilon_0>0$, three continuous functions $t^i:[0,\tilde \epsilon_0)\to [0,1]$ with $t^i(0)=t_0$, $i=1,2,3$ and  smooth closed embedded associative submanifolds $\tilde P_{\epsilon, t^i(\epsilon)}$ in $(Y,\phi_{t^i(\epsilon)})$ for all $\epsilon\in (0,\tilde \epsilon_0)$ such that $\tilde P_{\epsilon, t^i(\epsilon)}\to P$ as $\epsilon\to 0$ in the sense of integral currents. The $\tilde P_{\epsilon, t^i(\epsilon)}$ are diffeomorphic to the Dehn filling of $P^o:=P\setminus B_\epsilon(x)$ along  simple closed curves $\mu_i\subset \partial P^o\cong T^2$ that satisfy $\mu_1.\mu_2=\mu_2.\mu_3=\mu_3.\mu_1=-1$. Furthermore, if $\bphi\in \bsP^\bullet$ then there is a constant $c\neq 0$ such that
 \begin{equation}\label{eq leading HL desing}
 \begin{split}
 	t^1(\epsilon)=t_0-\frac {cb_2}{a}\epsilon^2+o(\epsilon^2),\ \ t^2(\epsilon)=t_0+\frac {cb_1}{a}\epsilon^2+o(\epsilon^2),\\  \text{and} \ \ t^3(\epsilon)=t_0+\frac {c(b_2-b_1)}{a}\epsilon^2+o(\epsilon^2).
 	\end{split}
 \end{equation}

 \end{theorem} 
 \begin{remark}
 We would like to acknowledge that our derivation of the leading-order expression in \autoref{eq leading HL desing} is influenced by \cite[Remark 5.4 a)]{Joyce2016}.
 \end{remark}
 \subsubsection*{Desingularizations for associative submanifolds with transverse intersection:}
 Let $\bphi \in \bsP_{\operatorname{cs}}^{\operatorname{reg}}$ be a generic path of co-closed $G_2$-structures on $Y$. Consider an associative submanifold $P$ in $(Y, \phi_{t_0})$  for some $t_0 \in (0,1)$, exhibiting a unique transverse self-intersection. More precisely, $P$ is a conically singular associative submanifold with a singularity at a single point modeled on the union of two transverse associative planes $\Pi_\pm \subset \R^7$ and there exists $B \in G_2$ such that $B \Pi_0 = \Pi_+$, $B \Pi_\theta = \Pi_-$, and 
\begin{equation}\label{eq B Lawlor}
\R^7 = \langle \mathbf{n} \rangle_{\R} \oplus B \Pi_0 \oplus B \Pi_\theta,
\end{equation}
which is an orientation-compatible splitting as described in \autoref{eg s-ind Pair of transverse associative planes}. To resolve this intersection, we aim to glue in a Lawlor neck--an asymptotically conical (AC) special Lagrangian submanifold asymptotic to the planes $\Pi_\pm$ (see \autoref{eg Lawlor neck}).

Unlike the case of desingularizations for CS associative submanifolds with a Harvey--Lawson $T^2$-cone singularity, the genericity assumption $\bphi \in \bsP_{\operatorname{cs}}^{\operatorname{reg}}$ does not guarantee the absence of non-trivial elements in the extended matching kernel on the CS side. That is, a family of intersecting associative submanifolds may persist along the path $\bphi$, potentially increasing the dimension of the extended matching kernel and thereby violating the hypothesis \autoref{hyp desing CS} required for the desingularization result \autoref{thm main desing}.
To address this issue, we restrict our choice of the path of $G_2$-structures $\bphi$ to a subset $\bsP^\dagger$ of $\bsP_{\operatorname{cs}}^{\operatorname{reg}}$ consisting of those paths for which associative submanifolds with a transverse, unique self-intersection are unobstructed as immersed associatives. This ensures that $P$ can be deformed into a family of embedded, closed associatives that separate the two sheets of the self-intersection along the path. As a consequence, any non-trivial elements in the extended matching kernel on the CS side must have different components in the direction perpendicular to the union of the two tangent planes $\Pi_\pm$ at the intersection point. However, any elements in the extended matching kernel arising from the AC special Lagrangian (Lawlor neck) side always have vanishing perpendicular components. Therefore, for this restricted subset of paths, the hypothesis \autoref{hyp desing CS} is satisfied.
We note that this subset is still expected to be comeager, so the choice of $\bphi$ remains generic, although a formal proof of this fact is not included in the present article. The following definition formalizes this restriction.

\begin{definition}\label{def dag}
 	$\bsP^\dagger\subset \bsP_{\operatorname{cs}}^{\operatorname{reg}}$ consists of all $\bphi\in \bsP_{\operatorname{cs}}^{\operatorname{reg}}$ such that along it, all associative submanifolds $P$ with transverse unique self intersection are (rigid) unobstructed as immersed associative submanifolds. In other words,
 $\ker \bD_{P,0}=\{0\}$ and there exists  $\hat v_P\in C_{P,0}^\infty$ such that 
 \begin{equation}\label{eq generic intersection}
 	\bD_{P} \hat v_P=\hat f_P,\ \ \text{and}\ \ a:= (\hat v_P^+(0)-\hat v_P^-(0))\cdot \bn \neq 0,
 \end{equation}
 where $\hat v_P^\pm(0)\in  \Pi_\pm^\perp$ are the asymptotic limits of $\hat v_P^\pm$ (see \autoref{lem main Fredholm} and \autoref{eg s-ind Pair of transverse associative planes}) and $\hat{f}_P$ is from \autoref{def 1 para moduli CS}.    
 
 The above implies that if $\bphi\in \bsP^{\dag}$ and $P$ is an associative submanifold in $(Y,\phi_{t_0})$ for some $t_0\in (0,1)$ with a transverse unique self intersection then there exists a family of immersed closed associative submanifolds \begin{equation}\label{eq P_t}
 	\{P_t:\abs{t-t_0}\ll 1,\ P_{t_0}=P,\ P_t \ \text{is an embedded closed associative in}\ (Y,\phi_t), \forall t\neq t_0 \}.
 \end{equation}
 The quantity $a$ essentially measures the speed at which the two local sheets of $P_t $ move across each other near the crossing point as the parameter $t$ passes through $t_0$. 
 \end{definition} 

We obtain the following theorem whose proof is given in \autoref{subsection desing intersecting}.

 \begin{theorem}\label{thm desing intersection}Let $\bphi\in \bsP^{\dag}$  be a path of co-closed  $G_2$-structures on $Y$ and $t_0\in (0,1)$. Let $P$ be an associative submanifold of $(Y,\phi_{t_0})$ with a transverse unique self intersection at $x$. There exist $\tilde \epsilon_0>0$, a continuous function $t:[0,\tilde \epsilon_0)\to [0,1]$ with $t(0)=t_0$ and smooth closed embedded associative submanifolds $\tilde P_{\epsilon, t(\epsilon)}$ in $(Y,\phi_{t(\epsilon)})$ for all $\epsilon\in (0,\tilde \epsilon_0)$ such that $\tilde P_{\epsilon, t(\epsilon)}\to P$ as $\epsilon\to 0$ in the sense of integral currents. The $\tilde P_{\epsilon, t(\epsilon)}$ are diffeomorphic to the connected sums $P_t\#(S^1\times S^2)$ if $P_t$ is connected, and otherwise to $P_t^+\#P_t^-$, where $P_t=P_t^+\amalg P_t^-$. Here $P_t$ is from \autoref{eq P_t} with $t\neq t_0$. 
\end{theorem}

\begin{remark}\label{remark ddag}
Although the condition $\bphi \in \bsP^{\dag}$ suffices to establish the desingularization result \autoref{thm desing intersection} following \autoref{thm main desing}, realizing the second transition discussed at the beginning of this article requires additional transversality conditions. These impose further restrictions on the path of $G_2$-structures $\bphi$, confining it to a subset of $\bsP^{\dag}$, which we denote by $\bsP^{\ddag}$. We expect this subset to remain comeager—so the choice of $\bphi$ remains generic—though we do not provide a proof here.
More precisely, define $\bsP^{\ddag}$ to consist of those $\bphi \in \bsP^{\dag}$ such that for every associative submanifold $P$ (along the path) with a transverse, unique self-intersection, there exists $\hat u_P \in \ker \bD_{P,-2}$ satisfying
\begin{equation}\label{eq generic intersection asymp limit}
    i_{P,-2} \hat u_P = (B \xi^+, B \xi^-), \quad \text{and} \quad b := (\hat u_P - B \xi^+)(0) \cdot \bn - (\hat u_P - B \xi^-)(0) \cdot \bn \neq 0,
\end{equation}
where $\xi^\pm$ denote the leading-order $O(r^{-2})$ terms over $\Pi_\pm$ for the Lawlor neck defined in \autoref{eq xi Lawlor neck} of \autoref{eg Lawlor neck}, and $B \in G_2$ is as above. The quantity $b$ essentially represents the first order obstruction to deforming $P$ into an embedded, closed associative submanifold within the fixed $G_2$-structure $\phi_{t_0}$. Note also that choosing $-\bn$ instead of $\bn$ swaps $\Pi_+$ and $\Pi_-$, ensuring that the definitions of $a$ and $b$ remain well-defined.
As we will discuss in \autoref{rmk discussion on leading order term}, following \cite[Remark 4.5(a)]{Joyce2016}, if $\bphi \in \bsP^{\ddag}$, one would expect the existence of a constant $c \neq 0$ such that
\begin{equation}\label{eq t(epsilon) intersection}
    t(\epsilon) = t_0 - \frac{cb}{a} \epsilon^3 + o(\epsilon^3).
\end{equation}
Establishing this expansion would confirm the second transition described at the beginning of the article. However, despite our efforts, we have not succeeded in proving \autoref{eq t(epsilon) intersection}, as explained in more detail in \autoref{rmk discussion on leading order term}. We would like to highlight that although \citet{Nordstrom2013} previously established \autoref{thm desing intersection}, he also did not derive \autoref{eq t(epsilon) intersection}. While we refine Nordström's argument by constructing an improved approximate desingularization that yields a smaller pre-gluing error, our analysis still lacks the precise estimates required to rigorously justify the expansion in \autoref{eq t(epsilon) intersection}. We hope that future refinements of the techniques developed in this article will lead to a complete proof. For this reason, as mentioned at the outset, our results concerning the second transition should be regarded as partial.
\end{remark}

\paragraph{Acknowledgements.} I am grateful to my PhD supervisor Thomas Walpuski for suggesting the problems solved in this article. I am also indebted to him for his constant encouragement and advice during this work. Additionally, I extend my thanks to him, Johannes Nordstr\"om, Jason Lotay and the anonymous referee for their meticulous review of the manuscript and valuable feedback. I also thank Dominik Gutwein and Viktor Majewski for their careful reading and helpful comments on the earlier drafts. This material is based upon work supported by the \href{https://sites.duke.edu/scshgap/}{Simons Collaboration on Special Holonomy in Geometry, Analysis, and Physics}.

\section{Preliminaries: $G_2$-manifolds and associative submanifolds}\label{section 1}
In this section we review definitions and basic facts about $G_2$-manifolds, more generally almost $G_2$-manifolds and associative submanifolds, which are important to understand this article. To delve further into these topics, we refer the reader to \cites{Salamon2010}{Joyce2007}{Karigiannis2020}{Harvey1990} and other relevant sources mentioned throughout the discourse. 
 
 \subsection{$G_2$-manifolds}
This subsection reviews definitions and basic facts about $G_2$-manifolds.
\begin{definition}
The group $G_2$ is the automorphism group of the normed division algebra of octonions $\O$, that is,
\[G_2:=\Aut(\O)\subset \SO(7).\qedhere\]
\end{definition}
This is a simple, compact, connected, simply connected Lie group of dimension $14$. Furthermore, there exists a fibration $\SU(3)\hookrightarrow G_2\to S^6$.
Writing $\O=\Re \O\oplus\Im \O\cong \R\oplus \R^7$, we define the \textbf{cross-product} $\times:\Lambda^2\R^7\to \R^7$ by
$$(u,v)\mapsto u\times v:=\Im(uv).$$
The \textbf{$3$-form} $\phi_e\in \Lambda^3(\R^7)^*$ is defined by
$$\phi_e(u,v,w):=g_e(u\times v,w),$$
where $g_e:S^2(\R^7)\to \R$ is $g_e(u,v)=-\Re(uv)$, the standard Euclidean metric on $\R^7$.
These are related by the important identity \begin{equation}\label{eq G2 identity}
 \iota_u\phi_e\wedge\iota_v \phi_e\wedge \phi_e=6g_e(u,v)\vol_{g_e}.	
 \end{equation}
 More explicitly, there $\{e^1,\dots,e^7\}$ is an oriented orthonormal frame on $\R^7$ such that $$\phi_e=e^{123}-e^{145}-e^{167}-e^{246}-e^{275}-e^{347}-e^{356},$$
  where $e^{ijk}:=e^i\wedge e^j \wedge e^k$. 
  The Lie group $G_2$ can also be expressed as
 $$G_2:=\{A\in \GL(\R^7):A^*\phi_e=\phi_e\}.$$
  There is also a \textbf{$4$-form} $\psi_e:=*_{g_e}\phi_e\in \Lambda^4(\R^7)^*$, which also can be defined by 
$$\psi_e(u,v,w,z):=g_e([u,v,w],z),$$
where $[\cdot,\cdot,\cdot]:\Lambda^3(\R^7)\to \R^7$ is the \textbf{associator}, defined as follows:
$$[u,v,w]:=(u\times v)\times w+\inp{v}{w}u-\inp{u}{w}v.$$

\begin{definition}\label{def almost G2}A \textbf{$G_2$-structure} on a $7$-dimensional manifold $Y$ is a principal $G_2$-bundle over $Y$ which is a reduction of the frame bundle $\GL(Y)$. 

An \textbf{almost $G_2$-manifold} is a $7$-dimensional manifold $Y$ equipped with a $G_2$-structure or equivalently, equipped with a \textbf{definite} $3$-form $\phi\in \Omega^3(Y)$, that is, the bilinear form $G_\phi: S^2TY\to \Lambda^7(T^*Y)$ defined by 
$$G_\phi(u,v):=\iota_u\phi\wedge\iota_v \phi\wedge \phi$$
is definite.
\end{definition}
 A $G_2$-structure $\phi$ on $Y$ defines uniquely a Riemannian metric $g_\phi$ and a volume form $\vol_{g_\phi}$ on $Y$ satisfying the identity \autoref{eq G2 identity}. Moreover it defines 
\begin{itemize}
\item a \textbf{cross product} $\times:\Lambda^2(TY)\to TY$,	
\item an \textbf{associator} $[\cdot,\cdot,\cdot]:\Lambda^3(TY)\to TY$,
\item a \textbf{$4$-form} $\psi:=*_{g_\phi} \phi\in \Omega^4(Y)$.\qedhere
\end{itemize}

\begin{remark}A $7$-dimensional manifold is an almost $G_2$-manifold if and only if it is spin, see \cite[Theorem 3.1-3.2]{Gray1969}. A $G_2$-structure is also equivalent to a choice of a \textbf{definite} $4$-form $\psi\in \Omega^4(Y)$ and an orientation on $Y$; see \cite[Section 8.4]{Hitchin2001}. Here $\psi$ being {definite} means the bilinear form $G_\psi:S^2T^*Y\to \Lambda^7(T^*Y)\tn \Lambda^7(T^*Y)$ defined by
$$G_\psi(u,v):=\iota_u\psi\wedge\iota_v \psi\wedge \psi$$
is definite. In this definition, the $4$-form $\psi$ is considered as a section of $\Lambda^3(TY)\tn  \Lambda^7(T^*Y) \cong \Lambda^4(T^*Y)$ and the contraction $\iota_u\psi$ is a section of $\Lambda^2(TY)\tn  \Lambda^7(T^*Y)$.
\end{remark}
\begin{definition}\label{def G2 manifold}
 A \textbf{$G_2$-manifold} is a $7$-dimensional manifold $Y$ equipped with a torsion-free $G_2$-structure, that is, equipped with a definite $3$-form $\phi\in \Omega^3(Y)$  such that  $\nabla_{g_\phi}\phi=0$,
 or equivalently,
 \begin{equation*}d\phi=0\ \ \text{and}\ \ d\psi=0.\qedhere \end{equation*} 
\end{definition}
The equivalence in the above definition was established by \citet[Theorem 5.2]{Fernandez1982}. 

\begin{example}$(\R^7,\phi_e)$ is a $G_2$-manifold.
	\end{example}
	\begin{example}\label{eg G2 hyperkahller}Let $(X,\omega_I,\omega_J,\omega_K)$ be a hyperk\"ahler $4$-manifold. The manifolds $\R^3\times X$ and $T^3\times X$ are $G_2$-manifolds with $G_2$-structure
	$$\phi:=dt^1\wedge dt^2\wedge dt^3- dt^1\wedge \omega_I-dt^2\wedge \omega_J-dt^3\wedge \omega_K,$$
	where $(t^1,t^2,t^3)$ are the coordinates of $\R^3$.	
	\end{example}

\begin{example}\label{eg G2 manifold}Let $(Z,\omega,\Omega)$ be a Calabi-Yau $3$-fold, where $\omega$ is a K\"ahler form and $\Omega$ is a holomorphic volume form on $Z$ satisfying
$$\frac{\omega^3}{3!}=-\Big(\frac i2\Big)^3\Omega\wedge \bar \Omega.$$ 
The product with the unit circle, $Y:=S^1\times Z$ is a $G_2$-manifold with the $G_2$-structure
$$\phi:=dt\wedge \omega+\Re \Omega, \ \psi=\frac 12\omega\wedge \omega+dt\wedge \Im \Omega,$$
where $t$ denotes the coordinate on $S^1$. In this case, the holonomy group $\Hol(Y,g_\phi)\subset \SU(3)$.
\end{example}
\begin{remark}
Any $G_2$-manifold $(Y,\phi)$ admits a nowhere vanishing parallel spinor and therefore the metric $g_\phi$ is Ricci-flat \cites[pg. 321]{Lawson1989}[Theorem 1.2]{Hitchin1974}. A compact $G_2$-manifold $Y$ has holonomy exactly equal to $G_2$ if and only if $\pi_1(Y)$ is finite \cite[Proposition 11.2.1]{Joyce2007}.
\end{remark}

\begin{example}
\citet{Bryant1987}, \citet{Bryant1989} first constructed local and complete manifolds with holonomy equal to $G_2$, respectively. Joyce \cite{Joyce1996a} first constructed compact manifolds with holonomy equal to $G_2$ by smoothing flat $T^7/\Gamma$, where $\Gamma$ is a finite group of isometries of $T^7$. This has been generalized later by \citet{Joyce2017}. \citet{Kov03} introduced the twisted connected sum  construction of $G_2$-manifolds which glues a suitable {matching} pair of {asymptotically cylindrical} $G_2$-manifolds. This construction was later improved by \citet{KL11,CHNP15} to produce hundreds of thousands of examples of compact manifolds with holonomy equal to $G_2$.
\end{example} 
The moduli space of torsion free $G_2$-structures over $Y$ is a smooth manifold of dimension $b^3(Y)$ \cite[Part I, Theorem C]{Joyce1996b}, therefore it is not enough to achieve transversality for various enumerative theories. To address this, one must consider an infinite dimensional space of $G_2$-structures \cites[Section 3.2]{Donaldson2009}[Section 2.5]{Joyce2016}, namely the following.
\begin{definition}\label{def coclosed and tamed G2}A $G_2$-structure $\phi$ is called a \textbf{co-closed $G_2$-structure} if $d\psi=0$, where $\psi:=*_{g_\phi} \phi$. 

An {almost $G_2$-manifold} $(Y,\phi)$ is called a \textbf{co-closed $G_2$-manifold} if $\phi$ is a co-closed $G_2$-structure.

A $G_2$-structure $\phi$ is called \textbf{tamed} by a closed $3$-form $\tau\in \Omega^3(Y)$ if 
for all $x\in Y$ and $u,v,w\in T_xY$ with $[u,v,w]=0$ and $\phi(u,v,w)>0$, we have $\tau(u,v,w)>0$.
\end{definition}
\begin{remark}The definition of the above tamed $G_2$-structures can be ignored for this article; they are used only in \autoref{thm cptness current} to bound the volume of associative sub-manifolds for compactness. This restricted class of $G_2$-structures has been employed to construct well-defined enumerative theories in \cite{Joyce2016} and \cite{Doan2017d}. However, for the purposes of this article, such a restriction is not required.
\end{remark}

\begin{example}
Nearly parallel $G_2$-manifolds, that is, co-closed $G_2$-manifolds $(Y,\phi)$ satisfying $d\phi=\lambda \psi$ for some constant $\lambda\in \R$, are examples of co-closed $G_2$-manifolds.	
\end{example}

\subsection{Associative submanifolds}	
In any almost $G_2$-manifold, we can consider a special class of $3$-dimensional calibrated submanifolds, called associative submanifolds, which are the main objects of study in this article. These were first invented by \citet{Harvey1982}.
\begin{definition}\label{def closed associative}Let $(Y,\phi)$ be an almost $G_2$-manifold. A $3$-dimensional oriented submanifold $P$ of $Y$ is called an \textbf{associative submanifold} if it is semi-calibrated by the $3$-form $\phi$, that is, $\phi_{|_P}$ is the volume form $\vol_{P,g_\phi}$ on $P$, or equivalently, the associator $[u,v,w]=0$, for all $x\in P$ and $u,v,w\in T_xP$. 
\end{definition} 

\begin{remark}If $\phi$ is a calibration (i.e. $d\phi=0$) and $P$ is compact then it is a minimal submanifold and volume minimizing in its homology class \cite[Theorem 4.2]{Harvey1982}. The equivalence in the above definition follows from the identity: $\abs{u\wedge v\wedge w}^2=\phi(u,v,w)^2+\abs{[u,v,w]}^2$ (see \cite[section IV, Theorem 1.6]{Harvey1982}).
\end{remark}
\begin{example} In \autoref{eg G2 hyperkahller}, $\R^3$ and $T^3$ are associative submanifolds.	\end{example}

\begin{example}Let $Z$ be a Calabi-Yau $3$-fold and $S^1\times Z$ be the $G_2$-manifold as in \autoref{eg G2 manifold}. For any holomorphic curve $\Sigma \subset Z$ and special Lagrangian $L\subset Z$ (i.e. $L$ is calibrated by $\Re\Omega$), the $3$-dimensional submanifolds $S^1\times \Sigma$ and $\{e^{i\theta}\}\times L$ with $e^{i\theta}\in S^1$, are associative submanifolds of $S^1\times Z$.
\end{example}	

\begin{example}
\citet[Section 4.2]{Joyce1996b} has produced examples of closed associative submanifolds which are the fixed point loci of $G_2$-involutions in his generalized Kummer constructions. Recently, new examples of associative submanifolds in these $G_2$-manifolds have been constructed by \citet{Dwivedi2022}. 
\end{example}
\begin{example}
Examples of closed associative submanifolds in the twisted connected sum (TCS) $G_2$-manifolds were first constructed by \citet[Section 5, Section 7.2.2]{CHNP15} from closed holomorphic curves and  closed special Lagrangians in asymptotically cylindrical (ACyl) Calabi-Yau submanifolds. The author in \cite{Bera2022} has constructed more examples of closed associative submanifolds in the twisted connected sum $G_2$-manifolds. These are obtained from ACyl holomorphic curves and ACyl special Lagrangians in ACyl Calabi-Yau submanifolds using a gluing construction. It is also expected that this gluing construction can be used to produce infinitely many closed associative submanifolds in a certain TCS $G_2$-manifold studied by Braun et al. \cite{Braun2018}.
\end{example}
\begin{remark}
 Examples of associative submanifolds of nearly parallel $G_2$-manifolds have been constructed by \citet{Lotay2012a} in $S^7$, \citet{Kawai2015} in the squashed $S^7$ and \citet{Ball2000} in the Berger space. 
\end{remark}
\subsection{Normal bundles and canonical isomorphisms}
This subsection sets up our conventions for the normal bundle of a submanifold, the tubular neighbourhood map, and various canonical isomorphisms, which will be used extensively throughout the article to describe the deformation theories of associative submanifolds and holomorphic curves. 
\begin{definition}[Normal bundle]\label{def prelim normal bundle}
	Let  $Y$ be a manifold and $M$ be a submanifold of it. The \textbf{normal bundle} $\pi:NM\to M$ is characterised by the exact sequence
	\begin{equation}\label{eq normal bundle exact seq}
	0\to TM\to TY_{|M}\to NM\to 0.\qedhere
	\end{equation}
\end{definition}	
\begin{definition}[Tubular neighbourhood map]\label{def prelim tubular neighbourhood map}
 A \textbf{tubular neighbourhood map} of $M$ is a diffeomorphism between an open neighbouhood $V_M$ of the zero section of the normal bundle $NM$ of $M$ that is convex in each fiber and an open neighbourhood $U_M$ (tubular neighbourhood) of $M$ in $Y$,
$$\Upsilon_M:V_M\to U_M$$ that takes the zero section $0$ to $M$ and  
the composition $$NM\to 0^*TNM\xrightarrow{d\Upsilon_M} TY_{|M}\to  NM$$ is the identity.
\end{definition}

\begin{definition}[Canonical extension of normal vector fields]
\label{def canonical extension normal vf}
	The tangent bundle $TNM$ fits into the exact sequence 
$$0\to \pi^*NM\xrightarrow{i} TNM \xrightarrow{d\pi} \pi^*TM\to 0.$$ 
This induces an \textbf{canonical extension map}, which extends normal vector fields on $M$ to vector fields on $NM$:
$$\widetilde{\bullet}:C^\infty(NM)\hookrightarrow \Vect(NM),\ \ \  u\mapsto \tilde u:=i(\pi^*u).$$
Here $\pi^*u\in C^\infty(\pi^*NM)$ is the pull back section. 
\end{definition}
\begin{notation}\label{notation u}
There are instances in this article where it is more appropriate to use the notation $\tilde u$ but for simplicity, we will abuse notation and denote it by $u$.
\end{notation}

\begin{remark}\label{rmk commutator}
Observe that, $[\tilde u,\tilde v]=0$ for all $u,v\in C^\infty(NM)$. This fact will be useful later.
\end{remark}
\begin{definition}[Canonical isomorphisms]\label{def prelim canonical prelim isomorphisms}
For any section $u\in C^\infty(NM)$ we define the \textbf{graph of $u$} by
\[\Gamma_u:=\{\big(x,u(x)\big)\in NM:x\in M\}.	
\]
This is a submanifold of $NM$ and the bundle $\pi^*NM_{|\Gamma_u}$ fits into the split exact sequence
\begin{equation}\label{eq normal bundle graph iso}
	\begin{tikzcd}
0\arrow{r} & T\Gamma_u \arrow{r} & TNM_{|\Gamma_u}\arrow[l, bend right, swap, "d(u\circ\pi)"]\arrow{r}{\id-d(u\circ\pi)} &\pi^*NM_{|\Gamma_u} \arrow{r} &0.
\end{tikzcd}
\end{equation}
In particular, this induces a canonical isomorphism $N\Gamma_u\cong \pi^*NM_{|\Gamma_u}$. Moreover, the composition $T\Gamma_u\to TNM_{|\Gamma_u}\xrightarrow{d\pi} \pi^*TM_{|\Gamma_u}$ is an isomorphism. 
Let $\Upsilon_M:V_M\to U_M$ be a tubular neighbourhood map of $M$. We define
$C^\infty(V_{M}):=\{u\in C^\infty(NM):\Gamma_u\subset V_M\}.$
Let $u\in C^\infty(V_{M})$. Denote by $M_u$ the submanifold $\Upsilon_M(\Gamma_u)$ of $Y$. Then there is a \textbf{canonical bundle isomorphism}:
\begin{equation}\label{eq prelim canonical iso map}
	\begin{tikzcd}
 NM \arrow{r}{\Theta^M_u}\arrow{d} & NM_{u}\arrow{d} \\
M \arrow{r}{\Upsilon_M\circ u} & M_{u}.
\end{tikzcd}
\end{equation}
 induced by the following commutative diagram of bundle isomorphisms:

\begin{equation*}
	\begin{tikzcd}[column sep=50pt]
 NM \arrow{r}{\Theta^M_u}\arrow{d}{u^*} & NM_{u} \\
 \pi^*NM_{|\Gamma_u}\arrow[u, bend left, "\pi^*"] & N\Gamma_u\arrow{l}{\id-d(u\circ\pi)}\arrow{u}{d\Upsilon_M}.
\end{tikzcd}\qedhere
\end{equation*}
\end{definition}

\begin{definition}[Normal connection] \label{def prelim normal connection}
The choice of a Riemannian metric $g$ on $Y$ induces a splitting of the exact sequence \autoref{eq normal bundle exact seq}, that is,
 \[TY_{|M}=TM\perp NM.\]
Denote by $\cdot^\parallel$ and $\cdot^\perp$ the projections onto the first and second summands respectively. The Levi-Civita connection $\nabla$ on $TY_{|M}$ decomposes as 
\[\nabla=
\begin{bmatrix}
\nabla^\parallel & -\operatorname{II}^*\\
\operatorname{II} &\nabla^\perp
\end{bmatrix}.
\]
Here $\operatorname{II}\in \Hom(S^2TM,NM) $ is the \textbf{second fundamental form} of M, $\nabla^\parallel$ is the Levi-Civita connection on $M$ and $\nabla^\perp$ is the \textbf{normal connection} on $NM$.
\end{definition}

\subsection{Moduli space of closed associative submanifolds}
This subsection reviews established results on the moduli space of closed associative submanifolds and motivates the study of conically singular associative submanifolds, which is the focus of this article.
\begin{definition}\label{def moduli closed asso} Let $(Y,\phi)$ be an almost $G_2$-manifold.
Let $\mathcal S_3$ be the set of all $3$-dimensional oriented, closed smooth submanifolds of $Y$. 
We define the \textbf{$C^k$-topology} on the set $\mathcal S_3$ by specifying a basis, which is a collection of all the sets of the form 
$\{\Upsilon_{P}(\Gamma_u):\ u\in\mathcal V^k_{P}\}$,
 where $P\in \mathcal S_3$, $\Upsilon_{P}$ is a tubular neighbourhood map of $P$, and $\mathcal V^k_{P}$ is an open set in $C^\infty(V_{P})$, whose topology is induced by the $C^k$-norm on $C^\infty(NP)$.
The \textbf{$C^\infty$-topology} on the set $\mathcal S_3$ is the inverse limit topology of $C^k$-topologies on it, that is, a set is open with $C^\infty$-topology if it is open for every $C^k$-topology. 
 
 The \textbf{moduli space}  $\cM^\phi$ of closed associative submanifolds in $(Y,\phi)$ is the subset of all submanifolds in $\mathcal S_3$ which are associatives. The $C^\infty$-topology on $\mathcal M^\phi$ is the subspace topology of ~$\mathcal S_3$. 
 
 Let $\sP$ be the set of all co-closed $G_2$-structures on $Y$. Denote the \textbf{universal moduli space} of closed associative submanifolds by $$\cM:=\{(\phi,P)\in \sP\times \cS_3:P\in \cM^\phi \}.$$  Equip $\sP$ with the $C^\infty$ topology. The topology of $\cM\subset \sP\times \cS_3$ is then given by the subspace topology of the product $C^\infty$-topologies. 
 
 Let $\bsP$ be the space of paths $\bphi:[0,1]\to \sP$ which are smooth as sections over $[0,1]\times Y$. Set $\phi_t:=\bphi(t)$. Define the \textbf{$1$-parameter moduli space} of closed associative submanifolds by the fiber product 
\begin{equation*}
	\bcM^\bphi:=[0,1]\times_\sP \cM\cong\{(t,P)\in [0,1]\times \cS_3 :P\in \cM^{\phi_t} \}.
\end{equation*}
The topology on $\bcM^\bphi$ is the fiber product topology, which is same as the subspace topology of the product topology of $[0,1]\times \cS_3$. 
\end{definition}
\begin{definition}\label{def D_M}
	Let $M$ be an associative submanifold (compact or noncompact) of an almost $G_2$-manifold $(Y,\phi)$. The operator $\mathbf D_{M}:C^\infty(NM)\to C^\infty(NM)$ is defined by
 $$\inp{\mathbf D_{M}v}{w}_{L^2}:=\int_M\biggl\langle{\sum_{i=1}^3e_i\times \nabla^{\perp}_{M,e_i}v}, {w}\biggr\rangle+\int_M\iota_w\nabla_v\psi,$$ 
 for all $v\in C^\infty(NM)$ and $w \in C^\infty_c(NM)$. Here $NM$ is the normal bundle of $M$ and $\nabla^{\perp}_M$ is the normal connection and  $\{e_1,e_2,e_3\}$ is an oriented local orthonormal frame for $TM$ with respect to the metric $g_\phi$.
 
 If $(Y,\phi)$ is a $G_2$-manifold then $\nabla\psi=0$ and $\mathbf D_{M}$ is a Dirac operator, called the \textbf{Fueter operator}.  A closed associative submanifold $M$ of a co-closed $G_2$-manifold is called \textbf{rigid}, or equivalently, \textbf{unobstructed} if $\ker \bD_M=\{0\}$.
\end{definition}
\begin{remark} \label{rmk D_M}The operator $\mathbf D_{M}$ is formally self adjoint if $d\psi=0$ (see \autoref{prop formal self adjoint}). This operator is the linearization of a nonlinear map $\fF_M^\phi$ which controls the deformation theory of associative submanifolds (see \autoref{prop linearization nonliear}). Let $\Upsilon_M:V_M\to U_M\subset Y$ be a tubular neighbourhood map. The map $\fF_M^\phi:C^\infty(V_M)\to C^\infty(NM)$ is defined by 
\begin{equation}\label{eq asso map}
	\inp{\fF_M^\phi u}{w}_{L^2}=\int_{\Gamma_u}\iota_w\Upsilon_M^*\psi, \ \ u\in C^\infty(V_M), w\in C_c^\infty(NM)
\end{equation}
The notation $w$ in the integrand is the extension vector field of $w\in C^\infty(NM)$ in the tubular neighbourhood as in \autoref{notation u}.  
The associative submanifolds can also be thought of as critical points of a functional \cites[Section 8]{Donaldson1998}[Section 2.3]{Doan2017d} on the space of submanifolds. The differential of this functional is a $1$-form which is locally of the form $\fF_M^\phi$. This led \citet{Doan2017d} to make a proposal of constructing Floer homology groups whose chain complex is generated by associative submanifolds.\qedhere
\end{remark}

 The following theorem summarizes the deformation theory of closed associative submanifolds, as established in the literature.
\begin{theorem}[{\cites[Theorem 5.2]{McLean1998}[Theorem 2.12]{Joyce2016}[Theorem 2.20, Proposition 2.23, Section 2.7]{Doan2017d}}]\label{thm closed asso}
Let $(Y,\phi)$ be a co-closed $G_2$-manifold. 
   \begin{enumerate}[(i)]
\item Let $P$ be an associative submanifold of $(Y, \phi)$. Then there exists an open neighbourhood $\cI_P$ of $0$ in $\ker \bD_P$ such that the moduli space $\cM^\phi$ near $P$ is homeomorphic to $\ob_P^{-1}(0)$, the zero set of a smooth map (obstruction map/Kuranishi map)
 $$\ob_P:\cI_P\to \coker \bD_P.$$
 Moreover, there is a comeager subset $\sP^{\reg}\subset \sP$ such that for all $\phi\in \sP^{\reg}$, the moduli space $\cM^\phi$ is a $0$-dimensional manifold and consists only of unobstructed (rigid) closed associative submanifolds.
\item  There is a comeager subset $\bsP^{\reg}\subset \bsP$ such that for all $\bphi\in \bsP^{\reg}$, the moduli space $\bcM^\bphi$ is a $1$-dimensional manifold and there is a discrete subset $I_o\subset [0,1]$ having the property that 
 \begin{itemize}
 \item for each $t\in [0,1]\setminus I_o$ the moduli space $\cM^{\phi_t}$ consists only of unobstructed (rigid) closed associative submanifolds.
 \item for each $\hat t \in I_o$ the moduli space $\cM^{\phi_{\hat t}}$ consists only of closed associative submanifolds ${\hat P}$ having $\dim \ker \bD_{\hat P}\leq 1$. If $P_o\in \cM^{\phi_{\hat t}}$ and $\dim \ker \bD_{P_{o}}=1$ then there exist non zero constants $a,b$ such that $\ob_{(\hat t,P_o)}$ can be written as 
 $$\ob_{(\hat t,P_o)}(t,x)= at+  bx^2+ \ \text{higher order terms}.$$
 \end{itemize}
 \end{enumerate}
\end{theorem}

Although \autoref{thm closed asso} implies that for generic $\phi$, the moduli space $\cM^\phi$of closed associative submanifolds is discrete, it does not guarantee finiteness, which is essential for defining counting invariants. This leads to the following natural-- and still open--question.
\begin{question}[{\citet[Conjecture 2.16]{Joyce2016}}]\label{question finite asso}Let $(Y,\phi)$ be a compact almost $G_2$-manifold and $\tau\in \Omega^3(Y)$ be a closed $3$-form. Let $\sP_\tau$ be the space of all co-closed $G_2$-structures that are tamed by $\tau$. 
Does there exist a comeager subset $\sP_\tau^{\blacklozenge}\subset \sP_\tau$ such that for all $\phi\in \sP_\tau^{\blacklozenge}$ the moduli space $\cM^{\phi}$ is compact?
\end{question}

If $(Y,\phi)$ is a compact almost $G_2$-manifold tamed by $\tau\in \Omega^3(Y)$, then there is a constant $c>0$ such that for every closed associative submanifold $P$ in $(Y,\phi)$ we have (see \cites[Section 3.2]{Donaldson2009}[Section 2.5]{Joyce2016})
$$\vol(P,g_\phi)\leq c [\tau]\cdot [P].$$
Therefore, we may use the following theorem of geometric measure theory to get a Federer-Fleming compactification of $\cM^{\phi}$. For a discussion on the proof of the following theorem in the special Lagrangian context we refer the reader to \cites[Section 6]{Joyce2004a}[Section 4]{Doan2018a}.  
\begin{theorem}[{\citet[section 6, Section 32]{Simon1983}, \citet{Spolaor}, \citet[Theorem 1]{Adams1988}, \citet[Theorem 6.8]{Joyce2004a}}]\label{thm cptness current}Let $P_n$ be a sequence of closed associative submanifolds in a sequence of compact almost $G_2$-manifolds $(Y,\phi_n)$ that are tamed by a fixed $\tau\in \Omega^3(Y)$. Assume $\phi_n$ converges to $\phi$ in $C^\infty$-topology. Then after passing to a subsequence $P_n$ converges (in the sense of currents) to a closed integral current $P_\infty$ which is  calibrated by $\phi$. Moreover 
\begin{enumerate}[(i)]
	\item the Hausdorff dimension of the singular set of $P_\infty$ is at most $1$.
	\item \label{thm Jacobi int}if all the tangent cones of $P_\infty$ are Jacobi integrable\footnote{A cone in $\R^7$ with smooth link $\Sigma\subset S^6$ is called Jacobi integrable if for every $\nu_\Sigma\in \ker (\bD_\Sigma+2J) \subset C^\infty(N\Sigma)$,  $\{\exp(t\nu_\Sigma):\abs{t}\ll1\}$ is a $1$-parameter family of holomorphic curves in $(S^6,J)$. Here '$\exp$' is the exponential map with respect to the round metric on $S^6$, and $\bD_\Sigma+2J$ is the deformation operator controlling the deformation theory of holomorphic curves in $(S^6,J)$ (see \autoref{prop linearization nonliear}).} multiplicity $1$ associative cones in $\R^7$ with smooth links, then $P_\infty$ is a conically singular associative submanifold of $(Y,\phi)$, in the sense of \autoref{def CS asso}. 
\end{enumerate}
\end{theorem}
The difficulty in addressing \autoref{question finite asso} arises primarily from our limited understanding of the singular set in the generic setting. A natural starting point is to analyze the simplest degeneration scenario, as described in \autoref{thm cptness current}\ref{thm Jacobi int}, where the limits are conically singular associative submanifolds. This article contributes specifically to advancing the understanding of this case.

\section{Associative cones}\label{subsection Associative cones}
This section focuses on associative cones in $\R^7$ and their links, which are holomorphic curves in $S^6$, and establishes \autoref{thm moduli holo}. We also examine the Fueter operator on these cones and compute or bound their stability index, leading to the proof of \autoref{thm stability index of cones}.
\subsection{Definition and examples of associative cones}
\begin{definition}
Let $\Sigma$ be a $2$-dimensional closed submanifold of $S^6\subset \R^7$. Define the inclusion map
 $\iota:(0,\infty)\times \Sigma\to \R^7$
 by $\iota(r,\sigma)= r\sigma$. A \textbf{cone} $C$ with link $\Sigma$ is the image of $\iota$ in $\R^7$. The Euclidean metric on $\R^7$ induces a metric $g_\Sigma$ on $\Sigma$ and a cone metric $g_{C}$ on $C$, that is, $g_{C}=dr^2+r^2g_{\Sigma}$. Furthermore, it induces a metric $g_{NC}$ and a connection $\nabla^\perp_{C}$ on the normal bundle $NC$ of the cone $C$. Let $\pi:\R^7\setminus \{0\}\to S^6$ be the projection. Then $NC=\pi^*(N\Sigma)$, pullback of the normal bundle $N\Sigma$ of $\Sigma$ in $S^6$, $g_{NC}=r^2\pi^*g_{N\Sigma}$ and $\nabla^\perp_{C}=\pi^*\nabla_\Sigma^\perp.$ 
 
 The \textbf{standard almost complex structure} on $S^6$, $J:T_xS^6\to T_xS^6$ is defined by the standard cross product `$\times$' on $\R^7$ as follows: $$J(v)=\partial_r\times v,$$ where $x\in S^6,v\in T_xS^6\subset \R^7$. 
 
 If the cone $C$ is an associative submanifold then we call it an \textbf{associative cone}. This is equivalent to saying that the link $\Sigma$ is a holomorphic curve in the almost complex manifold $(S^6,J)$. 
\end{definition}

	Any special Lagrangian cone in $\C^3$ is an associative cone in $\R^7=\R\oplus \C^3$ and its link is special Legendrian in $S^5$. \autoref{eg Transverse pair of SL planes} and \autoref{eg Harvey-Lawson $T^2$-cone} describes two examples of special Lagrangian cones in $\C^3$ that are important for the disingularization theorems (see \autoref{thm desing intersection} and \autoref{thm desing HL sing}) discussed in this article. For more examples of special Lagrangians cones, see \cites{Haskins2004b}{Haskins2007}{Joyce2001}{Joyce2002}.

\begin{example}\label{eg Transverse pair of SL planes}
(\textbf{Transverse pair of SL planes} \cite[page 328]{Joyce2003}) Let $C_\times$ be the union of a pair of special Lagrangian (SL) planes in $\C^3$ with transverse intersection at the origin. Then there exist a $B\in SU(3)$ and unique $\theta_1,\theta_2,\theta_3\in (0,\pi)$ satisfying $\theta_1\leq\theta_2\leq\theta_3$ and $\theta_1+\theta_2+\theta_3=\pi$ such that $C_\times=B\Pi_0\cup B\Pi_\theta$, where \[\Pi_0=\R^3\ \ \ \   \text{and}\ \ \ \ \  \Pi_\theta:=\diag(e^{i\theta_1},e^{i\theta_2},e^{i\theta_3})\cdot\R^3.\]
We define $\Pi_+:=B\Pi_0$ and $\Pi_-:=B\Pi_\theta$. Note that $\Pi_\pm$ are uniquely determined. 
\end{example}
\begin{example}\label{eg Harvey-Lawson $T^2$-cone}(\textbf{Harvey-Lawson $T^2$-cone} \cite[Theorem 3.1]{Harvey1982}) The Harvey-Lawson $T^2$-cone is given by
	\begin{align*}
	C_{HL}&:=\{(z_1,z_2,z_3)\in \C^3:\abs{z_1}=\abs{z_2}=\abs{z_3}, z_1z_2z_3\in (0,\infty)\}\\	
	&=\{r(e^{i\theta_1},e^{i\theta_2},e^{-i(\theta_1+\theta_2)})\in \C^3:r\in(0,\infty), \theta_1,\theta_2 \in [0,2\pi)\},
	\end{align*}
	 which is a special Lagrangian cone in $\C^3$ whose link $\Sigma_{HL}=C_{HL}\cap S^5$ is isometric to the flat Clifford torus ~$T^2$.
\end{example}
\begin{example}(\textbf{Null torsion holomorphic curves} \cite[Section 4]{Bryant1982}\label{eg Bryant null torsion})We follow here the exposition about null torsion holomorphic curves given in \cite{Madnick2021}. 
Let $\Sigma$ be a closed holomorphic curve in $(S^6,J)$. Then the characteristic $\SU(3)$ connection $\widetilde\nabla$ (see \autoref{eq characteristic SU(3) connection}) induces holomorphic structures on $TS^6_{|\Sigma},T\Sigma$ and $N\Sigma$. The second fundamental form of $\Sigma$ is the obstruction ${{\operatorname{II}}}\in \Ext^1(N\Sigma, T\Sigma)\cong H^0(\Sigma,K_\Sigma^2\tn N\Sigma)$ (by Serre duality) to the holomorphic splitting of the following exact sequence:
$$0\to T\Sigma\to TS^6_{|\Sigma}\to N\Sigma\to0.$$ Moreover, ${\operatorname{II}}\neq 0$ if and only if $\Sigma$ is not a totally geodesic $S^2$. In this case, denote the effective divisor of the zero set of $\operatorname{II}$ by $Z$. We define a holomorphic line bundle $L_B$ by the following exact sequence:
\begin{equation}\label{eq LB}
0\to L_N:=K_\Sigma^{-2}\tn\mathcal O(Z)\hookrightarrow N\Sigma \rightarrow L_B \to0.
\end{equation} 
The {torsion} of $\Sigma$ is the obstruction $\operatorname{III}\in \Ext^1(L_B,L_N)\cong H^0(\Sigma,K_\Sigma^3\tn\mathcal O(-Z)\tn L_B)$ (by Serre duality) to the holomorphic splitting of the exact sequence \autoref{eq LB}.

 If $\Sigma$ is \textbf{null-torsion} (i.e. ${\operatorname{II}}\neq 0,\operatorname{III}=0$), then there is a holomorphic isomorphism $$L_B\cong K_\Sigma^3\tn \mathcal O(-Z)$$
and $$\operatorname{Area}(\Sigma)=4\pi b\geq 24\pi,$$ where $ b=-c_1(L_B)=3\chi(\Sigma)+[Z].$ Moreover, no null torsion holomorphic curves in $S^6$ are contained in a totally geodesic $S^5$.
 If $\Sigma$ is of genus zero and not a totally geodesic $S^2$, then it must be a null-torsion holomorphic curve. \citet[Theorem 4.10]{Bryant1982} and later \citet{Rowland1999}
proved that closed Riemann surfaces of any genus can be conformally embedded as a null-torsion $J$-holomorphic curve in $S^6$. \end{example}

We would like to point out that there are more examples of associative cones which are not special Lagrangians discussed in \cites[Section 7]{Lotay2011}{Lotay2007a}.

\subsection{Moduli space of holomorphic curves in $S^6$.}
\begin{definition}\label{def moduli holo}
Let $\mathcal S$ be the set of all $2$-dimensional oriented, closed smooth submanifolds of ~$S^6$. Equip $\mathcal S$ with $C^\infty$-topology in the same way as in \autoref{def moduli closed asso}. 

The \textbf{moduli space}  $\cM^{\operatorname{hol}}$ of embedded holomorphic  curves in $(S^6,J)$ is the subset of all submanifolds $\Sigma$ in $\mathcal S$ that are $J$-holomorphic. The topology on $\cM^{\operatorname{hol}}$ is the subspace topology of the above. 
\end{definition}
Let  $\Sigma$ be a holomorphic curve in $S^6$, that is,  $\Sigma\in \cM^{\operatorname{hol}}$. We denote the complex structure on $\Sigma$ by $j$, which is just the restriction of $J$. Let $\Upsilon_\Sigma:V_\Sigma\to U_\Sigma$ be a tubular neighbourhood map of $\Sigma$.  For $u\in C^\infty(V_\Sigma)$ denote the submanifold $\Upsilon_{\Sigma}(\Gamma_u)$ by $\Sigma_u$. Note that, $\Sigma_u$ is $J$-holomorphic if and only if $u$ satisfies the following \textbf{non-linear Cauchy--Riemann equation}:
 $$0=\bar{\partial}_Ju:=\frac12(du+\Upsilon_{\Sigma}^*J(u) \circ du\circ j)\in C^\infty(\overline{\Hom}_\C(T\Sigma,u^*TV_\Sigma)).$$  
   The linearization of the nonlinear map $\bar{\partial}_J:C^\infty(V_\Sigma)\to C^\infty(\overline{\operatorname{Hom}}_\C(T\Sigma,u^*TV_\Sigma))$ at the zero section is described in \cite[Proposition 3.1.1]{McDuff2012}. This is the linear map 
$\mathfrak d_{\Sigma,J}:C^\infty(N\Sigma)\to C^\infty(\overline{\operatorname{Hom}}_\C(T\Sigma,T{V_\Sigma}_{|\Sigma}))$ defined by 
$$\mathfrak d_{\Sigma,J}\xi:=\frac12(\nabla_\Sigma\xi+J\circ (\nabla_\Sigma\xi)\circ j+\nabla_\xi J\circ j), \ \ \xi\in C^\infty(N\Sigma).$$
The tangential component of $\mathfrak d_{\Sigma,J}$ can be discarded for the deformation theory. The following normal component actually controls the deformation theory.
\begin{definition}\label{def normal Cauchy Riemann}For $\Sigma\in \cM^{\operatorname{hol}}$, the \textbf{normal Cauchy--Riemann operator} $\bar\partial^N_{\nabla}:C^\infty(N\Sigma)\to C^\infty(\overline{\operatorname{Hom}}_\C(T\Sigma,N\Sigma))\cong \Omega^{0,1}(\Sigma, N\Sigma)$ is defined by
\[\bar\partial^N_{\nabla}\xi:=\frac12(\nabla_\Sigma^\perp\xi+J\circ (\nabla_\Sigma^\perp\xi)\circ j+\nabla^\perp_\xi J\circ j).\qedhere\]
\end{definition}
The moduli space of holomorphic curves is usually studied using the non-linear Cauchy--Riemann map  (see \cite[Theorem 3.1.5]{McDuff2012}) but here we study it in a different way.

 On $\R^7$, the Euclidean metric is $g_e= dr^2+r^2g$ and the standard $G_2$-structure $\phi_e$, $\psi_e=*_{g_e}\phi_e$ can be written as
$$\phi_e=r^2dr\w \omega+r^3\Re\Omega,\ \ \ \ 
\psi_e=\frac{r^4}2\omega^2-r^3dr\w \operatorname{Im}\Omega$$
where $\omega(X,Y):=g(JX,Y)$, and $\Omega$ is a nowhere vanishing complex $(3,0)$-form.  Together they give an $\SU(3)$-struture on $S^6$. Also, we have:
$$\operatorname{vol}_{\R^7}=r^6dr\w\operatorname{vol}_{S^6}\ \ \text{and} \ \operatorname{vol}_{S^6}=\frac{\omega^3}6=\frac i8\Omega\w \bar{\Omega}=\frac 14\operatorname{Re}\Omega\w \operatorname{Im}\Omega.$$
In particular, $*_{g}\operatorname{Re}\Omega=\operatorname{Im}\Omega,\ *_{g}\operatorname{Im}\Omega=-\operatorname{Re}\Omega.$ Moreover, $d\phi_e=0$ and $d\psi_e=0$ on $\R^7$ are equivalent to the following equations on $S^6$, respectively:
$$d\omega=3\operatorname{Re}\Omega\ \ \ \ \text{and}\ \ \ \ d\operatorname{Im}\Omega=-2\omega^2.$$
That means $S^6$ with this $\SU(3)$ structure is a \textbf{nearly Kähler} manifold. Let $\nabla$ be the Levi-Civita connection on $S^6$ with respect to the round metric $g$. The \textbf{characteristic $\SU(3)$ connection} given by
\begin{equation}\label{eq characteristic SU(3) connection}
\widetilde\nabla_uv:=\nabla_uv+\frac12(\nabla_uJ)Jv,
\end{equation}
satisfies $\widetilde\nabla g=0$, $\widetilde\nabla J=0$ and $\operatorname{Hol}(\widetilde \nabla)\subset SU(3)$. It turns out that $\widetilde\nabla$ is not torsion free. Furthermore, it induces the following Cauchy--Riemann operator.
\begin{definition}Let $\Sigma$ be a $J$-holomorphic embedded curve in $S^6$.
The \textbf{characteristic Cauchy--Riemann operator} on $N\Sigma$, $\bar\partial^N_{{\widetilde\nabla}}:C^\infty(N\Sigma)\to \Omega^{0,1}(N\Sigma)$ is defined to be the induced Cauchy--Riemann operator from the characteristic $\SU(3)$ connection $\widetilde \nabla$, that is, 
\begin{equation*}\bar\partial^N_{{\widetilde\nabla}}\xi:=\frac12(\widetilde\nabla^\perp\xi+J\circ (\widetilde\nabla^\perp\xi)\circ j).\qedhere \end{equation*}
\end{definition}
The following lemma will help us to relate the two Cauchy--Riemann operators $\bar\partial^N_{{\widetilde\nabla}}$ and $\bar\partial^N_{\nabla}$.
\begin{lemma}[{\citet[Lemma C.7.1]{McDuff2012}}]\label{lem nearly kahler }For any vector fields $u,v,w$ on $S^6$,
\begin{enumerate}[(i)]
\item $(\nabla_uJ)Jv=-J(\nabla_uJ)v$ and $g((\nabla_uJ)v,w)=-g((\nabla_uJ)w,v)$,
\item $(\nabla_uJ)v=-(\nabla_vJ)u$ and the torsion $T_{\widetilde\nabla}(u,v)=\frac 14N_J(u,v)=(\nabla_uJ)Jv$,
\item  $3g((\nabla_uJ)v,w)=d\omega(u,v,w)=3\operatorname{Re}\Omega(u,v,w)$.
 \end{enumerate} 
\end{lemma}

\begin{definition}\label{def J-antilinear multiplication map}
The \textbf{multiplication map}, $\times_{S^6}:TS^6\times TS^6\to TS^6$ is defined by the orthogonal projection on $TS^6$ of the cross product in $\R^7$, or equivalently for all vector fields $u,v,w$ on $S^6$, \[g(u\times_{S^6}v,w)=\operatorname{Re}\Omega(u,v,w),\  \text{or equivalently}\ u\times_{S^6}v=(\nabla_uJ)v.\qedhere\]
\end{definition}

\begin{remark}
	\label{rmk cross product S^6}
Let $\Sigma$ be an oriented smooth surface in $S^6$ and $C$ be the cone in $\R^7$ with link $\Sigma$. The following are equivalent: (i) $\Sigma$ is $J$-holomorphic, (ii) $C$ is associative, (iii) for all $u,v\in T\Sigma$ and $w\in TS^6$, $\operatorname{Re}\Omega(u,v,w)=0$, (iv) for all $u,v\in T\Sigma$, $u\times_{S^6}v=0$, (v) for all $u\in T\Sigma$ and $v\in N\Sigma$, $u\times v\in N\Sigma$.
\end{remark}
The following proposition is the desired relation between the above Cauchy--Riemann operators.
\begin{prop}\label{prop relation two cauchy rimann}
Let $\Sigma$ be an embedded $J$-holomorphic curve in $S^6$. Then for all $\xi\in C^\infty(N\Sigma)$
$$\bar\partial^N_{\nabla}\xi=\bar\partial^N_{{\widetilde\nabla}}\xi-(\nabla^\perp J)J\xi.$$
Here $\nabla^\perp $ denotes the normal connection on $N\Sigma$, and its action on
$\End(N\Sigma)$ is understood as the one induced from this normal connection. 
\end{prop}
\begin{proof} By \autoref{lem nearly kahler }, for all $\xi\in C^\infty(N\Sigma)$, we have
$$J\circ \widetilde\nabla\xi\circ j=J\circ \nabla\xi\circ j+\frac12J(\nabla J)J\xi\circ j=J\circ \nabla\xi\circ j+\frac12(\nabla J)J\xi.$$
The proposition now follows from the definitions.
\end{proof}

The following defines a canonical Dirac operator on a holomorphic curve $\Sigma$, which is also related to the above Cauchy--Riemann operators.
\begin{definition}
Let $\Sigma$ be an embedded $J$-holomorphic curve in $S^6$. By \autoref{lem nearly kahler } and \autoref{rmk cross product S^6}, the map $\gamma_\Sigma:T\Sigma\to \overline{\End}_\C(N\Sigma)$ given by 
$$\gamma_\Sigma(f_\Sigma)(v_\Sigma):=f_\Sigma\times v_\Sigma,\ \ \ \ \forall \  v_\Sigma \in N\Sigma, \ \  f_\Sigma \in T\Sigma,$$
is a skew symmetric $J$-anti-linear Clifford multiplication. Moreover the normal bundle $N\Sigma$ together with the metric $g_{N\Sigma}$, Clifford multiplication $\gamma_\Sigma$ and the metric connection $\widetilde \nabla^\perp:=\widetilde \nabla^\perp_\Sigma$ is a {Dirac bundle}, that is, $\widetilde \nabla^\perp\gamma_\Sigma=0$. The associated {Dirac operator} is given by
\begin{equation}\label{eq dirac sigma}
	\mathbf D_\Sigma:=\sum_{i=1}^2f_i\times \widetilde \nabla_{f_i}^\perp,
\end{equation}
	where $\{f_i\}$ is a local orthonormal oriented frame on $\Sigma$. Note that $\mathbf D_\Sigma$ is $J$-anti-linear. Moreover the map $\gamma_\Sigma$ induces a $J$-anti-linear isomorphism 
	$$\boldsymbol{\gamma}_\Sigma: \overline{\operatorname{Hom}}_\C(T\Sigma,N\Sigma)\to {N\Sigma},$$
	which is defined by
\begin{equation*}	
	\boldsymbol{\gamma}_\Sigma(f^*_\Sigma\tn v_\Sigma)=\gamma_\Sigma(f_\Sigma)(v_\Sigma).
\qedhere \end{equation*}
\end{definition}
\begin{remark}\label{rmk boldsymbol{gamma}_Sigma CS}
	$\boldsymbol{\gamma}_\Sigma$ fits into the following commutative diagram:
\[	
\begin{tikzcd}
C^\infty(N\Sigma) \arrow{r}{\bar \partial ^N_{\widetilde \nabla,J}} \arrow[swap]{dr}{\mathbf D_\Sigma} & C^\infty(\overline{\operatorname{Hom}}_\C(T\Sigma,N\Sigma)) \arrow{d}{\boldsymbol{\gamma}_\Sigma} \\
 &C^\infty(N\Sigma).
\end{tikzcd}\qedhere
 \]
\end{remark}
The following proposition is the desired relation between the above Cauchy--Riemann operators and the Dirac operator.
\begin{prop}\label{prop relation dirac and cauchy riemann} Let $\{f_i\}$ be a local orthonormal oriented frame on $\Sigma$. Then 
\[\mathbf D_\Sigma=\sum_{i=1}^2f_i\times \nabla_{f_i}^\perp-J, \qandq \boldsymbol{\gamma}_\Sigma\circ \bar\partial^N_{\nabla} =\mathbf D_\Sigma+2J.\]
\end{prop}

\begin{proof}  First we prove that $\boldsymbol{\gamma}_\Sigma((\nabla^\perp J)J)=-2J$. For a local oriented orthonormal frame $\{f_1,f_2=jf_1\}$ on $\Sigma$ we have $\nabla_{f_i}^\perp J={\gamma}_\Sigma f_i$, by \autoref{def J-antilinear multiplication map} and \autoref{lem nearly kahler }. Therefore 
\begin{equation*}\boldsymbol{\gamma}_\Sigma((\nabla^\perp J)J)=\sum_{i=1}^2f_i\times(\nabla_{f_i}^\perp J)J=\sum_{i=1}^2({\gamma}_\Sigma f_i)^2 J=-2J.\end{equation*}
Now the first equality in the proposition follows from \autoref{eq  dirac sigma} and \autoref{eq characteristic SU(3) connection}. The second equality follows from 
\autoref{prop relation two cauchy rimann} and \autoref{rmk boldsymbol{gamma}_Sigma CS}.
\end{proof}
The moduli space of holomorphic curves can be expressed locally as the zero set of the following non-linear map.
\begin{definition}\label{def nonlinear holo in S6}
 Let $\Sigma$ be a holomorphic curve in $S^6$ and $\Upsilon_\Sigma: V_\Sigma \to S^6$ be a tubular neighbourhood map.  We define $\mathcal F:C^\infty(V_\Sigma)\to C^\infty(N\Sigma)$  by
 \begin{equation*}
 \inp{\mathcal F(u)}{v}_{L^2}:=\int_{\Gamma_u}\iota_v(\Upsilon_\Sigma^*\Re\Omega),
 \end{equation*}
 where $u\in C^\infty(V_\Sigma)$ and $v\in C^\infty(N\Sigma)$. The notation $v$ in the integrand is the extension vector field of $v\in C^\infty(N\Sigma)$ over the tubular neighbourhood $V_\Sigma$ as in \autoref{notation u}.
\end{definition} 

The proof of \autoref{thm moduli holo} requires computing the linearization of the map; the following proposition concerns this computation.

\begin{prop}\label{prop lin of nonlinear holo}For $u\in C^\infty(V_\Sigma)$, we have $\mathcal F(u)=0$ if and only if the graph $\Sigma_u:=\Upsilon_{\Sigma}(\Gamma_u)$ is $J$-holomorphic.  If $\Sigma\in \cM^{\operatorname{hol}}$, then the linearization of $\mathcal F$ at zero, $d\mathcal F_0:C^\infty(N\Sigma)\to C^\infty(N\Sigma)$ 
is given by $$d\mathcal F_0=J\mathbf D_\Sigma-2=J(\mathbf D_\Sigma+2J).$$
This is a formally self adjoint first order elliptic operator of index $0$.
\end{prop}
\begin{proof}The first part follows from \autoref{rmk cross product S^6}. Let $\{f_1,f_2=Jf_1\}$ be a local oriented orthonormal frame for $T\Sigma$. For $u, v \in C^\infty(N\Sigma)$, 
$$\frac d{dt}\big|_{{t=0}}\inp{\mathcal F(tu)}{v}_{L^2}=\frac d{dt}\big|_{{t=0}} \int_{\Gamma_{tu}}\iota_v(\Upsilon_\Sigma^*\operatorname{Re}\Omega)
=\int_{\Sigma}\mathcal L_u\iota_v(\Upsilon_\Sigma^*\operatorname{Re}\Omega).$$ 
This is same as $\int_{\Sigma}\iota_v\mathcal L_u(\Upsilon_\Sigma^*\operatorname{Re}\Omega)+{\iota_{[u, v]}(\Upsilon_\Sigma^*\operatorname{Re}\Omega)}$. As, $[u, v]=0$ (by \autoref{rmk commutator}, following \autoref{notation u}) and $\Sigma\in \cM^{\operatorname{hol}}$, this is equal to
\begin{align*}
&\int_{\Sigma}\iota_v(\Upsilon_\Sigma^*\operatorname{Re}\Omega)(\nabla_{f_1}u,f_2)+\iota_v(\Upsilon_\Sigma^*\operatorname{Re}\Omega)(f_1,\nabla_{f_2}u)+\iota_v \nabla_u(\Upsilon_\Sigma^*\operatorname{Re}\Omega)\\
&=\int_{\Sigma}-\inp{f_2\times \nabla^{\perp}_{f_1}u}{v}+\inp{f_1\times \nabla^{\perp}_{f_2}u}{v}-\iota_v(u\wedge\Upsilon_\Sigma^*\omega)\ (\text{as}\  \nabla\Re\Omega=-\omega\wedge\cdot)\\
&=\int_{\Sigma}\inp{J(f_1\times \nabla^{\perp}_{f_1}u)}{v}+\inp{J(f_2\times \nabla^{\perp}_{f_2}u)}{v}-\inp{u}{v}=\int_{\Sigma}\inp{(J\mathbf D_\Sigma-2) u}{v}.
\end{align*}
 Therefore $d\mathcal F_0=J\mathbf D_\Sigma-2$.
Since $\mathbf D_\Sigma$ is a $J$-anti-linear (this follows from the fact that $\widetilde \nabla J=0$ and $\gamma_\Sigma$ is $J$-anti-linear), formally self adjoint Dirac operator, it proves the last part of the proposition.
\end{proof}
\begin{proof}[{\normalfont{\textbf{Proof of \autoref{thm moduli holo}}}}]
In \autoref{prop relation dirac and cauchy riemann}, we see that $\boldsymbol{\gamma}_\Sigma\circ \bar\partial^N_{\nabla}=\mathbf D_\Sigma+2J$. Extending the nonlinear map $\mathcal F$ to Hölder spaces we get a smooth map
 $$\mathcal F:C^{2,\gamma}(V_\Sigma)\to C^{1,\gamma}(N\Sigma).$$
\autoref{prop lin of nonlinear holo} implies that the linearization of $\mathcal F$ at zero is an elliptic operator and hence is Fredholm. By implicit function theorem applied to $\mathcal F$ we obtain the map $\ob_\Sigma$ as stated in the theorem (see \cite[Proposition 4.2.19]{Donaldson1990}). We only prove that, if $u\in C^{2,\gamma}(V_\Sigma)$ with $\mathcal F(u)=0$, then $u\in C^\infty(V_\Sigma)$. To prove this, we observe  
 $$0=\bD_\Sigma(\mathcal F(u))=a(u,\nabla_\Sigma^\perp u) (\nabla_\Sigma^\perp)^2 u+ b(u,\nabla_\Sigma^\perp u)$$
 Since $a(u,\nabla_\Sigma^\perp u)\in C^{1,\gamma}$ and  $b(u,\nabla_\Sigma^\perp u)\in C^{1,\gamma}$, by Schauder elliptic regularity (see \cite[Theorem 1.4.2]{Joyce2007}), we obtain $u\in C^{3,\gamma}$. By repeating this argument, we get higher regularity, which completes the proof of the theorem. 
\end{proof}

\subsection{The Fueter operator of an associative cone}The operator controlling the deformation theory for asymptotically conical (AC) or conically singular (CS) associatives will be an  AC or CS uniformly elliptic operator, asymptotic to a conical elliptic operator (see \autoref{def CS/AC elliptic operator}). This conical opperator is the Fueter operator on the asymptotic associative cone. The Fredholm theory of AC and CS uniformly elliptic operator has been studied by \citet{Lockhart1985}, see also \cites{Marshall2002}{Karigiannis_2020}. It suggests that we must study the indicial roots and homogeneous kernels of this Fueter operator.

Let $C$ be an associative cone in $\R^7$ with link $\Sigma$. The Fueter operator (see \autoref{def D_M}) on the cone $C$, $\mathbf D_C:C^\infty(NC)\to C^\infty(NC)$ is 
	$$\mathbf D_C=\sum_{i=1}^3e_i\times \nabla_{C,e_i}^\perp,$$
	where $\{e_i\}$ is a local oriented orthonormal frame on $C$ and $\nabla_{C}^\perp$ is the normal connection induced by the Levi-Civita connection on $\R^7$.

\begin{definition}[Homogeneous kernel and indicial roots]\label{def homogeneous kernel}	
	For $\lambda\in \R$ we define the \textbf{homogeneous kernel} of rate $\lambda$ by
	$$V_\lambda:=\{r^{\lambda-1}\nu_\Sigma\in C^\infty(NC): \nu_\Sigma\in C^\infty(N_{S^6}\Sigma),\ \mathbf D_C(r^{\lambda-1}\nu_\Sigma)=0\}.$$ It's dimension is denoted by 
	 $$d_\lambda:=\dim V_\lambda.$$ 
	 
	The set of \textbf{indicial roots} (or, critical rates) is defined by \begin{equation*}\mathcal D_C:=\{\lambda\in \R: d_\lambda\neq0\}.\qedhere \end{equation*}
	\end{definition}
\begin{remark}Note that, $\abs {r^{\lambda-1}\nu_\Sigma}_{g_{NC}}=r^\lambda\abs{\nu_\Sigma}_{g_\Sigma}$.
\end{remark}
The following proposition not only studies properties of the Fueter operator but also describes the homogeneous kernel as an eigenspace of an operator on the link $\Sigma$, which will be useful later. Moreover, the almost complex structure $J$ induces a symmetry on the homogeneous kernels.

\begin{prop}\label{prop Fueter}
	Let $\nu_\Sigma$ be an element in  $C^\infty(N_{S^6}\Sigma)$ and $\lambda\in \R$. Then the following hold
	  \begin{enumerate}[(i)]  				
	    \item  \label{prop Fueter decomposition} $\mathbf D_C=J\partial_r+\frac1r\mathbf D_\Sigma+\frac2rJ$.
		\item \label{prop Fueter decomposition homogeneous}$\mathbf D_C(r^{\lambda-1}\nu_\Sigma)=r^{\lambda-2}\big(\mathbf D_\Sigma \nu_\Sigma+(\lambda+1) J\nu_\Sigma \big)$.
		\item \label{prop Fueter decomposition general homogeneous}$\mathbf D_C(r^{\lambda-1}(\log r)^j\nu_\Sigma)=r^{\lambda-2}(\log r)^j\big(\mathbf D_\Sigma \nu_\Sigma+(\lambda+1) J\nu_\Sigma \big)+jr^{\lambda-2}(\log r)^{j-1}J\nu_\Sigma$.
		\item \label{prop Fueter anti linear}$\mathbf D_\Sigma(J\nu_\Sigma)=-J\mathbf D_\Sigma\nu_\Sigma$.
		\item \label{prop Fueter homogeneous kernel}$V_\lambda=\{r^{\lambda-1}\nu_\Sigma: \mathbf D_\Sigma \nu_\Sigma=-(\lambda+1) J\nu_\Sigma \}= \{r^{\lambda-1}\nu_\Sigma: (J\mathbf D_\Sigma) \nu_\Sigma=(\lambda+1) \nu_\Sigma \}$
		\item\label{prop Fueter symmetry} $JV_{-1+\lambda}=V_{-1-\lambda}$ and $d_{-1+\lambda}=d_{-1-\lambda}$ for all $\lambda\in \R$.
\end{enumerate}
\end{prop}
\begin{proof}
 We have $\nabla_{C}^\perp=dr\tn {\partial_r}+\nabla_{\Sigma}^\perp$. For a local oriented orthonormal frame $\{f_i\}$ on $\Sigma\subset S^6$ and $v(r,\sigma)\in C^\infty(NC)$, applying \autoref{prop relation dirac and cauchy riemann} we compute
\begin{align*}
\mathbf D_Cv(r,\sigma)&=\partial_r \times \nabla_{C,\partial_r}^\perp v(r,\sigma) +\frac 1{r^2}\sum_{i=1}^2f_i\times \nabla_{C,f_i} ^\perp v(r,\sigma)	\\
&=\partial_r \times_{S^6} \partial_r v(r,\sigma)+\frac 1r \partial_r \times_{S^6} v(r,\sigma)  +\frac 1{r}\sum_{i=1}^2f_i\times_{S^6} \nabla_{\Sigma,f_i} ^\perp v(r,\sigma)\\&=J\partial_r v(r,\sigma)+\frac1r\mathbf D_\Sigma v(r,\sigma)+\frac2rJ.
\end{align*}
This proves (i). Now (ii) and (iii) follows from (i), indeed
$$\mathbf D_C(r^{\lambda-1}(\log r)^j\nu_\Sigma)=r^{\lambda-2}(\log r)^j\big((\lambda-1) J\nu_\Sigma +\mathbf D_\Sigma \nu_\Sigma+2J\nu_\Sigma\big)+jr^{\lambda-2}(\log r)^{j-1} J\nu_\Sigma .$$
Finally, (v) follows from (ii) and (vi) follows from (iv). And (iv) follows from the fact that $\widetilde \nabla J=0$ and $\gamma_\Sigma$ is $J$-anti-linear.
\end{proof}

Although it follows from \cite[Equation 1.11]{Lockhart1985} that a general element in the homogeneous kernel is of the form $\sum_{j=0}^mr^{\lambda-1}(\log r)^j\nu_{\Sigma,j}$, the following \autoref{prop no sum homogeneous} implies that it must be of the form we have defined in \autoref{def homogeneous kernel}. We also see that there is a canonical one to one correspondence between the indicial roots and eigenvalues of the self adjoint elliptic operator $J\mathbf D_\Sigma-1$. In particular they are countable, discrete and will have finite intersection with any closed bounded interval of $\R$. Moreover, \autoref{prop Fueter}\ref{prop Fueter symmetry} implies that the indicial roots and homogeneous kernels are symmetric with respect to $-1$. 
	
\begin{prop}\label{prop no sum homogeneous}
Let $m\in \N_0$ and $v(r,\sigma)=\displaystyle\sum_{j=0}^mr^{\lambda-1}(\log r)^j\nu_{\Sigma,j} \in C^\infty(NC)$ and $\nu_{\Sigma,j}\in  C^\infty(N_{S^6}\Sigma) $. If $\ \mathbf D_Cv(r,\sigma)=0$ then $m=0$ and therefore $v(r,\sigma)=r^{\lambda-1}\nu_{\Sigma,0}$. 	
\end{prop}
\begin{proof}
	If $\mathbf D_Cv(r,\sigma)=0$, then by \autoref{prop Fueter}\ref{prop Fueter decomposition general homogeneous} and comparing the coefficients of $r^{\lambda-2}(\log r)^{j-1}$, $j\geq 1$ we see that
 $$\mathbf D_\Sigma \nu_{\Sigma,m}+(\lambda+1) J\nu_{\Sigma,m}=0 \ \text{and}\ jJ\nu_{\Sigma,j}+\mathbf D_\Sigma \nu_{\Sigma,j-1}+(\lambda+1) J\nu_{\Sigma,j-1}=0.$$
 Therefore, 
 \begin{align*}
m\norm{\nu_{\Sigma,m}}^2_{L^2(\Sigma)}&=-\inp{J\nu_{\Sigma,m}}{\mathbf D_\Sigma \nu_{\Sigma,m-1}+(\lambda+1) J\nu_{\Sigma,m-1}}_{L^2(\Sigma)}\\
&=\inp{J(\mathbf D_\Sigma \nu_{\Sigma,m}+(\lambda+1) J\nu_{\Sigma,m})}{\nu_{\Sigma,m-1}}_{L^2(\Sigma)}=0.
 \end{align*}
  Here we have used \autoref{prop Fueter}\ref{prop Fueter anti linear}. The proof is completed by backwards induction starting with $j=m$.
\end{proof}

Whenever the above associative cone is a special Lagrangian cone, the homogeneous kernel can be expressed in a more explicit form with Hodge--deRham operators, which in turn will help us compute or estimate lower bounds for the stability-index. The following discusses this in detail.

\begin{definition}
	Let $L$ be a special Lagrangian submanifold in $\C^3\subset \R\oplus \C^3$. Set $e_1=(1,0)\in\R\oplus\C^3$. We define the isometry $\Phi_L:C^\infty(NL)\to \Omega^0(L,\R)\oplus \Omega^1(L,\R)$ by
$$\Phi_L(\nu):=(\inp{e_1}{\nu},(e_1\times\nu)^\flat).$$ 
We denote by $ \widecheck{\mathbf D}_L$ the conjugation of the Fueter operator $\bD_L$ defined in \autoref{def D_M} under $\Phi_L$, that is, 
\[\widecheck{\mathbf D}_L\coloneqq \Phi_{L}{\mathbf D}_{L}\Phi_{L}^{-1}: \Omega^0(L,\R)\oplus \Omega^1(L,\R)\to  \Omega^0(L,\R)\oplus \Omega^1(L,\R).\] 
A direct computation shows that	
\begin{equation}\label{eq hatDL}
\widecheck{\mathbf D}_L=
\begin{bmatrix}
0 & d^* \\
d & *d \\
\end{bmatrix}.\qedhere
\end{equation}
\end{definition}
A direct computation yields the following lemma.
\begin{lemma}\label{lem direct computation}
	Let $C$ be a cone in $\C^3$ with link $\Sigma$, then for $(f_\Sigma,h_\Sigma,\sigma_\Sigma)\in\Omega^0(\Sigma,\R)\oplus\Omega^0(\Sigma,\R)\oplus \Omega^1(\Sigma,\R)$  we have
	\begin{enumerate}[(i)]
		\item $d(r^\lambda f_\Sigma)=\lambda r^{\lambda-1}f_\Sigma dr+r^\lambda d_\Sigma f_\Sigma$.
		\item $d(r^{\lambda}h_\Sigma dr+r^{\lambda+1} \sigma_\Sigma)=-r^{\lambda}dr\w d_\Sigma h_\Sigma+(\lambda+1)r^\lambda dr\w\sigma_\Sigma+r^{\lambda+1}d_\Sigma \sigma_\Sigma$.
		\item $*dr=r^2*_\Sigma 1,\ *\sigma_\Sigma=-dr\w*_\Sigma\sigma_\Sigma,  *(dr\w\sigma_\Sigma)=*_\Sigma\sigma_\Sigma, \ * d_\Sigma \sigma_\Sigma=r^{-2}dr$.
		\item $d^{*}(r^{\lambda}h_\Sigma dr+r^{\lambda+1} \sigma_\Sigma)=-(\lambda+2)r^{\lambda-1}h_\Sigma+r^{\lambda-1} d_\Sigma^{*_\Sigma}(\sigma_\Sigma)$.
		\item ${*}d(r^{\lambda}h_\Sigma dr+r^{\lambda+1} \sigma_\Sigma)=-r^\lambda*_\Sigma d_\Sigma h_\Sigma+(\lambda+1)r^\lambda{*_\Sigma}\sigma_\Sigma+r^{\lambda-1} ({*_\Sigma} d_\Sigma\sigma_\Sigma )dr$.
	\end{enumerate}
\end{lemma}

\begin{prop}\label{prop indicial laplace}If $C$ is a special Lagrangian cone in $\C^3$ whose link is $\Sigma$, then the homogeneous kernel from \autoref{def homogeneous kernel} is:
\begin{equation*}
V_\lambda\cong \begin{cases*}
		\{\sigma_\Sigma\in \Omega^1(\Sigma,\R):\Delta_\Sigma \sigma_\Sigma=0\}& if $\lambda=-1,$\\
		\{(f_\Sigma,h_\Sigma)\in\Omega^0(\Sigma,\R){\oplus}\Omega^0(\Sigma,\R): (f_\Sigma,h_\Sigma) \ \text{satisfies}\ \autoref{eq laplace SL Vlambda}\} & if $\lambda\neq -1$,
		\end{cases*}
		\end{equation*}
\begin{equation}\label{eq laplace SL Vlambda}
	\Delta_\Sigma f_\Sigma=\lambda(\lambda+1)f _\Sigma,\ \ \ \ 
	\Delta_\Sigma h_\Sigma=(\lambda+2)(\lambda+1)h _\Sigma.
\end{equation}
		
\end{prop}
\begin{proof}Since $C$ is a special Lagrangian, by \autoref{eq hatDL} we have
$$V_\lambda\cong\{(f_\Sigma,h_\Sigma,\sigma_\Sigma)\in\Omega^0(\Sigma,\R)\oplus\Omega^0(\Sigma,\R)\oplus \Omega^1(\Sigma,\R):(f_\Sigma,h_\Sigma,\sigma_\Sigma) \ \text{satisfies}\ \autoref{eq: homogeneous kernel SL}\},$$
\begin{equation}\label{eq: homogeneous kernel SL}
\begin{cases*}
	d^*(r^{\lambda}h_\Sigma dr+r^{\lambda+1} \sigma_\Sigma)=0\\
	d(r^\lambda f_\Sigma)+{*}d(r^{\lambda}h_\Sigma dr+r^{\lambda+1} \sigma_\Sigma)=0.
	\end{cases*}\end{equation}
By \autoref{lem direct computation} we obtain that \autoref{eq: homogeneous kernel SL} is equivalent to the following:
\begin{equation}\label{eq:indicial SL}
\begin{cases}
(\lambda+2)h_\Sigma=d_\Sigma^{*_\Sigma}(\sigma_\Sigma)\\
\lambda f_\Sigma= -{*_\Sigma} d_\Sigma(\sigma_\Sigma)\\
(\lambda+1){*_\Sigma}\sigma_\Sigma=-d_\Sigma f_\Sigma+*_\Sigma d_\Sigma h_\Sigma.
\end{cases}
\end{equation}
This yields the required proposition.
\end{proof}

\begin{cor}\label{cor homogeneous kernel} If $C$ is a special Lagrangian cone in $\C^3$ whose link is $\Sigma$, then  
\begin{equation*}
V_\lambda\cong\begin{cases*}
		 \{h_\Sigma\in\Omega^0(\Sigma,\R):\Delta_\Sigma h_\Sigma=(\lambda+2)(\lambda+1)h _\Sigma\}& if $\lambda\in (-1,0)$\\
		H^0(\Sigma,\R)\oplus\{h_\Sigma\in\Omega^0(\Sigma,\R):\Delta_\Sigma h_\Sigma=2h _\Sigma\}& if $\lambda=0$\\
		H^1(\Sigma,\R)& if $\lambda=-1$\\
		 \{f_\Sigma\in\Omega^0(\Sigma,\R):\Delta_\Sigma f_\Sigma=(\lambda+1)\lambda f _\Sigma\}& if $\lambda\in (-2,-1).$ 
		\end{cases*}
		\qedhere
	\end{equation*}
		\end{cor}

\begin{remark}\label{rmk d0 geq 7} If the associative cone $C$ is not a $3$-plane then all the translations by vectors in $\R^7$ yield a $7$-dimensional subspace of $V_0$ and hence $d_0\geq 7$. Indeed, there are no non-trivial translations that preserve the cone $C$ because $0$ is the unique singular point of $C$.  Thus if $C$ is not a $3$-plane, then the lower stability-index from \autoref{def uplo stability index} has the following lower bound:
\begin{equation*}\frac {d_{-1}}2\leq \operatorname{s-ind}_-(C).\qedhere \end{equation*}
\end{remark}

The following two examples compute the stability-indices of a pair of transverse associative planes, and the Harvey-Lawson $T^2$-cone.
\begin{example} \textbf{(Pair of transverse associative planes)}\label{eg s-ind Pair of transverse associative planes} Let $\Pi_\pm$ be a pair of associative planes in $\R^7$ with transverse intersection at the origin, that is $\Pi_+\cap \Pi_-=\{0\}$. Set $C_\times:=\Pi_+\cup \Pi_- $ and $e_1:=(1,0\dots,0)\in \R^7=\R\oplus \C^3$. We choose a unit vector $\bn\in \R^7$ orthogonal to both $\Pi_\pm$ so that we have an orientation compatible splitting $$\R^7=\langle \bn\rangle_{\R}\oplus \Pi_+\oplus \Pi_- .$$ Here the orientations of $\Pi_\pm$ are given by the restrictions of the standard $3$-form $\phi_e$ and the orientation of $\langle \bn\rangle_{\R}$ is given by $\bn$. If we choose $-\bn$ instead of $\bn$ in the above, we interchange the role of $\Pi_+$ and $\Pi_-$. 
 Then there exists $B\in G_2$ such that 
$$\bn=Be_1,\ \ \ \ B\Pi_0=\Pi_+,\ \ \ \ B\Pi_\theta=\Pi_-, $$
 where $\Pi_0$ and $\Pi_\theta$ are special Lagrangian planes in $\C^3$ as in \autoref{eg Transverse pair of SL planes}.  
 By \autoref{cor homogeneous kernel} and \autoref{prop indicial laplace} we have	
		$$V_\lambda\cong
		\begin{cases*}
		0& if $\lambda\in[-1,0)\cup(0,1)$\\
		(\R\oplus\R)\oplus (\R^3\oplus\R^3)& if $\lambda=0$\\
		(\R^3\oplus\R^5)\oplus (\R^3\oplus\R^5)& if $\lambda=1$
		\end{cases*}
		$$
		If $H$ is the symmetry group of an associative $3$-plane for the standard action of $G_2$ on $\R^7$, then $H\cong SO(4)$. Hence, $C_\times$ is rigid as in \autoref{def rigid cones}, that is, $d_1=2(\dim{G_2}-\dim{H})=16$,
		and its stability-index from \autoref{def stability index}:
\begin{equation*}
\operatorname{s-ind}(C_{\times})=\frac {d_{-1}}2+\sum_{-1<\lambda\leq 1}d_\lambda-2(\dim{G_2}-\dim{H})-7=0+8+16-16-7=1.\qedhere
	\end{equation*}
			\end{example}
		
\begin{example}\label{eg s-ind Harvey-Lawson cone}
 \textbf{(Harvey-Lawson cone)} Let $C_{HL}$ be the Harvey-Lawson special Lagrangian cone, whose link is the flat Clifford tori in $\C^3$, given in \autoref{eg Harvey-Lawson $T^2$-cone}. Then by \autoref{cor homogeneous kernel} and \autoref{prop indicial laplace}, and \cite[Section 6.3.4, p. 132]{Marshall2002} we have	
		$$V_\lambda\cong
		\begin{cases*}
		\R^2& if $\lambda=-1$\\
		0& if $\lambda\in (-1,0)\cup(0,1)$\\
		\R\oplus \R^6,& if $\lambda=0$\\
		\R^6\oplus\R^6& if $\lambda=1$
		\end{cases*}$$
If $H$ is the symmetry group of $C_{HL}$ for the standard action of $G_2$ on $\R^7$, then $H\cong U(1)^2$. Hence $C_{HL}$ is rigid as in \autoref{def rigid cones}, that is, $d_1=\dim{G_2}-\dim{H}=12$, and its stability-index from \autoref{def stability index}:
\begin{equation*}
	\operatorname{s-ind}(C_{HL})=\frac {d_{-1}}2+\sum_{-1<\lambda\leq 1}d_\lambda-(\dim{G_2}-\dim{H})-7=1+19-12-7=1 \qedhere
	\end{equation*}
\end{example}
The following proposition provides a classification of special Lagrangian cones that attain the minimal stability-index.
\begin{prop}\label{prop SL stability index} Let $C$ be a special Lagrangian cone in $\C^3$ whose link $\Sigma$ is connected and not a totally geodesic $S^2$. Then its stability-index from \autoref{def stability index} and lower stability-index from \autoref{def uplo stability index} satisfy the following:
\begin{equation*}
	\operatorname{s-ind}(C)\geq \operatorname{s-ind}_-(C)\geq \frac {b^1(\Sigma)}2\geq 1
\end{equation*}
with equality if and only if $C$ is $C_{HL}$, the Harvey-Lawson $T^2$-cone (up to special unitary equivalence).
	\end{prop}
	\begin{proof}As $\Sigma$ is not a totally geodesic $S^2$, $C$ is not contained in any hyperplane \cite[Lemma 3.13]{Haskins2004a} in $\C^3$. Moreover the genus of $\Sigma$ is at least $1$ (see \cite[Theorem 2.7]{Haskins2004b}). The space of real linear functions on $\C^3$ induces a $6$ dimensional subspace of the $2$-eigenspace of $\Delta_\Sigma$. Therefore, by \autoref{cor homogeneous kernel} we have (see \autoref{rmk d0 geq 7}) $d_0\geq 7$ and  
		$$\operatorname{s-ind}_-(C)\geq \frac {d_{-1}}2=\frac {b^1(\Sigma)}2\geq1.$$
	In \autoref{eg s-ind Harvey-Lawson cone}, we see $\operatorname{s-ind}(C_{HL})=1$. If $\operatorname{s-ind}(C)=1$, then $b^1(\Sigma)=1$, $d_0=7$ and $d_\lambda=0$ for all $\lambda \in (-1,0)\cup(0,1)$. This implies that the first eigenvalue of $\Delta_\Sigma$ is $2$ with multiplicity $6$. Then by a theorem of \citet[Theorem A]{Haskins2004a} $C$ is $C_{HL}$, the Harvey-Lawson $T^2$-cone (up to special unitary equivalence). 
	\end{proof}
The remainder of this section focuses on the stability-index of a null-torsion holomorphic curve.

	\begin{prop}\label{prop null torsion}
	Let $\Sigma$ be a null-torsion holomorphic curve in $S^6$ and let $C$ be the cone in $\R^7$ with link $\Sigma$. Then the lower stability-index from \autoref{def uplo stability index}:
	$$\operatorname{s-ind}_-(C)>4.$$
	In particular, this inequality holds for any genus zero holomorphic curve in $S^6$ except a totally geodesic $S^2$.
\end{prop}
The proof of this proposition needs the following small preparation.
\begin{definition}Let $\Sigma$ be a $J$-holomorphic curve in $S^6$. The \textbf{Jacobi operator} $\mathcal L_\Sigma:C^\infty(N\Sigma)\to C^\infty(N\Sigma)$ is 
\begin{equation*}\mathcal L_\Sigma=(\nabla^\perp_\Sigma)^*\nabla^\perp_\Sigma+\sum_{i=1}^2(R(f_i,\cdot)f_i)^\perp-\sum_{i,j=1}^2\inp{\operatorname{II}(f_i,f_j)}{\cdot}\operatorname{II}(f_i,f_j).\qedhere
	\end{equation*}
\end{definition}

\begin{prop}\label{prop Jacobi operator}
Let $\Sigma$ be a $J$-holomorphic curve in $S^6$. Then the Jacobi operator $\mathcal L_\Sigma$ satisfies 
$$\mathcal L_\Sigma=(\mathbf D_\Sigma+J)(\mathbf D_\Sigma+2J)=(J\mathbf D_\Sigma)^2-(J\mathbf D_\Sigma)-2.$$
Moreover, $\spec(\mathcal L_\Sigma)=\{\lambda^2+\lambda-2:\lambda\in \cD_C\}$, where $\cD_C$ is defined in \autoref{def homogeneous kernel}, and the $\lambda^2+\lambda-2$ eigenspace of $\mathcal L_\Sigma$ has the following decomposition:
$$E_{\mathcal L_\Sigma}^{\lambda^2+\lambda-2}\cong V_{\lambda-2}\oplus V_{\lambda+1}\cong V_{\lambda}\oplus V_{\lambda+1}.$$
In particular, $E_{\mathcal L_\Sigma}^{-2}\cong V_{-1}\oplus V_0$ and $E_{\mathcal L_\Sigma}^{0}\cong V_{0}\oplus V_1$.
\end{prop}
\begin{proof}
We denote the Jacobi operator for the cone $C$ of $\Sigma$ in $\R^7$ by $\mathcal L_C$. \citet[Theorem 2.8, Appendix 5.3]{Gayet2014} proved that
$$\mathcal L_C=\mathbf D^2_C.$$
Now for all $\nu_\Sigma\in C^\infty(N_{S^6}\Sigma)$, we have $\mathcal L_C\nu_\Sigma=r^{-2}\mathcal L_\Sigma\nu_\Sigma$.  Therefore by \autoref{prop Fueter}\ref{prop Fueter decomposition homogeneous} we conclude that \begin{equation*}\mathbf D_C^2=r^{-2}(\mathbf D_\Sigma+J)(\mathbf D_\Sigma+2J).
	\end{equation*}
This proves the first part. The remaining part follows from parts \ref{prop Fueter homogeneous kernel} and \ref{prop Fueter symmetry} of \autoref{prop Fueter}.
\end{proof}

\begin{proof}[{\normalfont{\textbf{Proof of \autoref{prop null torsion}}}}]Let $\hat J$ be an almost complex structure on $N\Sigma$ defined by the following exact sequence of complex vector bundles (see \autoref{eg Bryant null torsion}):
	$$0\to (L_N,J)\xrightarrow{\alpha} (N\Sigma,\hat J)\xrightarrow{\beta} (L_B,-J) \to0.$$  \citet{Madnick2021} has proved that $\Sigma$ is null-torsion iff $\nabla^\perp \hat J=0$ and in that case $$\mathcal L_\Sigma=2\bar\partial_{\nabla^\perp,\hat J}^*\bar\partial_{\nabla^\perp,\hat J}-2,$$ where $\bar\partial_{\nabla^\perp,\hat J}$ is the Cauchy--Riemann opeartor induced by $\nabla^\perp$ and $\hat J$ on $N\Sigma$.
Moreover,
\begin{equation}\label{eq eigen space dimension jacobi}
	\dim E_{\mathcal L_\Sigma}^{-2}\geq \ind\bar\partial_{\nabla^\perp,\hat J}=2c_1(N\Sigma,\hat J)+2\chi(\Sigma)=-4c_1(L_B)=4b\geq 24.
\end{equation} 
Here $b= \frac {\operatorname{Area}(\Sigma)}{4\pi}\geq 6$. By \autoref{prop Jacobi operator} we see that $\dim E_{\mathcal L_\Sigma}^{-2}=d_{-1}+d_0$  and therefore we have 
	\begin{equation*}\operatorname{s-ind}_-(C)\geq\frac {d_{-1}}{2}+d_0-7\geq  \frac {d_{-1}+d_0}2-7\geq2b-7\geq 5.\qedhere \end{equation*}
\end{proof}
\begin{proof}[{{{\normalfont{\textbf{Proof of \autoref{thm stability index of cones}}}}}}](ii) is proved in \autoref{prop null torsion}, (iii) is computed in \autoref{eg s-ind Pair of transverse associative planes} and \autoref{eg s-ind Harvey-Lawson cone}, (iv) follows from  \autoref{prop SL stability index} and the fact that $d_0-7\geq b^0(\Sigma)-1$. We will prove now (i).
Consider the short exact sequence \autoref{eq LB}. If the genus of $\Sigma$ is $1$ then $K_\Sigma\cong \mathcal O$ and therefore $\mathcal O(Z)\tn L_B=L_N \tn L_B\cong \mathcal O$. Assume $\Sigma$ is not a null-torsion holomorphic curve, then the zero set of $\operatorname{III}$ induces an effective divisor, in particular $\deg L_B\geq [Z]$. Therefore $\deg L_B= [Z]=0$ and hence $L_B\cong \mathcal O$ and $L_N\cong \mathcal O$. This implies that there is a non zero section $\nu_\Sigma$ of $N\Sigma$ such that 	$\bar\partial^N_{\widetilde \nabla}\nu_\Sigma=0$ and hence $d_{-1}\geq 2$. Moreover, considering the long exact sequence of sheaf cohomologies corresponding to \autoref{eq LB} we obtain
$$0\to \C\to H^0(\Sigma, N\Sigma)\to \C\xrightarrow{\operatorname{III}}  \C \to  H^1(\Sigma, N\Sigma)\to \C\to 0.$$ 
Since $\operatorname{III}\neq 0$, we obtain more precisely that, $d_{-1}=2\dim_{\C}H^0(\Sigma, N\Sigma)=2$. \end{proof}
\begin{remark}Let $\Sigma$ be a $J$-holomorphic curve in $S^6$. The \textbf{Morse index} of the Jacobi operator $\mathcal L_\Sigma$ is defined by
 $\operatorname{Ind}\mathcal L_\Sigma:=\sum_{\delta<0}\operatorname{dim}E_{\mathcal L_\Sigma}^\delta.$
 Now \autoref{prop Jacobi operator} implies that 
 $$\operatorname{Ind}\mathcal L_\Sigma=d_{-1}+2\displaystyle\sum_{-1<\lambda<0}d_\lambda+\displaystyle\sum_{0\leq\lambda<1}d_\lambda.$$
   If the genus of $\Sigma$ is $1$ then the proof of \autoref{thm stability index of cones} implies that, $\operatorname{Ind}\mathcal L_\Sigma\geq 9$. If $\Sigma$ is a null-torsion holomorphic curve in $S^6$ then $\mathcal L_\Sigma$ does not have an eigenvalue less than $-2$. Therefore by \autoref{eq eigen space dimension jacobi} the Morse index of $\mathcal L_\Sigma$ satisfies
\begin{equation*}\operatorname{Ind}\mathcal L_\Sigma\geq 4b+\sum_{0<\lambda<1}d_\lambda \geq 24.  \end{equation*}
Comparing the above observations with a generalized Willmore type conjecture made by \citet[Remark 3.6 (1)]{Kusner23} we  conjecture that if an associative cone $C$ in $\R^7$ is not a plane then $ \operatorname{s-ind}(C)\geq 1$ with equality if and only if $C$ is the Harvey-Lawson $T^2$-cone  or a union of two special Lagrangian planes with transverse intersection at the origin. This has been confirmed in \autoref{thm stability index of cones}  for special Lagrangian cones in ~$\C^3$.
\end{remark}

\section{AC and CS associative submanifolds}\label{section deform CS asso}
This section introduces the definitions and examples of asymptotically conical (AC) and conically singular (CS) associative submanifolds. We also study the linear analysis for the Fueter operator, which governs the deformations of AC and CS associative submanifolds. Since these submanifolds are non-compact, we carry out the Fredholm theory in weighted function spaces.

\subsection{Asymptotically conical (AC) associative submanifolds}
Before defining AC associative submanifolds, we introduce the notion of a conical tubular neighbourhood map of a cone in the following definition.
\begin{definition}\label{def conical tubular nbhd}
Let $C$ be a cone in $\R^7$ with link $\Sigma\subset S^6$, that is, $C=\iota((0,\infty)\times \Sigma)$. Let $\Upsilon_\Sigma:V_\Sigma\subset N_{S^6}\Sigma\to U_\Sigma$ be a {tubular neighbourhood map} for $\Sigma$. It induces a \textbf{conical tubular neighbourhood map} of $C$ 
$$\Upsilon_C:V_C\to U_C$$ as follows. Define $V_C\subset NC$ at $(r,\sigma)$ by $r V_\Sigma$ and $\Upsilon_C(u(r,\sigma)):=r\Upsilon_\Sigma(r^{-1}u(r,\sigma))$ and $U_C:=\im \Upsilon_C\subset \R^7$. 
\end{definition}
 Observe that, for each $\epsilon>0$, the scaling map $s_\epsilon:\R^7\to \R^7$, $x\mapsto \epsilon x$ induces an action ${s_\epsilon}_*$ on $NC$. Moreover, $\Upsilon_C$ is equivariant with respect to these actions, that is,
 $$\Upsilon_C({s_\epsilon}_*u)=s_\epsilon \Upsilon_C u.$$ 
We now introduce the notion of AC associative submanifolds, which are associative submanifolds in $\R^7$ that approach a cone at infinity in a controlled manner.
  \begin{definition}\label{def AC associative}Let $C$ be a cone in $\R^7$ with link $\Sigma$, let $\Sigma=\amalg_{i=1}^m\Sigma_i$ be the decomposition into components and let $C_i$ be the cone with link $\Sigma_i$.  
An oriented three dimensional submanifold $L$ of $\R^7$ is called an \textbf{asymptotically conical (AC) submanifold} with cone $C$ and rate $\nu:=(\nu_1,\cdots, \nu_m)$ with $\nu_i<1$ if  there exist
\begin{itemize}
\item  a compact submanifold with boundary $K_L$ of $L$, 
\item a real number ${R_\infty}>1$ and  a diffeomorphism $$\Psi_L:({R_\infty},\infty)\times \Sigma\to L_\infty:=L\setminus K_L\subset \R^7,$$ such that $\Psi_L-\iota$ is a section of the normal bundle $NC$ over $\iota((R_\infty,\infty)\times \Sigma)\subset C$ lying in the tubular neighbourhood $V_C$ of $C$ and 
 		 $${\abs{(\nabla^\perp_{C_i})^k(\Psi_L-\iota)}}=O(r^{\nu_i-k}),$$
 		 for all $k\in \N\cup\{0\}$ as $r\to\infty$. 
		 \end{itemize}
 		  Here $\nabla^\perp_{C}$ is the normal connection on $NC$ induced from Levi-Civita connection on $\R^7$ and $\abs{\cdot}$ is taken with respect to the normal metric on $NC$ and the cone metric on $C$. 
 		 The cone $C$ is called the \textbf{asymptotic cone} of $L$ and $L_\infty$ is the \textbf{end} of $L$.

 	The above $L$ is called an \textbf{asymptotically conical (AC) associative submanifold} if, in addition, it is associative.
	\end{definition}

\begin{remark}Deformation theory of AC associative submanifolds has been studied by \citet{Lotay2011}. If $L$ is an AC associative submanifold, then $C$ is automatically an associative cone in $\R^7$ (see \cite[Proposition 2.15]{Lotay2011}). An AC submanifold of rate $\nu<1$ is also an AC Riemannian manifold with rate $\nu-1$, as well as an AC submanifold of any rate $\nu^\prime$ with $\nu \leq\nu^\prime<1$.  
\end{remark}
The following two definitions will be useful later in the context of desingularization but may be skipped for now. They explain how, given an AC associative submanifold, one can define a corresponding end-conical submanifold, over which the AC submanifold can be expressed as the graph of a decaying section inside an appropriate end-conical tubular neighbourhood.

\begin{definition}[End conical submanifold]\label{def LC}
Let $\rho:(0,\infty)\to [0,1]$ be a smooth function such that 
\begin{equation}\label{eq:cutoff function}
\rho(r)=
\begin{cases*}	
1, r\leq 1\\
0, r\geq 2.	
\end{cases*}
\end{equation}
For every $s>0$, $\rho_s:(0,\infty)\to[0,1]$ is defined by $\rho_s(r):=\rho(s^{-1}r)$. Let $L$ be an AC submanifold with cone $C$ as in \autoref{def AC associative}. We define an end conical submanifold $L_C$, which is diffeomorphic to $L$:
 $$L_C:=K_L\cup\Upsilon_C(\rho_{_{R_\infty}}(\Psi_L-\iota)).$$ 
Set $C_{\infty}:=\iota((2{R_\infty},\infty)\times \Sigma)\subset L_C$ and $K_{L_C}:=L_C\setminus C_\infty$.
\end{definition}

\begin{definition}[End conical tubular neighbourhood map]\label{def section beta for L}
Let $L$ be an AC submanifold with cone $C$ as in \autoref{def AC associative} and $\Upsilon_{C}$ be as in \autoref{def conical tubular nbhd}. We say a tubular neighbourhhood map:
	$$\Upsilon_{L_C}:V_{L_C}\to U_{L_C}$$
of ${L_C}$ is end conical, if $V_{L_C}$ and $\Upsilon_{L_C}$ are chosen so that they agree with $V_C$ and $\Upsilon_{C}$, respectively on $C_\infty$.
 
	There is a section $\beta$ in $V_{L_C}\subset N{L_C}$ which is zero on $K_L$ and $\Upsilon_{L_C}(\Gamma_\beta)$ is $L$, and 
$${\abs{(\nabla_{C_i}^\perp)^k\beta}}=O(r^{\nu_i-k})$$ for all $k\in \N\cup\{0\}$ as $r\to\infty$, where $\nabla^\perp_{C}$ is the normal connection on $N{L_C}$ induced from Levi-Civita connection on $\R^7$.	
\end{definition}

There are many examples of AC special Lagrangians in $\C^3$, see \cites{Joyce2001}{Joyce2002}, but we mention only two types of AC special Lagrangians in \autoref{eg Lawlor neck} and \autoref{eg Harvey-Lawson special Lagrangians} which are important for this article. For examples of AC associatives which are not AC special Lagrangians, see \cites[Section 7]{Lotay2011}{Lotay2007a}.

\begin{example}(\textbf{Lawlor neck}, \cites{Lawlor1989}[Section 4.1]{Joyce2016}[pg. 139-143]{Harvey1990}\label{eg Lawlor neck} Let $\Pi_0=\R^3$ and $\Pi_\theta$ be a pair of transverse special Lagrangian planes in $\C^3$ as in \autoref{eg Transverse pair of SL planes}, that is, \[\Pi_0=\R^3\ \ \ \   \text{and}\ \ \ \ \  \Pi_\theta:=\diag(e^{i\theta_1},e^{i\theta_2},e^{i\theta_3})\cdot\R^3,\]
where $\theta_1,\theta_2,\theta_3\in (0,\pi)$ satisfy $\theta_1\leq\theta_2\leq\theta_3$ and $\theta_1+\theta_2+\theta_3=\pi$.
 Define
$$\iota^+:\R^+\times S^2\to \Pi_\theta,\ \ \iota^+(r,\sigma)=r\diag({e^{i\theta_1},e^{i\theta_2},e^{i\theta_3}})\cdot  \sigma$$
and $\iota^-:\R^+\times S^2\to \Pi_0,\ \ \iota^-(r, \sigma)=r \sigma.$
For any $A>0$ there exists a unique triple of positive real numbers $(a_1,a_2,a_3)$ such that $A=\frac{4\pi}{3\cdot \sqrt{a_1a_2a_3}}$ and $\theta_k=\theta_k(\infty)$, where
$$\theta_k(y)=a_k\int_{-\infty}^{y}\frac{dx}{(1+a_k x^2)\sqrt{P(x)}}, \ \ \ y\in \R,\  k=1,2,3, \ \ P(x):=\frac{{\prod}_{j=1}^3(1+a_j x^2)-1}{x^2}.$$ 
The \textbf{Lawlor neck} $L_{\theta,A}$ is defined as follows:
$$L_{\theta,A}:=\{(z_1(y)\sigma_1,z_2(y)\sigma_2,z_3(y)\sigma_3)\in \C^3: y,\sigma_1,\sigma_2,\sigma_3 \in \R,\ \sigma_1^2+\sigma_2^2+\sigma_3^2=1\},$$
where for $k=1,2,3$, $z_k(y):=e^{i\theta_k(y)}\sqrt{a_k^{-1}+y^2}$.
The Lawlor neck $L_{\theta,A}$ is a smooth special Lagrangian submanifold \cite[Theorem 7.78]{Harvey1990} diffeomorphic to $\R\times S^2$. 
The map 
 \begin{alignat*}{2}
\Phi_{L_{\theta,A}}:\ & \R\times S^2\longrightarrow L_{\theta,A}\\
&(y,(\sigma_1,\sigma_2,\sigma_3))\mapsto (z_1(y)\sigma_1,z_2(y)\sigma_2,z_3(y)\sigma_3)
  \end{alignat*}
 is a diffeomorphism. Define $\Phi^\pm_{L_{\theta,A}}:\R^+\times S^2\longrightarrow L_{\theta,A}$ by $\Phi^\pm_{L_{\theta,A}}(r,\sigma)=\Phi_{L_{\theta,A}}(\pm r,\sigma)$. For every $\epsilon>0$, we observe that the rescaled Lawlor neck satisfies: $$\epsilon L_{\theta,A}=L_{\theta,\epsilon^{3}A}.$$ Therefore, we will always use "Lawlor neck" to mean $L_{\theta,A}$ with the normalization $A=1$, if not mentioned otherwise.
 
We define $\xi^-\in C^\infty({N\Pi_0}_{|\Pi_0\setminus\{0\}})$ and $\xi^+:=\diag({e^{i\theta_1},e^{i\theta_2},e^{i\theta_3}})\cdot \xi^- \in C^\infty({N\Pi_\theta}_{|\Pi_\theta\setminus\{0\}})$ as follows:
\begin{equation}\label{eq xi Lawlor neck}
\xi^-(x_1,x_2,x_3):=r^{-3}(ix_1,ix_2,ix_3),\  \text{where}\  r=\sqrt{x_1^2+x_2^2+x_3^2}.
\end{equation}
Observe that, $\bD_{\Pi_\pm}\xi^\pm=0$ and $\xi^\pm\in V_{-2}$. The Lawlor neck $L_{\theta,A}$ is an AC special Lagrangian submanifold asymptotic to $\Pi_0\cup \Pi_\theta$ with rate $\nu=-2$. Moreover, the asymptotic normal sections can be written as
 \begin{equation}\label{eq AC Lawlor neck}
	\Psi^\pm_{L_{\theta,1}}-\iota^\pm=\xi^\pm+ O(r^{-4}).
\end{equation} 
Here $\Psi^\pm_{L_{\theta,1}}$ can be defined (up to normalizations) as follows: \begin{equation*}
 	\Psi^-_{L_{\theta,1}}(\Re \Phi^-_{L_{\theta,1}})=\Phi^-_{L_{\theta,1}}\  \text{and} \ \Psi^+_{L_{\theta,1}}(\Re \diag({e^{-i\theta_1},e^{-i\theta_2},e^{-i\theta_3}})\Phi^+_{L_{\theta,1}})=\Phi^+_{L_{\theta,1}}.\qedhere
 \end{equation*}

\end{example}
\begin{example}\label{eg Harvey-Lawson special Lagrangians}(\textbf{Harvey-Lawson AC special Lagrangians}, \cites[Theorem 3.1]{Harvey1982}[Section 5.1]{Joyce2016}) Let $C:=C_{HL}$ be the Harvey-Lawson $T^2$-cone as in \autoref{eg Harvey-Lawson $T^2$-cone} and let $\iota$ be the inclusion map in $\R^7$. For any positive real number $a>0$, the \textbf{Harvey-Lawson special Lagrangian} $3$-folds are defined by
	\begin{align*}
	&L^1_{a}:=\{(z_1,z_2,z_3)\in \C^3:\abs{z_1}^2-a=\abs{z_2}^2=\abs{z_3}^2, z_1z_2z_3\in [0,\infty)\},\\	
	&L^2_{a}:=\{(z_1,z_2,z_3)\in \C^3:\abs{z_1}^2=\abs{z_2}^2-a=\abs{z_3}^2, z_1z_2z_3\in [0,\infty)\},\\
	&L^3_{a}:=\{(z_1,z_2,z_3)\in \C^3:\abs{z_1}^2=\abs{z_2}^2=\abs{z_3}^2-a, z_1z_2z_3\in [0,\infty)\}.
	\end{align*}
	Each $L^k_{a}$, $k=1,2,3$ is a smooth special Lagrangian in $\C^3$ \cite[Section III.3.A]{Harvey1982} diffeomorphic to $S^1\times \C$. Define the diffeomorphims
\begin{align*}
 	&\Phi_{L^1_{a}}: \R^+\times T^2\longrightarrow L^1_{a},\ \ \ \ \ \ \ 
 	(r,e^{i\theta_1},e^{i\theta_2})\mapsto (e^{i\theta_1}\sqrt{r^2+a},re^{i\theta_2},re^{-i(\theta_1+\theta_2)}),\\
 	&\Phi_{L^2_{a}}: \R^+\times T^2\longrightarrow L^1_{a},\ \ \ \ \ \ \ 
 	(r,e^{i\theta_1},e^{i\theta_2})\mapsto (re^{-i(\theta_1+\theta_2)},e^{i\theta_1}\sqrt{r^2+a},re^{i\theta_2}),\\
 	&\Phi_{L^3_{a}}:\R^+\times T^2\longrightarrow L^1_{a},\ \ \ \ \ \ \ 
 	(r,e^{i\theta_1},e^{i\theta_2})\mapsto (re^{i\theta_2},re^{-i(\theta_1+\theta_2)},e^{i\theta_1}\sqrt{r^2+a}).
\end{align*}
 For every $\epsilon>0$, we have
	$$\epsilon L^k_{a}=L^k_{\epsilon^2 a}.$$
Therefore, we always assume $L^k_{a}$ with the normalization $a=1$, if not mentioned otherwise. Each $L^k_{a}$ is an AC special Lagrangian submanifold asymptotic to $C_{HL}$ with rate $\nu=-1$. Moreover, the asymptotic normal sections can be written as
 \begin{equation}\label{eq AC HL}
	\Psi_{L^k_{1}}-\iota=\xi_k+ O(r^{-2}),
\end{equation} 
where $\xi_k\in V_{-1}$ and $\xi_3=-\xi_1-\xi_2$.
\end{example}

\subsection{Conically singular (CS) associative submanifolds} We define conically singular (CS) associative submanifolds with singularity at $m$ points, modeled on associative cones in $\R^7$. Before that we need a preferred choice of coordinate system at each of these $m$ points, which we define first.
	\begin{definition}
		Let $(Y,\phi)$ be an almost $G_2$-manifold (see \autoref{def almost G2}) and let $p$ be a point in $Y$. A \textbf{$G_2$-framing} at $p$ is a linear isomorphism $\upsilon:\R^7\to T_pY$ such that $\upsilon^*(\phi(p))=\phi_e$, where $\phi_e$ is the standard $G_2$ structure on $\R^7$. A $G_2$-\textbf{coordinate system} at $p$ with $G_2$-framing $\upsilon$ is a diffeomorphism 
		$$\Upsilon:B_R(0)\subset \R^7\to U$$
	for some constant $0<R<1$ and some open set $U$ containing $p$ in $Y$ satisfying $\Upsilon(0)=p$ and $d\Upsilon_{0}=\upsilon$. 
	 
	 Two $G_2$-{coordinate systems} $\Upsilon_1$ and $\Upsilon_2$ at $p$ are called \textbf{equivalent} if they have the same $G_2$-framing at $p$.
	 \end{definition} 
 
\begin{definition}\label{def CS asso}Let $P$ be a compact subset of an almost $G_2$-manifold $Y$ and $\sing (P):=\{p_1,\dots,p_m\}$ be a finite set points in $P$ such that $\hat P:=P\setminus \sing (P)$ is a smooth three dimensional oriented submanifold of $Y$. We call $\hat P$ the \textbf{smooth part} of $P$ and $\sing (P)$ the \textbf{set of singular points} of $P$.  Let 
		$$\Upsilon^i:B_{R}(0)\to U\subset Y$$ be a $G_2$-{coordinate system} with $G_2$-framing $\upsilon_i$ at $p_i$,  $i=1,\dots,m$. Let $\{C_1,\dots,C_m\}$ be a set of cones in $\R^7$. Denote the link of $C_i$ by $\Sigma_i$ and the inclusion map into $\R^7$ by $\iota_i$.
\begin{enumerate}
 \item $P$ is said to be a \textbf{conically singular (CS) submanifold} with singularities at $p_i$ modeled on cones $C_i$ (with respect to the $G_2$-framing $\upsilon_i$) and rates $\mu_i$ with $1<\mu_i\leq 2$ if there exist
 \begin{itemize}
  \item a compact submanifold with boundary $K_P$ of $P$,
 \item $\epsilon_0>0$ with $2\epsilon_0<R$ and smooth embeddings $$\Psi^i_P:(0,2\epsilon_0)\times \Sigma_i\to B_{R}(0), \ i=1,2,...,m$$
	with $\bigcup_{i=1}^m\Upsilon^i\circ \Psi^i_P:(0,2\epsilon_0)\times \Sigma_i\to Y$ is a diffeomorphism onto $\hat P\setminus K_P$, such that $\Psi_P^i-\iota_i$ is a smooth section of the normal bundle $NC_i$ over $\iota((0,2\epsilon_0)\times \Sigma_i)$ which lies in $V_{C_i}$ and
		 $$\abs{(\nabla^\perp_{C_i})^k(\Psi^i_P-\iota_i)}=O(r^{\mu_i-k})$$
		 for all $k\in \N\cup\{0\}$ as $r\to0$, $i=1,2...,m$. 
		 \end{itemize}
		  Here $\nabla^\perp_{C_i}$ is the normal connection on $NC_i$ induced from Levi-Civita connection on $\R^7$ and $\abs{\cdot}$ is taken with respect to the normal metric on $NC_i$ and cone metric on $C_i$.
 \item The cone $\tilde C_i:=\upsilon_i(C_i)\subset T_{p_i}Y$ is said to be the \textbf{tangent cone} of $P$ at $p_i$.

\item $P$ is said to be a \textbf{conically singular (CS) associative submanifold} if it is conically singular submanifold as above and $\hat P$ is an associative submanifold of $(Y,\phi)$. \qedhere
\end{enumerate}	
	\end{definition}

\begin{remark}
We make some remarks about the above definition.
\begin{enumerate}[(i)] 
\item A CS submanifold with rates $\mu_i$ is also a CS submanifold with all rates $\mu_i^\prime$ such that $1<\mu^\prime_i\leq \mu_i\leq 2$ for all $i=1,2,..,m$. 
\item A CS submanifold with rates $\mu_i$ is a conically singular Riemannian manifold with rates $\mu_i-1$ for the metric $g_\phi$ induced by $\phi$. Indeed 
\[\abs{(\nabla^\perp_{C_i})^k\big((\Upsilon^i\circ\Psi_P^i)^*g_\phi-g_{C_i}\big)}\lesssim \sum_{j=0}^{k+1}\abs{(\nabla^\perp_{C_i})^j(\Psi^i_P-\iota_i)}.\]
		\item $\mu_i>1$ also implies that if $P$ is a CS associative submanifold as above, then each cone $C_i$ is also an associative cone in $\R^7$. Indeed, the associator on $C_i$
\[ \abs{[\cdot, \cdot,\cdot]_{|C_i}}=O(r^{\mu_i-1}),\ \text{as}\  r\to 0.\]
But $\abs{[\cdot, \cdot,\cdot]_{|C_i}}$ is dilation invariant and therefore $[\cdot, \cdot,\cdot]_{|C_i}=0$. 
\item Since $\Psi_P^i-\iota_i$ is a section of the normal bundle $NC_i$ which represents $P$ over the end, $\Psi^i_P$ is uniquely determined by the constant $\epsilon_0$, the normal bundle $NC_i$ and the $G_2$-coordinate system $\Upsilon^i$ at $p_i$.
\item The condition $\mu_i \leq 2$ guarantees that the above definition of a CS submanifold is independent of the choice of $G_2$-coordinate systems $\Upsilon^i$ at $p_i$ within an equivalence class; that is, it depends only on the $G_2$-framings $\upsilon_i$ at those points. However, observe that for any other $G_2$-framing $\upsilon_i'$ at $p_i$, there exists $A_i \in G_2$ such that $\upsilon_i = \upsilon_i' \circ A_i$. Therefore, the same CS submanifold is also a CS submanifold modeled on the cones $A_i(C_i) \subset \R^7$.  As a consequence, the tangent cones are independent of the choice of $G_2$-framings. Whenever we say that a CS submanifold is modeled on certain cones in $\R^7$, an implicit framing is there. \qedhere
		\end{enumerate}
\end{remark}

 \subsection{Linear analysis on CS and AC associative submanifolds}
 This subsection discusses the linear analysis of the deformation operator, which will be essential for the deformation theory and desingularization results presented later. By considering appropriate weighted function spaces of sections of vector bundles over CS and AC submanifolds (where the weights represent decay rates of those sections) we obtain a  Fredholm theory for CS and AC elliptic operators if the weights avoid the wall of critical weights or rates. The index of these operators changes when weights cross the wall of critical rates. This theory originally appeared in Lockhart-McOwen \cite{Lockhart1985}. A very good exposition can be found in \cite{Marshall2002}, \cite{Karigiannis_2020}.

\begin{notation}For a conically singular manifold $P$ we will denote by $NP$ the normal bundle of $\hat P=P\setminus\operatorname{sing}(P)$. 
\end{notation}

\subsubsection{Weighted function spaces}\label{Weighted Banach spaces} 
Let $P$ be a CS submanifold as in \autoref{def CS asso} and $L$ be an AC submanifold as in \autoref{def AC associative} asymptotic to $C_i$, $i=1,\dots,m$. To treat the AC and CS together introduce the following notation.
 
	\begin{notation}\label{notation M=P,M=L}
		We denote 
	\begin{equation*}
		M:=
		\begin{cases*}
		P  & if $P$ is a CS submanifold \\
		L & if $L$ is a AC submanifold.	
		\end{cases*}
	\end{equation*}
	and 
\begin{equation*}
		\eta:=
		\begin{cases*}
		\mu  & if $P$ is a CS submanifold with rate $\mu$ \\
		\nu & if $L$ is a AC submanifold  with rate $\nu$.	
		\end{cases*}
	\qedhere\end{equation*} 
		\end{notation}
		\begin{notation}
We would like to define weighted Sobolev and Hölder spaces with rate $\lambda=(\lambda_1,\lambda_2,...,\lambda_m)\in \R^m$. We say $\lambda<, =,>\lambda^\prime$ if for each $i=1,2,..,m$ we have $\lambda_i<,=,>\lambda_i^\prime$, respectively.  For any $s\in \R$, $\lambda+s:=(\lambda_1+s,\lambda_2+s,...,\lambda_m+s)$ and set $\abs{\lambda}:=\sum_{i=1}^m\abs{\lambda_i}.$
		\end{notation}
 \begin{definition}[Weighted function spaces]For each $\lambda=(\lambda_1,\lambda_2,...,\lambda_m)$ $\in \R^m$, a \textbf{weight function} on $\hat P$, 
	$w_{P,\lambda}:\hat P\to(0,\infty)$, is a smooth function on $\hat P$ such that if $x=\Upsilon\circ \Psi_P(r,\sigma)$ with $(r,\sigma)\in (0,2\epsilon_0)\times \Sigma_i$ then
		$$w_{P,\lambda}(x)= r^{-\lambda_i}.$$
	The weight function on $L$, 
 $w_{L,\lambda}:L\to(0,\infty)$ is defined by
	$$w_{L,\lambda}(x):=(1+\abs{x})^{-\lambda_i}.$$
		Let $\lambda\in\R^m$ and  $k\geq0$,\ $p\geq1$,\ $\gamma\in (0,1)$. For a continuous section $u$ of $NM$, we define the \textbf{weighted} $L^\infty$ \textbf{norm} and the \textbf{weighted Hölder semi-norm} respectively by 
	$$\norm{u}_{L^\infty_{M,\lambda}}:=\norm{w_{M,\lambda}u}_{L^\infty({NM})},\ \ [u]_{C_{M,\lambda}^{0,\gamma}}:=[w_{M,\lambda-\gamma}u]_{C^{0,\gamma}({NM})}.$$ 
	For a continuous section $u$ of $NM$ with $k$ continuous derivatives, we define the \textbf{weighted $C^k$ norm} and the \textbf{weighted Hölder norm} respectively by
	$$\norm{u}_{C^{k}_{M,\lambda}}:=\sum_{j=0}^{k}\norm{(\nabla_{{M}}^\perp)^ju}_{L^\infty_{M,\lambda-j}}, \ \ \norm{u}_{C^{k,\gamma}_{M,\lambda}}:=\norm{u}_{C^{k}_{M,\lambda}}+[(\nabla_{{M}}^\perp)^ku]_{C^{0,\gamma}_{M,\lambda-k}},$$ 
		and the \textbf{weighted Sobolev norm} by	
	$$\norm{u}_{W^{k,p}_{M,\lambda}}:=\Big(\sum_{j=0}^{k}\int_{M}\abs{w_{M,\lambda-j}(\nabla_{M}^\perp)^ju}^p w_{M,3}dV_{M}\Big)^{\frac1p}.$$	
	Here, the connection $\nabla_{M}^\perp$ on $NM$ is the projection of the Levi-Civita connection induced by $g_\phi$ for the decomposition $TY_{|M}=TM\oplus NM$  and the $\abs{\cdot}$ is with respect to the metric $g_\phi$. Set $$L^{p}_{M,\lambda}:=W^{0,p}_{M,\lambda}.$$
	
	We define the \textbf{weighted Hölder space} $C^{k,\gamma}_{M,\lambda}$ to be the space of continuous sections of $NM$ with $k$ continuous derivatives and finite {weighted Hölder norm} $\norm{\cdot}_{C^{k,\gamma}_{M,\lambda}}$. This is a Banach space but not separable.
	
	We define the \textbf{weighted Sobolev space} $W^{k,p}_{M,\lambda}$, the \textbf{weighted $C^k$-space $C^{k}_{M,\lambda}$}, the \textbf{weighted $L^p$-space $L^{p}_{M,\lambda}$} and the \textbf{weighted $L^\infty$-space $L^{\infty}_{M,\lambda}$} to be the completion of the space of compactly supported smooth sections  of $NM$, namely $C_c^\infty(NM)$ with respect to  weighted Sobolev norm $\norm{\cdot}_{W^{k,p}_{M,\lambda}}$, the weighted $C^k$-norm $\norm{\cdot}_{C^{k}_{M,\lambda}}$, the weighted $L^p$-norm $\norm{\cdot}_{L^{p}_{M,\lambda}}$ and the weighted $L^\infty$-norm $\norm{\cdot}_{L^{\infty}_{M,\lambda}}$ respectively. These are all separable Banach spaces.
 Moreover, we define the \textbf{weighted $C^\infty$-space $C^{\infty}_{M,\lambda}$} by
	\begin{equation}\label{eq decay weighted smooth function space}
	C^{\infty}_{M,\lambda}:=\bigcap_{k=0}^\infty C^{k}_{M,\lambda}.
	\qedhere 
	\end{equation}	
	\end{definition}
\begin{remark}  
 The spaces $W^{k,2}_{M,\lambda}$ are all Hilbert spaces with inner product coming from the polarization of the norm. We also note that 
 \begin{equation*}L^{p}(NM)=L^{p}_{M,-\frac 3p}.\qedhere 
	\end{equation*}		
\end{remark}
The $L^2$-inner product gives rise to the following useful result.
\begin{prop}\label{prop dual weighted space}Let $k,l\geq0$, $p,q>1$ with $\frac1p+\frac1q=1$ and $\lambda, \lambda_1,\lambda_2\in \R^m$. Then the $L^2$-inner product  map $$W^{k,p}_{M,\lambda_1}\times W^{l,q}_{M,\lambda_2}\xrightarrow{\inp{\cdot}{\cdot}_{L^2}}\R$$ 
is continuous  provided $M= P$ (CS) with $ \lambda_1+\lambda_2\geq-3$, or $M=L$ (AC) with $ \lambda_1+\lambda_2\leq-3$. Moreover, this $L^2$-inner product yields a Banach space isomorphism 
$$(L^{p}_{M,\lambda})^*\cong L^{q}_{M,-3-\lambda}$$
\end{prop}
\begin{proof}
 For $u\in L^{p}_{M,\lambda_1}$ and $v\in L^{q}_{M,\lambda_2}$ we consider 
 \begin{align*}
 \inp{u}{v}_{L^2}&=\int_M(w_{M,\lambda_1}u w_{M,\frac3p})(w_{M,\lambda_2}v  w_{M,\frac3q})w_{M,-\lambda_1-\lambda_2-3}\operatorname{Vol}_{M}\\
 &\lesssim \norm{u}_{W^{k,p}_{M,\lambda_1}}\norm{v}_{W^{k,q}_{M,\lambda_2}},
  \end{align*}
  provided $M= P$ (CS) with $ \lambda_1+\lambda_2\geq-3$, or $M=L$ (AC) with $ \lambda_1+\lambda_2\leq-3$. This proves that  the $L^2$-inner product map is continuous. The isomorphism stated in the proposition follows from the standard fact that $\inp{\cdot}{\cdot}_{L^2}:L^{p}\times L^{q}\rightarrow\R$ is a dual pairing. 
\end{proof}
The following result is a weighted embedding and compactness theorem that will be used in this article; see \cite[Theorems 4.17 and 4.18]{Marshall2002}.
\begin{prop}\label{prop Sobolev embedding and cptness}
	Let $\lambda,\lambda_1,\lambda_2 \in \R^m$, $\ k,l\in \mathbf N\cup\{0\}$, $ p,q\geq 1$, $\alpha,\beta\in(0,1).$ Then 
	\begin{enumerate}[i)]
		\item  If $k\geq l$ and $k-\frac np\geq l-\frac nq$, the inclusion $W^{k,p}_{M,\lambda_1}\hookrightarrow W^{l,q}_{M,\lambda_2}$ is a continuous embedding, provided for $M=L$ (AC) either $p\leq q,\lambda_1\leq\lambda_2$ or $p>q, \lambda_1<\lambda_2$, and for $M= P$ (CS) either $p\leq q,\lambda_1\geq\lambda_2$ or $p>q, \lambda_1>\lambda_2$ .
		\item If $k+\alpha\geq l+\beta$,  the inclusion ${C^{k,\alpha}_{M,\lambda_1}}\hookrightarrow {C^{l,\beta}_{M,\lambda_2}}$ is a continuous embedding, provided $\lambda_1\leq\lambda_2$ for $M=L$ (AC) and  $\lambda_1\geq\lambda_2$ for $M= P$ (CS) .
		\item If $k-\frac np\geq l+\alpha$,  the inclusion $W^{k,p}_{M,\lambda}\hookrightarrow {C^{l,\alpha}_{M,\lambda}}$ is a continuous embedding.
		\item All of the above embeddings are compact provided all the inequalities in the corresponding hypotheses are strict inequalities.
	\end{enumerate}
\end{prop}

\subsubsection{AC and CS elliptic operators on AC and CS associatives}Let $P$ be a CS associative submanifold as in \autoref{def CS asso} and $L$ be an AC associative submanifold as in \autoref{def AC associative} asymptotic to $C$. We continue to use $M$ to denote either $P$ or $L$, and $\eta$ as $\mu$ for CS and $\nu$ for AC submanifolds. For any associative cone $C$ in $\R^7$ with link $\Sigma$, the Fueter operator $\mathbf D_C$ is a {conical operator} (see  \cite[Section 4.3.2]{Marshall2002} for precise definition), that is after the identification of $\R^7\setminus \{0\}$ with the cylinder $\R \times S^6$ by substituting $r=e^t$ we obtain that (see \autoref{prop Fueter})
\begin{equation}\label{eq conical cylindrical DC}
	r^2\mathbf D_Cr^{-1}=Jr\partial_r+(\mathbf D_\Sigma+J)=J\partial_t+(\mathbf D_\Sigma+J)
\end{equation}
is a translation invariant operator.
	
\begin{definition}\label{def CS/AC elliptic operator}
	Let $\bD_{M}:C_c^\infty(NM)\to C_c^\infty(NM)$ be a first order, formally self-adjoint elliptic operator. It is called \textbf{asymptotically conical (AC) uniformly elliptic} operator for $M=L$ (AC) and \textbf{conically singular (CS) uniformly elliptic} operator for $M= P$ (CS) respectively, asymptotic to the conical operators $\mathbf D_{C_i}$ over the ends of $M$ if the operator $r^2\bD_{M}r^{-1}$ is an {asymptotically translation invariant uniformly elliptic} operator asymptotic to $r^2\mathbf D_{C_i}r^{-1}$ (see \cite[Section 4.3.2]{Marshall2002}) and after the identifications by canonical bundle isomorphisms
	$$\bD_{M}=\bD_{C_i}+O(r^{\eta_i-1}).$$ 
	The operator $\bD_{M}$ has canonical extensions to weighted function spaces and we denote them as follows:
\begin{equation*}\bD_{M,\lambda}^{k,p}:W^{k+1,p}_{M,\lambda}\to W^{k,p}_{M,\lambda-1}, \ \ \ \bD_{M,\lambda}^{k,\gamma}:C^{k+1,\gamma}_{M,\lambda}\to C^{k,\gamma}_{M,\lambda-1}.\qedhere\end{equation*}
 \end{definition}

The following proposition is about elliptic regularity statements for CS or AC uniformly elliptic operator $\bD_{M}$ \cite[Theorem 4.6]{Marshall2002}. 
\begin{prop} \label{prop elliptic regularity} Suppose $\lambda\in \R^m$, $k\geq0$, $p>1$, $\gamma\in(0,1)$. Let, $u,v\in L^1_{\operatorname{loc}}$ be two locally integrable sections of $NM$ such that $u$ is a weak solution of the equation $\bD_{P}u=v$. 
	\begin{enumerate}[i)]
		\item If $v\in C^{k,\gamma}_{M,\lambda-1}$, then $u\in C^{k+1,\gamma}_{M,\lambda}$ is a strong solution and there exists a constant $C>0$ such that 		
		$$\norm{u}_{C^{k+1,\gamma}_{M,\lambda}}\leq C \Big(\norm{\bD_{M,\lambda}^{k,\gamma}u}_{C^{k,\gamma}_{M,\lambda-1}}+\norm{u}_{L^\infty_{M,\lambda}}\Big).$$
		\item   If $v\in W^{k,p}_{M,\lambda-1}$, then $u\in W^{k+1,p}_{M,\lambda}$ is a strong solution and there exists a constant $C>0$ such that
		$$\norm{u}_{W^{k+1,p}_{M,\lambda}}\leq C \Big(\norm{\bD_{M,\lambda}^{k,p}u}_{W^{k,p}_{M,\lambda-1}}+\norm{u}_{L^\infty_{M,\lambda}}\Big).$$
	\end{enumerate}
\end{prop}
The following proposition identifies the weights for which the elliptic operator $D_M$ are adjoints of each other with the weighted function spaces.
\begin{prop}\label{prop integration by parts}Suppose $k\geq0$, $p,q>1$ with $\frac1p+\frac1q=1$ and $\lambda_1,\lambda_2\in \R^m$. For all $u\in W^{k+1,p}_{M,\lambda_1},v\in W^{k+1,q}_{M,\lambda_2}$ we have $\inp{\bD_{M,\lambda_1}^{k,p}u}{v}_{L^2}=\inp{u}{\bD_{M,\lambda_2}^{k,q}v}_{L^2}$, provided $M= P$ (CS) with $ \lambda_1+\lambda_2\geq-2$, or $M=L$ (AC) with $ \lambda_1+\lambda_2\leq-2$. 
\end{prop}
\begin{proof}
The result is true for $u, v \in C_c^\infty(NM)$ and therefore the general statement follows from \autoref{prop dual weighted space}. 
\end{proof}
In the following, we discuss the Fredholm theory for the operator $D_M $. It turns out that the Fredholm property depends on the choice of weights, and in particular, it requires that the weight does not lie on the wall of critical rates, which we define below.

\begin{definition}[Wall of critical rates]Let $\{C_1,\dots,C_m\}$ be a set of associative cones. Set $C:=(C_i)_{i=1}^m$. The set of critical rates $\mathcal D_C$ is defined by
	$$\mathcal D_C:=\{(\lambda_1,\lambda_2,..,\lambda_m)\in \R^m: \lambda_i\in \mathcal D_{C_i} \ \text{for some}\ i\},$$ 
	where $\mathcal D_{C_i}$ is the set of all indicial roots of $\mathbf D_{C_i}$ defined in \autoref{def homogeneous kernel}. We call $\mathcal D_C$ the \textbf{wall of critical rates} in $\R^m$.   We define \begin{equation*}
	 	V_\lambda:=\bigoplus_{i=1}^m V_{\lambda_i}, \ \ d_{\lambda_i}:=\dim 
	V_{\lambda_i} \ \text{and} \ \ \ d_\lambda:=\sum_{i=1}^md_{\lambda_i}.
	 \end{equation*} 
	where $V_{\lambda_i}$ are defined in \autoref{def homogeneous kernel}.
	\end{definition}
We also need the weighted function spaces over the cone, which we now define.
\begin{definition}Let $C$ be a connected associative cone. We can define the Banach spaces $W^{k,p}_{C,\lambda}$, $C^{k,\gamma}_{C,\lambda}$ and all others over $C$ analogously to \autoref{Weighted Banach spaces}, replacing $M$ by $C$, $NM$ by $NC$ and  the weight function $w_{M,\lambda}$ by $w_{C,\lambda}:C\to \R$ where $w_{C,\lambda}(r,\sigma)= r^{-\lambda}$, $\lambda\in \R$.  
	\end{definition}
 The following key lemma, which is crucial to the Fredholm theory, follows from \cite[Lemma 3.1, Proposition 3.4, Section 3.3.1]{Donaldson2002} and \cite[Theorem 5.1]{Mazya1978}, using the observation made in \autoref{eq conical cylindrical DC} above.
 \begin{lemma}\label{lem fourier series Donaldson}Let $C$ be a connected associative cone. The conical Fueter operator $\mathbf D_C:C^{k+1,\gamma}_{C,\lambda}\to C^{k,\gamma}_{C,\lambda-1}$ is invertible  if and only if $\lambda\in \R\setminus \mathcal D_C$. 
 Moreover any element $u\in \ker \mathbf D_C$ has a $L^2$-orthogonal decomposition
$$u=\sum_{\lambda\in \mathcal D_{C}}r^{\lambda-1} u_{\Sigma,\lambda}.$$
where $u_{\Sigma,\lambda}$ are $\lambda$-eigensections of $J\mathbf D_{\Sigma}-\id$, that is $r^{\lambda-1} u_{\Sigma,\lambda}\in V_{\lambda}$.
 \end{lemma}
We obtain the following proposition.
\begin{prop}\label{prop Fredholm Donaldson}Let $\lambda\in \R^m$, $k\geq0$, $p>1$, $\gamma\in(0,1)$. Then  $\bD_{M,\lambda}^{k,p}$ and $\bD_{M,\lambda}^{k,\gamma}$ are Fredholm for all $\lambda\in \R^m\setminus\mathcal D_C$. Moreover,  for all $\lambda\in \R^m$, 
$\operatorname{Ker}\bD_{M,\lambda}^{k,p}=\operatorname{Ker}\bD_{M,\lambda}^{k,\gamma}$ is finite dimensional, independent of $k$, $p$ and $\gamma$. 
\end{prop}
\begin{proof} The operators $\bD_{M,\lambda}^{k,p}$ and $\bD_{M,\lambda}^{k,\gamma}$ are Fredholm with weight $\lambda\in \R^m\setminus\mathcal D_C$ follows from \cite[Proposition 3.6, Section 3.3.1]{Donaldson2002} or \cite[Theorem 6.2]{Lockhart1985}. Independence of $k,p,\gamma$ follows from \autoref{prop elliptic regularity} and \autoref{prop Sobolev embedding and cptness}. 
\end{proof}
The following lemma defines the asymptotic limit map, which maps each element in the kernel of the above AC or CS elliptic operator (with an appropriate weight) its leading-order asymptotic term. Moreover, the lemma provides a wall-crossing formula describing how the index of the operator changes as the weight crosses a wall of critical rates.

\begin{lemma}\label{lem main Fredholm}
Let $\lambda\in \R^m$ and $\lambda_1,\lambda_2$ be two elements in $\R^m\setminus\mathcal D_C$ with $\lambda_1<\lambda<\lambda_2$ for $M=L$ (AC) and $\lambda_2<\lambda<\lambda_1$ for $M= P$ (CS), $\abs{\lambda_2-\lambda_1}\leq \abs{\eta-1}$  and there are no other indicial roots in between $\lambda_{1},\lambda_{2}$ except possibly $\lambda$. We define the following set
	$$\mathcal S_{\lambda_2}:=\{u\in C^{k+1,\gamma}_{M,\lambda_2}: \bD_{M,\lambda_2}^{k,\gamma}u\in C^{k,\gamma}_{M,\lambda_1-1}\}.$$
	Define a linear map $e_{M,\lambda}:V_\lambda\to W^{k+1,p}_{M,\lambda_2}$ (under the identifications of normal bundles $NM$ and $NC_i$ by the canonical bundle isomorphisms over the ends) by \[e_{M,\lambda}(w):=\begin{cases*}
	\bigoplus_{i=1}^m\rho_{\epsilon_0} w_i &  if  $M= P$  (CS)\\
	\bigoplus_{i=1}^m(1-\rho_{R_\infty}) w_i & if   $M=L$  (AC).
	\end{cases*}\] 
	Then there exists a unique linear map, called \textbf{asymptotic limit map}, $\tilde i_{M,\lambda}:\mathcal S_{\lambda_2}\to V_{\lambda}$ satisfying for any $u\in \mathcal S_{\lambda_2} $, 
	\[ u-e_{M,\lambda}\circ \tilde i_{M,\lambda} u\in C^{k,\gamma}_{M,\lambda_1}.\]
	Moreover, the following statements hold. 
	\begin{enumerate}[(i)]
	\item $\mathcal S_{\lambda_2}\subset C^{k+1,\gamma}_{M,\lambda}$ and $\operatorname{Ker}\bD_{M,\lambda_2}^{k,\gamma}=\operatorname{Ker}\bD_{M,\lambda}^{k,\gamma}$, and 		 $\operatorname{Ker}\tilde i_{M,\lambda}=C^{k+1,\gamma}_{M,\lambda_1}$. Moreover, $$\operatorname{Ker}\{i_{M,\lambda}:=\tilde i_{M,\lambda}:\operatorname{Ker}\bD_{M,\lambda_2}^{k,\gamma}\to V_{\lambda}\}=\operatorname{Ker}\bD_{M,\lambda_1}^{k,\gamma}.$$
	\item $\mathcal S_{\lambda_2}=C^{k+1,\gamma}_{M,\lambda_1}+\im e_{M,\lambda}$ and the restriction of ${\bD}_{M,\lambda_2}^{k,\gamma}$, denoted by  
$$\widehat{\bD}_{M,\lambda_1}^{k,\gamma}:C^{k+1,\gamma}_{M,\lambda_1}+\im e_{M,\lambda}\to C^{k,\gamma}_{M,\lambda_1-1} $$ has the property that
$\operatorname{Ker}\bD_{M,\lambda_2}^{k,\gamma}=\operatorname{Ker}\widehat{\bD}_{M,\lambda_1}^{k,\gamma}$ and $\operatorname{Coker}\bD_{M,\lambda_2}^{k,\gamma}\cong\operatorname{Coker}\widehat{\bD}_{M,\lambda_1}^{k,\gamma}$.
\item  \textbf{(Wall crossing formula)} $\ind \bD_{M,\lambda_2}=\ind \bD_{M,\lambda_1}+d_\lambda$.
	\end{enumerate} 
\end{lemma}
\begin{proof}The proof of this lemma can be found in \cite[Proposition 4.21 - Corollary 4.24]{Karigiannis_2020}. We will only show the existence of the asymptotic map $\tilde i_{M,\lambda}$. Given $u\in \mathcal S_{\lambda_2} $ we define $\tilde u\in \bigoplus_{i=1}^m C^{k+1,\gamma} (NC_i)$ by  $$\tilde u:=\begin{cases*}
	\bigoplus_{i=1}^m\tilde u_i:=\rho_{\epsilon_0} u &  if  $M= P$  (CS)\\
	\bigoplus_{i=1}^m(1-\rho_{R_\infty})u & if   $M=L$  (AC).
	\end{cases*}$$ 
Denote any of the asymptotic cone by $C$. Since  $\bD_M=\bD_C+O(r^{\eta-1})$ and $\abs{\lambda_2-\lambda_1}\leq \abs{\eta-1}$, therefore $\bD_C\tilde u\in C^{k,\gamma}_{C,\lambda_1-1}$. \autoref{lem fourier series Donaldson} implies that there exists a unique $v\in C^{k+1,\gamma}_{C,\lambda_1}$ such that $\bD_C(\tilde u-v)=0$. We define the asymptotic map $\tilde i_{M,\lambda}$ as follows:
\begin{equation}\label{eq asymptotic limit map}
	\tilde i_{M,\lambda}(u):=r^{\lambda-1}(\tilde u-v)_{\Sigma,\lambda}.\qedhere
\end{equation}
\end{proof}
Finally, the following proposition computes the index of the AC or CS elliptic operator with an appropriate weight. This is essential for determining the dimension of the moduli spaces of AC and CS associative submanifolds.

\begin{prop}\label{prop index}$\operatorname{Ker}\bD_{M,\lambda}^{k,\gamma}$ is independent of $\lambda$ in each connected component of $\lambda\in\R^m\setminus\mathcal D_C$. Moreover,
		for all $\lambda\in\R^m\setminus\mathcal D_C$, we have
		\begin{enumerate}[(i)]
	\item $\operatorname{Coker}\bD_{M,\lambda}\cong \operatorname{Ker}\bD_{M,-2-\lambda} $. If $\lambda\geq -\frac 12$ for $M=L$ (AC) or $\lambda\leq -\frac 12$ for $M= P$ (CS), then  $\operatorname{Ker}\bD_{M,-2-\lambda}$ is equal to $\operatorname{Coker}\bD_{M,\lambda}$.
	\item If $s>0$ for $M=L$ (AC) or $s<0$ for $M= P$ (CS) such that $-1$ is the possibly only critical rates in between $-1-s$ and $-1+s$, then 
	$$\ind \bD_{M,-1+s}=\dim \operatorname{Ker}\bD_{M,-1+s}-\dim \operatorname{Ker}\bD_{M,-1-s}=\frac{d_{-1}}{2}$$
	
	\item $\ind \bD_{L,\lambda}=-\ind \bD_{P,\lambda}$ and \[\ind \bD_{L,\lambda}=\displaystyle\sum_{\lambda_i\geq -1}\Big(\frac{d_{-1,i}}{2}+\displaystyle\sum_{\zeta_i\in\mathcal D_{C_i}\cap(-1,\lambda_i)}d_{\zeta_i}\Big)-\sum_{\lambda_i< -1}\Big(\frac{d_{-1,i}}{2}+\sum_{\zeta_i\in\mathcal D_{C_i}\cap(\lambda_i,-1)}d_{\zeta_i}\Big).\]
		\end{enumerate}
		\end{prop}
\begin{proof}The first statement is a direct consequence of part (iii) of previous lemma.
Now the first part of (i) follows from \autoref{prop dual weighted space} and \autoref{prop Fredholm Donaldson}, and if $\lambda$ as in second part of (i) we use the fact that $C^{k,\gamma}_{M,-2-\lambda}\subset C^{k,\gamma}_{M,\lambda}$. To see (ii), observe from (i) and the wall crossing formula of \autoref{lem main Fredholm} that $$\dim \operatorname{Ker}\bD_{M,-1+s}-\dim \operatorname{Ker}\bD_{M,-1-s}=\dim \operatorname{Ker}\bD_{M,-1-s}-\dim \operatorname{Ker}\bD_{M,-1+s}+d_{-1}.$$  
Finally, (iii) is an immediate consequence of (ii) and the wall crossing formula stated in the \autoref{lem main Fredholm}.
\end{proof}

\section{Deformations of CS associative submanifolds}
In this section, we study the moduli space of conically singular (CS) associative submanifolds in two settings: first, within a fixed co-closed $G_2$-structure, and second, within a smoothly varying one-parameter family of co-closed $G_2$-structures. We prove two key results: \autoref{thm moduli cs asso}, which describes the structure of the above moduli spaces, and \autoref{thm generic moduli cs asso}, which establishes transversality and generic smoothness results.

\subsection{Moduli space of CS associative submanifolds}\label{subsection moduli CS asso}
We begin by defining the universal moduli space and the moduli space in a fixed co-closed $G_2$-structure of CS associative submanifolds as a set.
	 	 \begin{definition}\label{def moduli CS asso}
We define the universal space of conically singular (CS) submanifolds, 
$\mathcal S_{\operatorname{cs}}:=\{(\phi,P):P \ \text{is a CS submanifold in}\ (Y,\phi), \phi \in \sP  \}.$ The universal \textbf{moduli space}  $\cM_{\operatorname{cs}}$ of conically singular associative submanifolds is defined by
$$\cM_{\operatorname{cs}}:=\{(\phi,P)\in \mathcal S_{\operatorname{cs}}:P \ \text{is a CS associative submanifold} \}.$$ 
There is a canonical projection map $\pi:\mathcal S_{\operatorname{cs}}\to\sP$ and we define the \textbf{moduli space} of CS associative submanifolds in $(Y,\phi)$ by \[\cM_{\operatorname{cs}}^\phi:=\pi^{-1}(\phi)\cap \cM_{\operatorname{cs}}.\qedhere\]
 \end{definition}

Next, we define a topology on $\mathcal{M}_{\operatorname{cs}}$, called the weighted $C^\infty$-topology. This topology is specified by defining a basis of open sets around each of its elements. Given an element of $\mathcal{M}_{\operatorname{cs}}$, we can make a choice of an end-conical singular (ECS) submanifold (see \autoref{def PC}) equipped with an end-conical (EC) tubular neighbourhood map (see \autoref{def section alpha for P}) near it. We then construct a canonical family of ECS submanifolds and corresponding EC tubular neighbourhood maps by varying the co-closed $G_2$-structures, as well as the singularity data—such as the positions of the singularities and their associated model cones. The open sets consisting of CS associative submanifolds lying inside this family of tubular neighbourhoods together with co-closed $G_2$-structures define a basis for the weighted $C^\infty$-topology around the given element.

\begin{definition}[ECS submanifold]\label{def PC} Let $P$ be a CS submanifold of $(Y,\phi)$ as in \autoref{def CS asso}. Given a choice of $G_2$-coordinate systems $\Upsilon^i$ at the singular points and the other data that is used for $P$ in \autoref{def CS asso}, we define an \textbf{end conical singular} (ECS) submanifold $P_C$ which is diffeomorphic to $\hat P$ but conical on the ends as follows:
$$P_C:=K_P\cup \big(\bigcup_{i=1}^m(\Upsilon^i\circ\Upsilon_{C_i})((1-\rho_{{\epsilon_0}})(\Psi^i_P-\iota_i))\big).$$ 
Here $\rho_{\epsilon_0}$ is the cut off function defined in \autoref{eq:cutoff function}. 

Set, \begin{equation*}
	C_{i,\epsilon_0}:=\iota((0,\epsilon_0)\times \Sigma_i) \qandq K_{P_C}:=P_C\setminus \bigcup_{i=1}^m\Upsilon^i(C_{i,\epsilon_0}).\qedhere
\end{equation*}
\end{definition}

\begin{definition}[EC tubular neighbourhood map]\label{def section alpha for P}
Let $P$ be a CS submanifold as in \autoref{def CS asso}. Let $P_C$ be a choice of an ECS submanifold as in \autoref{def PC}. A tubular neighbourhood map of ${P_C}$
	$$\Upsilon_{P_C}:V_{P_C}\to U_{P_C}$$
is called \textbf{end conical (EC)} if $V_{P_C}$ and $\Upsilon_{P_C}$ agree with $\Upsilon_*(V_C)$ and $\Upsilon\circ\Upsilon_{C}\circ\Upsilon_*^{-1}$ on $\Upsilon\big(\iota((0,\epsilon_0)\times \Sigma)\big)$, respectively. 
Here, $\Upsilon_{C}:V_C\to U_C$ is a conical tubular neighbourhood map of $C$ as in \autoref{def conical tubular nbhd}	

Given a choice of an ECS submanifold $P_C$ and a choice of an EC tubular neighbourhood map $\Upsilon_{P_C}$, there is a section $\alpha$ in $V_{P_C}\subset N{P_C}$ which is zero on $K_P$ and $\Upsilon^i_*\circ(\Psi^i_P-\iota_i)=\alpha\circ \Upsilon^i$ such that $\Upsilon_{P_C}(\Gamma_\alpha)$ is $P$.
\end{definition}

\begin{remark}
Set in the above framework, $$V_{C_{i,\epsilon_0}}:=V_{{C_i}_{|C_{i,\epsilon_0}}}, \quad U_{C_{i,\epsilon_0}}:=U_{{C_i}_{|C_{i,\epsilon_0}}}, \quad P_{C_{i,\epsilon_0}}:=P_{C_{|\Upsilon^i(C_{i,\epsilon_0})}},\qandq V_{P_{C_{i,\epsilon_0}}}:=V_{P_C{_{|\Upsilon^i(C_{i,\epsilon_0})}}}.$$ The following commutative diagram helps us to keep track of the definitions above.
\[\begin{tikzcd}
C_{i,\epsilon_0}\arrow[r,bend left,"\Psi^i_P-\iota_i"] \arrow{d}{\Upsilon^i}&V_{C_{i,\epsilon_0}} \arrow{l}\arrow{r}{\Upsilon_{C_i}} \arrow[swap]{d}{\Upsilon^i_*} & U_{C_{i,\epsilon_0}} \arrow{d}{\Upsilon^i} \\
P_{C_{i,\epsilon_0}}\arrow[r,bend right,"\alpha"]&V_{P_{C_{i,\epsilon_0}}}=\Upsilon^i_*(V_{C_{i,\epsilon_0}})\arrow{l} \arrow{r}{\Upsilon_{P_C}} & Y\setminus K_Y. 
\end{tikzcd}\qedhere
\]
\end{remark}

\begin{definition}[CS submanifolds after varying $G_2$-structure and singularity data]\label{def smooth family EC}
	Let $P\in \cM^\phi_{\ocs}$ be a CS associative submanifold with singularities at $p_i$ modeled on cones $C_i$, $i=1,\dots,m$ as in \autoref{def CS asso}. Let $\Sigma_i$ be the link of $C_i$. Let $P_C$ be a choice of an ECS submanifold as in \autoref{def PC}, and let $\Upsilon_{P_C} : V_{P_C} \to U_{P_C}$ be a choice of an EC tubular neighbourhood map of $P_C$ as in \autoref{def section alpha for P}. Since the singular points $p_i$ and associative cones $C_i$ are allowed to vary we need to also vary the $P_C$ and $\Upsilon_{P_C}$. This is done as follows.
	By \autoref{thm moduli holo}, there are tubular neighbourhood maps $\Upsilon_{\Sigma_i}:V_{\Sigma_i}\to U_{\Sigma_i}$ of ${\Sigma_i}$ and obstruction maps $\ob_{\Sigma_i}:\mathcal I_{\Sigma_i} \to \mathcal O_{\Sigma_i}$. Let $U_{\phi}$ and $U_{p_i}$ be sufficiently small neighbourhoods of $\phi$ in $\sP$ and $p_i$ in $ \Upsilon^i\big(B(0,{\epsilon_0})\big)\subset Y$ , respectively, where $\Upsilon^i$ is a $G_2$-coordinate system from \autoref{def CS asso}. Set 
\begin{equation}\label{eq Utau0}
U_{\tau_0}:=U_\phi \times \Big(\prod_{i=1}^m U_{p_i}\Big) \times \Big(\prod_{i=1}^m \cI_{\Sigma_i}\Big) \qandq \tau_0:=(\phi,p_1,\dots,p_m, 0, \dots, 0)\in U_{\tau_0}.
\end{equation}
The open set $U_{\tau_0}$ essentially parametrizes nearby co-closed $G_2$-structures together with nearby singularity data (positions and model cones).

For the parameter $\tau_0$, we already have that $P_C$ is an ECS submanifold and $\Upsilon_{P_C}$ is an EC tubular neighbourhood map. Our goal is to construct a canonical smoothly varying (in the parameter $\tau \in U_{\tau_0}$) family of ECS submanifolds $P_C^\tau$ with $G_2$-structure and the singularities and model cones are moved according to the data encoded in $\tau$. Furthermore, we aim to construct an EC tubular neighbourhood map $\Upsilon_{P_C^\tau}$, which will serve to define the weighted $C^\infty$-topology.

To achieve this, we construct a \textbf{canonical smooth family of diffeomorphisms} (introduced in the next \autoref{def canonical diffeo for weighted topo}),
\begin{equation}\label{eq canonical diffeo for weighted topo}
\Phi: U_{\tau_0} \to \prod_{i=1}^m \Diff(B(0, R)), \quad \tau := \big( \phi_\tau, (p_i^\tau), (\xi_i^\tau) \big) \mapsto (\Phi_i^\tau),
\end{equation}
satisfying the following conditions:
\begin{enumerate}[(i)]
\item $\Upsilon_i^\tau := \Upsilon^i \circ \Phi_i^\tau$ defines a $G_2$-coordinate system at $p_i^\tau$ for the co-closed $G_2$-structure $\phi_\tau$, where $\phi_\tau$  and $p_i^\tau$ are the data encoded in $\tau$.
\item $\Phi_i^{\tau_0}$ is the identity on $B(0, R)$, and for all $\tau \in U_{\tau_0}$, $\Phi_i^\tau$ is equal to the identity outside $B(0, 2\epsilon_0)$.
\end{enumerate}

With this in place, we define for each $\tau \in U_{\tau_0}$ an ECS submanifold $P_C^\tau$ and an associated EC tubular neighbourhood map $\Upsilon_{P_C^\tau}$ by
\[
P_C^\tau := \left( \bigcup_{i=1}^m \Upsilon_i^\tau \circ (\Upsilon^i)^{-1}(P_C) \right) \cup K_P, \quad \text{and} \quad \Upsilon_{P_C^\tau} := \left( \bigcup_{i=1}^m \Upsilon_i^\tau \circ (\Upsilon^i)^{-1} \circ \Upsilon_{P_C} \right) \cup \Upsilon_{{P_C}_{|K_P}} : V_{P_C} \to Y.
\]

Note that $P_C^\tau$ has singularities located at $p_i^\tau$, modeled on associative cones $C^\tau_i$ whose links are $\Sigma_i^\tau := \Upsilon_{\Sigma_i}(\xi_i^\tau)$, where $\xi_i^\tau \in \mathcal{I}_{\Sigma_i}$ represents the infinitesimal deformation of $\Sigma_i$ encoded in $\tau$. Moreover, any CS submanifold in $(Y, \phi)$ lying inside $U_{P_C}$ with the same asymptotic data as $P$ is mapped, via $\Upsilon_{P_C^\tau} \circ \Upsilon_{P_C}^{-1}$, to a CS submanifold in $(Y, \phi_\tau)$ with asymptotic data matching that of $P_C^\tau$, and vice versa.
 \end{definition}

\begin{definition}[Canonical smooth family of diffeomorphisms]\label{def canonical diffeo for weighted topo}
The time-$1$ flows of the following family of vector fields in \autoref{eq vf 3} parametrized by $U_{\tau_0}$, defines the required smooth family of diffeomorphisms $\Phi$ in \autoref{eq canonical diffeo for weighted topo}.

To define \autoref{eq vf 3}, we first construct a family of vector fields parametrized by $ \mathcal{I}_{\Sigma_i} $ as follows. Given an element $ \xi_i \in \mathcal{I}_{\Sigma_i} $, we use the extension map $ \widetilde{\bullet}$ from \autoref{def canonical extension normal vf} to obtain a vector field $ \widetilde{\xi}_i $ on $ V_{\Sigma_i} $. Applying the differential of the tubular neighbourhood map $ \Upsilon_{\Sigma_i} $, we obtain the vector field $ d\Upsilon_{\Sigma_i}(\widetilde{\xi}_i) $ on $ U_{\Sigma_i} $. Next, we extend this to a global vector field on $ S^6 $ by multiplying with a cut-off function supported in a neighbourhood of $ U_{\Sigma_i} $. Finally, we obtain a vector field $ v_{\xi_i} $ supported on $B(0, 2\epsilon_0)\subset B(0, R) $ by radially translating and multiplying by the cut-off function $ \rho_{\epsilon_0} $ defined in \autoref{eq:cutoff function}. In summary, the construction can be expressed as:
\begin{equation}\label{eq vf 1}
	\mathcal{I}_{\Sigma_i} \xrightarrow{\widetilde{\bullet}} \Vect(V_{\Sigma_i}) \xrightarrow{d\Upsilon_{\Sigma_i}} \Vect(U_{\Sigma_i}) \hookrightarrow \Vect(S^6) \xrightarrow{\rho_{\epsilon_0} \cdot} \Vect(B(0,R)).
\end{equation}
The time-$1$ flows of the vector fields constructed above move the holomorphic curves and their corresponding associative cones near origin but fixes the positions of the singularities.

Next, we construct another smooth family of vector fields parametrized by $U_\phi \times U_{p_i}$ as follows. Given a pair $({\phi^\prime}, {p_i^\prime}) \in U_\phi \times U_{p_i}$, we use the $G_2$-coordinate system $\Upsilon^i$ to pull back the $G_2$-structures $\phi$ at $p_i$  and $\phi^\prime$ at $p_i^\prime$. The resulting pullbacks $(\Upsilon^i)^*\phi(p_i)$ and $(\Upsilon^i)^*\phi^\prime(p_i^\prime)$ are two $G_2$-structures on $\R^7$, where the first can be mapped to the second by an element of $\mathrm{GL}_7(\R)$. This element can be taken to be the exponential of a matrix $A_{\phi', p_i'} \in M_7(\R)$. However, this choice is not unique, since composing with an element of $G_2$ yields another valid representative in $\mathrm{GL}_7(\R)$. To resolve this ambiguity, we fix a smooth family of matrices $A_{\phi', p_i'}\in M_7(\R)$ satisfying the initial condition $A_{\phi, p_i}=0$. We interpret $A_{\phi^\prime, p_i^\prime}$ as a linear vector field on $\R^7$ and hence consider its restriction to the ball $B(0, R)$. Translating by $(\Upsilon^i)^{-1}(p_i^\prime)$ and after that multiplying by the cut-off function $\rho_{\epsilon_0}$, we obtain a vector field $v_{\phi^\prime, p_i^\prime}$ supported in $B(0, 2\epsilon_0)$.  In summary, this defines the smooth family:
\begin{equation}\label{eq vf 2}
	U_\phi \times U_{p_i} \to \Vect(B(0,R)), \quad ({\phi^\prime}, {p_i^\prime}) \mapsto v_{\phi^\prime, p_i^\prime}\ ,
\end{equation}
where \[v_{\phi^\prime, p_i^\prime}(x):=\rho_{\epsilon_0} \big(A_{\phi^\prime, p_i^\prime}(x)+(\Upsilon^i)^{-1}(p_i^\prime)\big).\]
The time-$1$ flows of this family vary both the $G_2$-structure and the positions of the singularities together with the model associative cones.

By summing the smooth families of vector fields from \autoref{eq vf 1} and \autoref{eq vf 2}, we obtain the desired smooth family of vector fields:
\begin{equation}\label{eq vf 3}
	U_{\tau_0} \to \prod_{i=1}^m \Vect(B(0,R)). \qedhere
\end{equation}
\end{definition}
\begin{definition}[\textbf{Weighted $C^\infty$-topology}]\label{def CS weighted topo}
We now define the {weighted $C^\infty$-topology} on the moduli space $\cM_{\operatorname{cs}}$ by specifying a local basis around each element $(\phi, P) \in \cM_{\operatorname{cs}}$. 
Given such an element, we make a choice of an ECS submanifold $P_C$ as in \autoref{def PC}, and a choice of an EC tubular neighbourhood map of $P_C$ as in \autoref{def section alpha for P}, and consider a neighbourhood $U_{\tau_0}$ as in \autoref{eq Utau0}, small enough so that, for all $\tau \in U_{\tau_0}$, the critical rates greater than $1$ of the associative cones $C_i^\tau$ remain uniformly bounded away from $1$. This allows us to choose a decay rate $\mu = (\mu_1, \dots, \mu_m)$ such that 
\begin{equation}\label{eq tech weighted topo}
2>\mu_i > 1, \qandq (1, \mu_i) \cap \mathcal{D}_{C^\tau_i} = \emptyset \quad \text{for all } \tau \in U_{\tau_0}.
\end{equation}
This technical assumption is explained in \autoref{rmk welldef weighted topo} below. For now, fix such $\mu$. The local basis of the topology is then given by all subsets of the form
\[
\mathcal{V_{\phi,P}^\mu}:=\left\{ \big( \phi_\tau, \Upsilon_{P_C^\tau}(\Gamma_u) \big) : u \in \mathcal{V}_{P_C, \mu}, \, \tau \in \mathcal{V}_{\tau_0} \right\} \cap \cM_{\operatorname{cs}},
\]
where $\mathcal{V}_{\tau_0} \subset U_{\tau_0}$ and $\mathcal{V}_{P_C, \mu} \subset C^\infty_\mu(V_{P_C}) := C^\infty_{P_C, \mu} \cap C^\infty(V_{P_C})$ are open subsets. Here  $ C^\infty_{P_C, \mu}$  is the space of smooth normal vector fields of order $O(r^\mu)$ on $P_C$ as defined in \autoref {eq decay weighted smooth function space}.
\end{definition}
\begin{remark}\label{rmk welldef weighted topo} The construction of $P_C$ and the tubular neighbourhood map $\Upsilon_{P_C}$ involves several choices; however, the weighted $C^\infty$-topology is independent of these choices.

 At first glance, the above definition may appear to depend on the choice of the decay rate $\mu$ satisfying \autoref{eq tech weighted topo}. However, for any other choice $\mu^\prime$ satisfying \autoref{eq tech weighted topo}, $\mu_i^\prime$ must lie in the same connected component of $(1,2)\setminus \mathcal{D}_{C_i^\tau}$ as $\mu_i$. Consequently, by \autoref{rmk CS two rates}, we have 
$\mathcal{V}_{\phi, P}^{\mu} = \mathcal{V}_{\phi, P}^{\mu\prime}$ and the above definition does not depend on the choice.

Another advantage of choosing each $\mu_i$ close to $1$ is to avoid crossing walls of critical rates. If $\mu$ were to cross such a wall, the corresponding set $\mathcal{V}_{\phi, P}^{\mu}$—parametrizing CS associative submanifolds with stronger decay—would become strictly smaller, yet still open in $\cM_{\operatorname{cs}}$, which is undesirable. In principle, one could define sub-moduli spaces by restricting to CS associative submanifolds with stronger decay rates; these would inherit the subspace topology from the weighted $C^\infty$-topology introduced above. However, we do not explore this direction in the present article.
\end{remark}

To understand the local structure of the moduli space $\mathcal{M}_{\operatorname{cs}}$ of CS associative submanifolds, we define below a nonlinear map whose zero set locally models $\mathcal{M}_{\operatorname{cs}}$.
Let $P\in \mathcal M^\phi_{\ocs}$ be a CS associative submanifold. We make a choice of an ECS submanifold $P_C$ as in \autoref{def PC}, and a choice of an EC tubular neighbourhood map of $P_C$ as in \autoref{def section alpha for P}. There is a \textbf{canonical} bundle isomorphism from \autoref{def prelim canonical prelim isomorphisms},
\begin{equation}\label{def identification of normal bundle}
\Theta_{P}^C:NP_C\to NP.
\end{equation}
 Choose $U_{\tau_0}$ and $\mu$ as in \autoref{def CS weighted topo}. For each $\tau \in U_{\tau_0}$, the ECS submanifold $P_C^\tau$ and the associated EC tubular neighbourhood map $\Upsilon_{P_C^\tau}$ be as in \autoref{def smooth family EC}.

\begin{definition}\label{def nonlinear map}
Define $\mathfrak F: C^{\infty}_{\mu} (V_{P_C})\times U_{\tau_0}\to C^{\infty}{(NP_C)}$ by for all $u\in C^\infty_{\mu} (V_{P_C})$, $\tau\in U_{\tau_0}$  and $w\in C_c^\infty (N{P_C})$, 
	 \begin{equation*}\inp{\mathfrak F(u,\tau)}{w}_{L^2}:=\int_{\Gamma_{u}}\iota_{w}\Upsilon_{P^\tau_C}^*\psi_\tau. 
 \end{equation*}
 The notation $w$ in the integrand is the extension vector field of $w\in C^\infty(NP_C)$ in the tubular neighbourhood as in \autoref{notation u}. The $L^2$ inner product we choose here is coming from the canonical bundle isomorphism $\Theta_{P}^C:NP_C\to NP$ as in \autoref{def identification of normal bundle} and the induced metric on $NP$ of $g_\phi$.
	 \end{definition}

\begin{remark}The CS associative $P$ is represented by a section $\alpha\in C^{\infty}_{\mu} (V_{P_C})$ from \autoref{def section alpha for P} and therefore $\mathfrak F(\alpha,\tau_0)=0$. Moreover, for $u\in C^{\infty}_{\mu} (V_{P_C})$ and $\tau\in U_{\tau_0}$ we have $\Upsilon_{P^{\tau}_C}(\Gamma_u)\in \cM^{\phi_\tau}_{\operatorname{cs}}$ iff $\mathfrak F(u,\tau)=0$.
	\end{remark}
To express the moduli space $\mathcal{M}^\phi_{\operatorname{cs}}$ locally as the zero set of a map (Kuranishi map) between finite-dimensional spaces, and to establish transversality results, we need to analyze the Fredholm property of the linearization of the nonlinear map introduced above. This is the focus of the following discussion.

\begin{prop}\label{prop linearization nonliear}For $\tau\in U_{\tau_0}$, the linearization of $\mathfrak F_\tau:=\mathfrak F(\cdot,\tau)$ at $u\in C^{\infty}_{P_C,\mu}(V_{P_C})$, 
$$d\mathfrak F_{\tau_{|u}}:C^{\infty}_{P_C,\mu}\to C^{\infty}_{P_C,\mu-1}$$ is given for all $v\in C^{\infty}_{P_C,\mu}$ and $w \in C_c^\infty(NP)$ by
$\inp{d\mathfrak F_{\tau_{|u}}(v)}{w}=\int_{\Gamma_u}\iota_w\mathcal L_v(\Upsilon_{P_C^\tau}^*\psi_\tau)$. This is same as
$$\int_{\Gamma_u}\sum_{\substack{\text{cyclic}\\\text{permutaions}}}\big\langle[e_2,e_3,\nabla_{e_1}v],w\big\rangle+\iota_w \nabla_v(\Upsilon_{P^{\tau}_C}^*\psi_\tau)+\iota_{\nabla_wv}(\Upsilon_{P^{\tau}_C}^*\psi_\tau),$$
where $\{e_1,e_2,e_3\}$ is a local orthonormal frame for $T\Gamma_u$ and the associator `$[\cdot,\cdot,\cdot]$' is defined with respect to $\Upsilon_{{P_C}^\tau}^*\phi_\tau$.
\end{prop}
\begin{proof} For a family $\{u+tv\in C^{\infty}_{\mu} (V_{P_C})  :\abs{t}\ll1\}$ we have $$\frac d{dt}\big|_{{t=0}}\mathcal F_{u+tv}(w)\\
=\frac d{dt}\big|_{{t=0}} \int_{\Gamma_{u+tv}}\iota_w(\Upsilon_{P^{\tau}_C}^*\psi_\tau)=\int_{\Gamma_u}\mathcal L_v\iota_w(\Upsilon_{P^{\tau}_C}^*\psi_\tau)=\int_{\Gamma_u}\iota_w\mathcal L_v(\Upsilon_{P^{\tau}_C}^*\psi_\tau)+{\iota_{[v,w]}(\Upsilon_{P^{\tau}_C}^*\psi_\tau)}.$$
As $[v,w]=0$ (see \autoref{notation u}), this is further equal to the following:
\begin{align*}
&\displaystyle\int_{\Gamma_u}\sum_{\substack{\text{cyclic}\\\text{permutaions}}}\iota_w(\Upsilon_{P^{\tau}_C}^*\psi_\tau)(\nabla_{e_1}v,e_2,e_3)+\int_{\Gamma_u}\iota_{\nabla_wv}(\Upsilon_{P^{\tau}_C}^*\psi_\tau)+\iota_w \nabla_v(\Upsilon_{P^{\tau}_C}^*\psi_\tau). \qedhere 
\end{align*}
\end{proof}
\begin{cor}\label{cor linearization cs}Let $u\in C^{\infty}_{\mu} (V_{P_C})$ and  $\tau \in U_{\tau_0}$. If $\Upsilon_{P^{\tau}_C}(\Gamma_u)\in \mathcal M_{\ocs}$, a CS associative submanifold then $d\mathfrak F_{\tau_{|u}}$ is given by the following: for all $v\in C^{\infty}_{P_C,\mu}$ and $w \in C_c^\infty(NP_C)$,
 \begin{equation*}
 	\inp{d\mathfrak F_{\tau_{|u}}(v)}{w}_{L^2} =\int_{\Gamma_u}\biggl\langle\sum_{\substack{\text{cyclic}\\\text{permutations}}}[e_2,e_3,\nabla_{e_1}v],w\biggr\rangle +\int_{\Gamma_u} \iota_w \nabla_v(\Upsilon_{P^{\tau}_C}^*\psi_\tau).\qedhere
 	 \end{equation*}
\end{cor}
\begin{prop}\label{prop formal self adjoint}For $\tau\in U_{\tau_0}$, the linearization of $\mathfrak F_\tau$ at $u\in C^{\infty}_{P_C,\mu}(V_{P_C})$, $d\mathcal F_{\tau_{|u}}$ is a formally self adjoint first order differential operator.
\end{prop}
\begin{proof}For all $v, w \in C_c^\infty(NP_C)$, the difference $d\mathcal F_{\tau_{|u}}(v)(w)-d\mathcal F_{\tau_{|u}}(w)(v)$ is
$$\int_{\Gamma_u}\mathcal L_v\iota_w(\Upsilon_{P_C^\tau}^*\psi_\tau)-\mathcal L_w\iota_v(\Upsilon_{P_C^\tau}^*\psi_\tau)=\int_{\Gamma_u}\iota_w\mathcal L_v(\Upsilon_{P_C^\tau}^*\psi_\tau)+{\iota_{[v,w]}(\Upsilon_{P_C^\tau}^*\psi_\tau)}-\mathcal L_w\iota_v(\Upsilon_{P_C^\tau}^*\psi_\tau).$$
As $[v,w]=0$ and $d\psi_\tau=0$, this is equal to $\int_{\Gamma_u}\iota_w\iota_v(\Upsilon_{P_C^\tau}^*d\psi_\tau)-d(\iota_w\iota_v(\Upsilon^*_{P_C^\tau}\psi_\tau))=0$.
\end{proof}

\begin{definition}\label{def linearization L on cs asso } We define the differential operator $\mathfrak L_P:C_c^\infty(NP_C)\to C_c^\infty(NP_C)$ by 
\begin{equation*}\mathfrak L_Pu:=d\mathfrak F_{{\tau_0|}_{\alpha}}(u), \end{equation*}
where the CS associative $P$ is represented by $\alpha\in C^{\infty}_{\mu} (V_{P_C})$ as in \autoref{def section alpha for P}.
\end{definition}	

\begin{prop}\label{prop identification of normal bundle}We have
   \begin{equation*}
\mathfrak L_{P}=(\Theta_{P}^C)^{-1}\circ \mathbf D_{P}\circ \Theta_{P}^C,
  \end{equation*}
  where $\mathbf D_{P}:C^\infty_c(NP)\to C^\infty_c(NP)$ is the Fueter operator defined in \autoref{def D_M} and the canonical bundle isomorphism $\Theta_{P}^C:NP_C\to NP$ is as in \autoref{def identification of normal bundle}.
 \end{prop}
\begin{proof}For all $v,w\in C_c^{\infty}(NP_C)$ we have,
	\begin{align*}
 	\inp{\mathfrak L_{P}v}{w}_{L^2}
 	&= \inp{\Theta_{P}^C\mathfrak L_{P}v}{\Theta_{P}^Cw}_{L^2(NP)}\\
 	&=\int_{\Gamma_\alpha}\biggl\langle\sum_{\substack{\text{cyclic}\\\text{permutations}}}[e_2,e_3,\nabla^\perp_{e_1}v],w\biggr\rangle+\int_{\Gamma_\alpha} \iota_w \nabla_v(\Upsilon_{P_C}^*\psi)\\
 	&=\int_{P}\biggl\langle\sum_{\substack{\text{cyclic}\\\text{permutations}}}[e_2,e_3,\nabla^\perp_{e_1}(\Theta_{P}^Cv)],\Theta_{P}^Cw\biggr\rangle+\int_{P} \iota_{\Theta_{P}^Cw} \nabla_{\Theta_{P}^Cv}\psi\\
 	&=\inp{\bD_P \Theta_{P}^Cv}{\Theta_{P}^Cw}_{L^2(NP)}.
 	 \end{align*}
 The equality before the last equality holds, because $w-\Theta_{P}^Cw\in TP$ and $v-\Theta_{P}^Cv\in TP$, and $[\cdot,\cdot,\cdot]_{|P}=0$.
	 \end{proof}
\begin{prop}\label{prop CS elliptic operator}The operator $\mathfrak L_{P}$ is a conically singular uniformly elliptic operator asymptotic to the conical operators $\mathbf D_{C_i}$.  
\end{prop}
\begin{proof}
Since $\mathbf D_P$ is elliptic and  $\mathfrak L_{P}=(\Theta_{P}^C)^{-1}\circ\mathbf D_{P}\circ\Theta^C_{P}$ by \autoref{prop identification of normal bundle},  we obtain that $\mathfrak L_{P}$ is elliptic.

Without loss of generality we assume that $m=1$ and $C=C_i$. It remains to prove that $\mathfrak L_{P}$ is CS asymptotic to $\mathbf D_C$. We substitute below $r=e^t$ and denote $C_t:=(t,\infty)\times\Sigma$. For all $v,w\in C^\infty_c(NC_t)$  we have
\begin{align*}
	\inp{r^2d\mathfrak F_{{\tau_0|}_{0}}\Upsilon_*(r^{-1}v)}{\Upsilon_*w}_{L^2}&=\int_{C_t}r\iota_w\mathcal L_v(\Upsilon_{C}^*\Upsilon^*\psi)\\
	&=\int_{C_t}r\iota_w\mathcal L_v(\Upsilon_{C}^*\psi_0)+O(e^{-t})\norm{v}_{W^{1,2}(NC)}\norm{w}_{L^2(NC)}.
\end{align*}
By \autoref{prop cs quadratic estimate} (i) we get $$\inp{\Upsilon^*\mathfrak L_{P}v}{w}_{L^2(NC)}= \inp{\mathbf D_Cv}{w}_{L^2(NC)}+ O(r^{\mu-1})\norm{v}_{W^{1,2}(NC)}\norm{w}_{L^2(NC)}.$$
This completes the proof of the proposition.
\end{proof}

\autoref{prop CS elliptic operator} and \autoref{prop identification of normal bundle} imply that $\mathbf D_{P}$ is also a CS uniformly elliptic operator and asymptotic to the conical operator $\mathbf D_C$. Thus $\fL_P$ and $\bD_P$ has canonical extensions to weighted function spaces and this is the content of the following definition. 

\begin{definition}\label{def linearization D on cs asso} Let $P$ be a CS associative submanifold and $\mathbf D_{P}:C^\infty_c(NP)\to C^\infty_c(NP)$ be the Fueter operator defined in \autoref{def D_M}. \autoref{prop CS elliptic operator} implies that it is a conically singular uniformly elliptic operator. 
Therefore it has the following canonical extensions: 
\begin{equation*}\mathbf D_{P,\lambda}^{k,p}:W^{k+1,p}_{P,\lambda}\to W^{k,p}_{P,\lambda-1}, \ \ \ \mathbf D_{P,\lambda}^{k,\gamma}:C^{k+1,\gamma}_{P,\lambda}\to C^{k,\gamma}_{P,\lambda-1}.
\end{equation*}
Similarly we have canonical extensions of the operator $\mathfrak L_{P}$ to weighted function spaces:
  \begin{equation*}
\mathfrak L_{P,\lambda}^{k,p}:W^{k+1,p}_{P_C,\lambda}\to W^{k,p}_{P_C,\lambda-1}, \ \ \ \mathfrak L_{P,\lambda}^{k,\gamma}:C^{k+1,\gamma}_{P_C,\lambda}\to C^{k,\gamma}_{P_C,\lambda-1}.\qedhere  \end{equation*}
\end{definition}

\begin{definition}[Asymptotic limit map]\label{def asymp limit map}
We define the \textbf{asymptotic limit map} $$i_{P,\lambda}:\operatorname{Ker}\mathbf D_{P,\lambda}\to V_\lambda$$ to be the map $i_{P,\lambda}$ defined in \autoref{lem main Fredholm} for $\bD_P$. 
\qedhere	
 \end{definition}

The following proposition computes the linearization of the non-linear map $\fF$ with respect to variations in the singularity data—that is, the effect of moving the singular points and deforming their associated model cones.

\begin{prop}\label{lem linearization 2nd component}
The linearization of $ \mathfrak{F}(\alpha, \phi, \cdot) $ at $ (p_1, \dots, p_m, 0, \dots, 0) $,
\[
L_1 \colon \bigoplus_{i=1}^m (\bR^7 \oplus V_{1,i}) \to C^{\infty}_{P_C, \mu-1}
\]
is of the form:
\[
L_1(\hat p_i, \xi_i) = \mathfrak{L}_P\left( \rho_{\epsilon_0}\Upsilon^i_* (\hat  p_i^\perp+\xi_i) \right)+  O(r^{\mu-1}).
\]
Here, $ \bR^7 $ is identified with the tangent space of $ U_{p_i} $ at $ p_i $ by the framing $\upsilon_i$, and $ V_{1,i} $ is the homogeneous kernel of the cone $ C_i $ whose link is $\Sigma_i$, at rate $ \lambda = 1 $ in the sense of \autoref{def homogeneous kernel}. The later is identified with the tangent space of $ \mathcal{I}_{\Sigma_i} $ at $0$, which is again $T_{\Sigma_i}\cM^{\operatorname{hol}}$. The vector $\hat p_i \in \bR^7$ projects to a normal vector field $\hat  p_i^\perp$ on $C_i$ and $\xi_i\in V_{1,i}$ is already a normal vector field on $C_i$.
\end{prop}

\begin{proof}Without loss of generality, we assume that $m=1$ and $C=C_i$. Suppose $\xi\in  V_{1}$ and for each $t$ small, $\tau_t:=(\phi, p, t\xi)$.
For all $w\in C^\infty_c(NP_C)$ we have	
$$\inp{L_1(0,\xi)}{w}_{L^2}=\int_{\Gamma_{\alpha}}\frac d{dt}\big|_{{t=0}} \iota_w(\Upsilon_{P_C^{\tau_t}}^*\psi)=\int_{\Gamma_{\alpha}}\iota_w\mathcal L_{(\rho_{\epsilon_0}\xi)}(\Upsilon_{P_C}^*\psi).$$

Suppose $\hat p\in \R^7$ and for each $t$ small, $\tau_t:=(\phi, \Upsilon (t\hat p),0)$.
For all $w\in C^\infty_c(NP_C)$ we have	
$$\inp{L_1(\hat p,0)}{w}_{L^2}=\int_{\Gamma_{\alpha}}\frac d{dt}\big|_{{t=0}} \iota_w(\Upsilon_{P_C^{\tau_t}}^*\psi)=\int_{\Gamma_{\alpha}}\iota_w\mathcal L_{(\rho_{\epsilon_0}\hat p^\perp)}(\Upsilon_{P_C}^*\psi)+O(r^{\mu-1}) \norm{w}_{L^2}.$$
 The proposition now follows from \autoref{def linearization L on cs asso } of $\fL_P$. 
 \end{proof}
 
\begin{remark}The expressions for $L_1$ in the above \autoref{lem linearization 2nd component} is given only up to terms of order $O(r^{\mu-1})$ near the singularities. This imprecision poses no issue for the purposes of this article, as the explicit form of $L_1$ is not used except in \autoref{lem generic matching kernel d=2}. In that lemma, the ambiguity is inconsequential.
\end{remark}

The operator $\fL_{P}$ from \autoref{def linearization L on cs asso } governs the associative deformations of $P$ with fixed singulaity data. By incorporating the additional operator $L_1$ from \autoref{lem linearization 2nd component}, we define the extended operator $\widetilde \fL_P$ in the following lemma, which is less obstructed and captures the broader class of deformations in which the singularities are permitted to move and the model cones may vary.

\begin{definition}\label{def lin L, D with Z}Let $P$ be a CS associative as in \autoref{def CS asso} asymptotic to $C_i$, $i=1,\dots,m$. Denote the link of $C_i$ by $\Sigma_i$. Let $\Sigma_i=\sqcup_{j=1}^l\Sigma_i^j$ be the decomposition into connected components. Let $\cZ_i^j$ be the manifold in the decomposition \autoref{eq stratification} that contains $\Sigma_i^j$. Set $\cZ_i=\prod_{j=1}^l\cZ_i^j$. We also denote by $\Sigma_i$ the product $\prod_{j=1}^l\Sigma_i^j\in \cZ_i $.  Since $T_{\Sigma_i}\cZ_i \subset T_0\cI_{\Sigma_i}=V_{1,i}$, we define
$$\widetilde \fL_{P,\mu,\cZ}:C^{\infty}_{P_C,\mu}\oplus \big(\bigoplus_{i=1}^m(\R^7\oplus T_{\Sigma_i}\cZ_i) \big)\to C^{\infty}_{P_C,\mu-1}$$
to be the restriction of the linearization of $\mathfrak F(\cdot,\phi,\cdot)$ at $(\alpha,\phi,p_1,\dots,p_m, 0, \dots, 0)$, that is, 
\begin{equation*}
	\widetilde \fL_{P,\mu,\cZ}(u, \hat p_i, \xi_i)=\fL_{P}u+L_1(\hat p_i, \xi_i), \end{equation*}
where $u\in C^{\infty}_{P_C,\mu} $, $\hat p_i\in \R^7$ and $\xi_i\in T_{\Sigma_i}\cZ_i$. The operators $\fL_{P}$ and $L_1$ are from \autoref{def linearization L on cs asso } and \autoref{lem linearization 2nd component}. 

The operator $\widetilde \bD_{P,\mu,\cZ}:C^{\infty}_{P,\mu}\oplus \big(\bigoplus_{i=1}^m(\R^7\oplus T_{\Sigma_i}\cZ_i) \big)\to C^{\infty}_{P,\mu-1}$ is the operator $\widetilde \fL_{P,\mu,\cZ}$ under the identification map $\Theta_{P}^C$ from \autoref{def identification of normal bundle}, that is, 
	\begin{equation*}
	\widetilde \bD_{P,\mu,\cZ}(u,\hat p_i, \xi_i):=\Theta_{P}^C\widetilde \fL_{P,\mu,\cZ}((\Theta_{P}^C)^{-1}u,\hat p_i, \xi_i)=\bD_P u+ \Theta_{P}^C L_1(\hat p_i, \xi_i).\qedhere
 \end{equation*}\end{definition}
The following proposition computes the linearization of the non-linear map $\fF$ with respect to variations of the co-closed $G_2$-structures.
\begin{prop}\label{lem linearization 3rd component}The linearization of $\mathfrak F(\alpha,\cdot,p_1,\dots,p_m, 0,\dots,0)$ at $\phi$,
 $$L_2:T_{\phi}\sP\to C^{\infty}_{P_C,\mu-1}$$ satisfies the following: 
 for any $\hat \phi\in T_{\phi}\sP=\Omega^3(Y) $ and $w\in C^\infty_c(NP_C)$, we have
$$\inp{L_2(\hat \phi)}{w}_{L^2}=\int_{\Gamma_{\alpha}}\iota_w(\Upsilon_{P_C}^*\hat \psi)+\sum_{i=1}^m\int_{\Gamma_{\alpha}}\iota_w\mathcal L_{\Upsilon^i_*X_{\hat \phi,i}}(\Upsilon_{P_C}^*\psi),$$
where $\hat \psi:=\frac d{dt}\big|_{{t=0}}\psi_t$ with $\psi_t=*_{\phi_t}\phi_t$ and $\phi_t:=\phi+t\hat \phi$, and the vector field $X_{\hat \phi,i}:=v_{\phi+\hat \phi,p_i}$ from \autoref{eq vf 2}. 
\end{prop}
\begin{proof}Without loss of generality we assume that $m=1$ and $C=C_i$. Suppose $\hat \phi \in T_{\phi}\sP$ and for each $t$ small, $\tau_t:=(\phi_t, p, 0)$. 
For all $w\in C^\infty_c(NP_C)$ we have,	
\begin{equation*}\inp{L_2(\hat \phi)}{w}_{L^2}=\int_{\Gamma_{\alpha}}\frac d{dt}\big|_{{t=0}} \iota_w(\Upsilon_{P_C^{\tau_t}}^*\psi_t)=\int_{\Gamma_{\alpha}}\iota_w(\Upsilon_{P_C}^*\hat \psi)+\int_{\Gamma_{\alpha}}\iota_w\mathcal L_{\Upsilon_*X_{\hat \phi}}(\Upsilon_{P_C}^*\psi). \end{equation*}
 The proposition now follows from \autoref{def linearization L on cs asso } of $\fL_P$.
\end{proof}

The following definition introduces the $1$-parameter moduli space of CS associative submanifolds, along with the corresponding linear operators used to describe the local structure of this moduli space.
\begin{definition}\label{def 1 para moduli CS}
Let $\bsP$ be the space of paths $\bphi:[0,1]\to \sP$ that are smooth as section over $[0,1]\times Y$. Equip $\bsP$ with the $C^\infty$ topology. Define the $1$-parameter moduli space of CS associatives by the fiber product 
\begin{equation*}
	\bcM_{\operatorname{cs}}^\bphi:=[0,1]\times_\sP \cM_{\operatorname{cs}}.\qedhere
\end{equation*}
If $\bphi\in \bsP$ and $(t_0,P)\in \bcM_{\operatorname{cs}}^\bphi$, we denote $$\hat\phi:=\frac d{dt}\big|_{{t=t_0}}\phi_t,\ \ f_P:= L_2(\hat \phi )\in C^\infty(NP_C),,\ \ \hat f_P:= \Theta_{P}^C f_P \in C^\infty(NP)$$
where $L_2$ and $\Theta_{P}^C$ are as in \autoref{lem linearization 3rd component} and \autoref{def identification of normal bundle} , respectively. We define 
$$\widebar \fL_{P,\mu,\cZ}:\R\oplus C^{\infty}_{P_C,\mu}\oplus \big(\bigoplus_{i=1}^m(\R^7\oplus T_{\Sigma_i}\cZ_i) \big)\to C^{\infty}_{P_C,\mu-1}$$
by $\widebar \fL_{P,\mu,\cZ}(t,\tilde u)=\widetilde \fL_{P,\mu,\cZ} \tilde u+ t f_P$, where $\widetilde \fL_{P,\mu,\cZ}$ is defined in \autoref{def lin L, D with Z}. 

The operator $\widebar  \bD_{P,\mu,\cZ}:\R\oplus C^{\infty}_{P,\mu}\oplus \big(\bigoplus_{i=1}^m(\R^7\oplus T_{\Sigma_i}\cZ_i) \big)\to C^{\infty}_{P,\mu-1}$ is the operator $\widebar \fL_{P,\mu,\cZ}$ under the identification map $\Theta_{P}^C$, that is,
 $$\widebar  \bD_{P,\mu,\cZ}(t,\tilde u)=\widetilde \bD_{P,\mu,\cZ} \tilde u+ t \hat f_P,$$
 where $\widetilde \bD_{P,\mu,\cZ}$ is again defined in \autoref{def lin L, D with Z}
\end{definition}

The final ingredient we need before proving \autoref{thm moduli cs asso} about the local structure of the moduli space is a quadratic estimate, which we now proceed to establish.

\begin{definition}
	For all $\tau\in U_{\tau_0}$, the nonlinear map $Q_\tau: C^{\infty}_{\mu} (V_{P_C})\to C^{\infty}{(NP_C)}$ is defined  by 
	  \begin{equation*}
Q_\tau:=\mathfrak F_\tau-d\mathfrak F_{\tau_{|0}}-\mathfrak F_{\tau}(0).\qedhere 
 \end{equation*}
\end{definition}

\begin{prop}\label{prop cs quadratic estimate}For all $\tau\in U_{\tau_0}$ and $u, v\in C^{\infty}_{P_C,\mu}(V_{P_C})$, and $\eta \in C^{\infty}_{P_C,\mu}$ we have
\begin{enumerate}[(i)]
\item $\abs{d\mathfrak F_{\tau_{|u}}(\eta)-d\mathfrak F_{\tau_{|v}}(\eta)}\lesssim (w_{P_C,1}\abs{u-v}+\abs{\nabla^\perp (u-v)})(w_{P_C,1}\abs{\eta}+\abs{\nabla^\perp \eta})$.
\item $\abs{Q_\tau(u)-Q_\tau(v)}\lesssim (w_{P_C,1}\abs{u}+\abs{\nabla^\perp u}+w_{P_C,1}\abs{v}+\abs{\nabla^\perp v})(w_{P_C,1}\abs{u-v}+\abs{\nabla^\perp (u-v)}).$ 
\item $\norm{Q_\tau(u)}_{C^0_{P_C,\mu-1}}\lesssim \norm{Q_\tau(u)}_{C^0_{P_C,2\mu-2}}\lesssim\norm{u}_{C^1_{P_C,\mu}}^2$.
\end{enumerate}
\end{prop}
\begin{proof}Since $\tau$ is fixed we abuse notation and denote $\psi_\tau$ by $\psi$. 
For all $w\in C_c^\infty(NP_C)$ and $\eta \in C^{\infty}_{P_C,\mu}$ we write
$$\inp{d\mathfrak F_{\tau_{|u}}(\eta)-d\mathfrak F_{\tau_{|v}}(\eta)}{w}=\int_0^1\frac d{dt}\Big(d\mathcal F_{\tau_{|tu+(1-t)v}}(\eta)(w)\Big) dt$$ and using \autoref{prop linearization nonliear} this becomes
$$\int_0^1\Big(\frac d{dt}\int_{\Gamma_{tu+(1-t)v}} L_{\eta}\iota_w(\Upsilon_{P^\tau_C}^*\psi)\Big)dt=\int_0^1\int_{\Gamma_{tu+(1-t)v}} L_{(u-v)}L_{\eta}\iota_w(\Upsilon_{P^\tau_C}^*\psi)dt.$$ 
As, $[u-v, w]=0$ and $[\eta,w]=0$ the last expression is same as 
$$\int_0^1\int_{\Gamma_{tu+(1-t)v}} \iota_wL_{(u-v)}L_{\eta}(\Upsilon_{P^\tau_C}^*\psi)dt.$$
The required estimate in (i) now follows from \autoref{lem quadratic}. 
The estimate in (ii) follows from (i) after writing 
$$Q_\tau(u)-Q_\tau(v)=\int_0^1dQ_{\tau_{|tu+(1-t)v}}(u-v)dt=\int_0^1\big(d\mathfrak F_{\tau_{|tu+(1-t)v}}(u-v)-d\mathfrak F_{\tau_{|0}}(u-v)\big)dt.$$
Finally (iii) follows from (ii). Indeed, substituting $v=0$ we have
\[w_{P_C,2\mu-2}\abs{Q_\xi(u)}\lesssim w_{P_C,2\mu-2}(w_{P_C,1}\abs{u}+\abs{\nabla^\perp u})^2\lesssim( w_{P_C,\mu}\abs{u}+ w_{P_C,\mu-1}\abs{\nabla^\perp u})^2.\]
Since $\mu>1$ therefore $w_{P_C,\mu-1}\abs{Q_\xi(u)}\lesssim w_{P_C,2\mu-2}\abs{Q_\xi(u)}$. This completes the proof.
\end{proof}

\begin{prop}\label{prop Holder cs quadratic estimate}For all $\tau\in U_{\tau_0}$ and $u, v\in C^{k+1,\gamma}_{P_C,\mu}(V_{P_C})$ we have
$$\norm{Q_\tau(u)-Q_\tau(v)}_{C^{k,\gamma}_{P_C,\mu-1}}\lesssim\norm{Q_\tau(u)-Q_\tau(v)}_{C^{k,\gamma}_{P_C,2\mu-2}}\lesssim\norm{u-v}_{C^{k+1,\gamma}_{P_C,\mu}}\big(\norm{u}_{C^{k+1,\gamma}_{P_C,\mu}}+\norm{v}_{C^{k+1,\gamma}_{P_C,\mu}}\big).$$
\end{prop}
\begin{proof}Since $\tau$ is fixed we abuse notation and again denote $\psi_\tau$ by $\psi$.
With the above notation and appropriate product operation `$\cdot$', one can express $\mathcal L_u\mathcal L_v\psi$ formally as a quadratic polynomial 
$$\mathcal L_u\mathcal L_v\psi=O(f_1)\cdot u\cdot v+O(f_2)\cdot (u\cdot \nabla^\perp v+v\cdot \nabla^\perp u)+\psi \cdot \nabla^\perp u\cdot \nabla^\perp v$$
where $O(f_{1}):=\psi\cdot \nabla B+\psi\cdot B\cdot B+B\cdot\nabla\psi+\nabla^2\psi+R\cdot\psi$ and $O(f_2):=\nabla\psi+B\cdot\psi$. With this observation and a similar computations as in \autoref{lem quadratic} one can prove the proposition. 	
\end{proof}

  \begin{proof}[{\normalfont{\textbf{Proof of \autoref{thm moduli cs asso}}}}]By \autoref{prop Holder cs quadratic estimate}, \autoref{prop linearization nonliear} and \autoref{lem linearization 2nd component} we conclude that the extension of the map $\mathcal F_{\phi}:=\mathcal F(\cdot,\phi,\cdot)$ to weighted Hölder spaces
 $$\mathfrak F_{\phi,\cZ}:C^{2,\gamma}_{P_C,\mu}(V_{P_C}) \times \prod_{i=1}^mU_{p_i}\times \prod_{i=1}^m\cZ_i\cap \cI_{\Sigma_i} \to C^{1,\gamma}_{P_C,\mu-1}$$
is a well-defined smooth map. Here the rate $\mu$ is chosen as in \autoref{def CS weighted topo}. By \autoref{prop Fredholm Donaldson} we have that the linearization of $\mathfrak F_{\phi,\cZ}$ at $(\alpha, p_1,\dots,p_m, 0,\dots,0)$, $\widetilde \fL_{P,\mu,\cZ}$ is a Fredholm operator. Moreover, \autoref{prop index} implies that
$$\ind \widetilde \fL_{P,\mu,\cZ}=-\displaystyle\sum_{i=1}^m\Big(\frac{d_{-1,i}}{2}+\displaystyle\sum_{\lambda_i\in\mathcal D_{C_i}\cap(-1,1]}d_{\lambda_i}\Big)+\sum_{i=1}^m(7+\dim \cZ_i)=-\displaystyle\sum_{i=1}^m\operatorname{s-ind}(C_i).$$
 Applying the implicit function theorem to $\mathfrak F_{\phi,\cZ}$ and the elliptic regularity from \autoref{prop elliptic regularity} we obtain the existence of $\ob_{P,\cZ}$ in (i) (see \cite[Proposition 4.2.19]{Donaldson1990}) of the theorem except the following. If $u\in C^{2,\gamma}_{P_C,\mu}(V_{P_C})$ with $\mathfrak F_{\phi_0,\cZ}(u,\xi)=0$ for some $\xi\in \prod_{i=1}^mU_{p_i}\times \prod_{i=1}^m\cZ_i\cap \cI_{\Sigma_i}$, then $u\in C^{\infty}_{P_C,\mu}(V_{P_C})$. To prove this, we observe  
 $$0=\fL_P\mathfrak F_{\phi,\cZ}(u,\xi)=A_\xi(u,\nabla_{P_C}^\perp u) (\nabla_{P_C}^\perp)^2 u+ B_\xi(u,\nabla_{P_C}^\perp u)$$
 Since $A_\xi(u,\nabla_{P_C}^\perp u), B_\xi(u,\nabla_{P_C}^\perp u)\in C^{1,\gamma}_{P_C,\mu}$, by a weighted version of Schauder elliptic regularity for second order operator (similar to \autoref{prop elliptic regularity}) we obtain $u\in C^{3,\gamma}_{P_C,\mu}$. By repeating this argument we get higher regularity. This completes the proof of (i).
 
 The proof of (ii) is similar to (i). Indeed, we replace $\mathfrak F_{\phi,\cZ}$ by $\mathfrak F_{\bphi,\cZ}$ as follows. There exists an interval $I\subset [0,1]$ containing $t_0$ such that $\bphi(I)\subset U_{\phi_{t_0}}$ which was defined in \autoref{def CS weighted topo}. We define
 $$\mathfrak F_{\bphi,\cZ}:I\times C^{2,\gamma}_{P_C,\mu}(V_{P_C}) \times \prod_{i=1}^mU_{p_i}\times \prod_{i=1}^m\cZ_i\cap \cI_{\Sigma_i} \to C^{1,\gamma}_{P_C,\mu-1}$$
by $\mathfrak F_{\bphi,\cZ}(t,\cdot):=\mathfrak F_{\phi_t,\cZ}$, $t\in I$. The linearization of $\mathfrak F_{\bphi,\cZ}$ is $\widebar \fL_{P,\mu,\cZ}$ as in \autoref{def 1 para moduli CS}. The remaining proof is left to the reader.
 \end{proof}

 \begin{remark}\label{rmk CS two rates}If $\mu'$ were another choice in \autoref{def CS weighted topo} distinct from $\mu$, then  $u\in C^{\infty}_{P_C,\mu}(V_{P_C})$ satisfying  $\mathfrak F_\tau(u)=0$ implies $u\in C^{\infty}_{P_C,\mu^\prime} (V_{P_C})$. To see this we write ${d\mathfrak F_\tau}_{|0}(u)=-Q_\tau(u)-{\mathfrak F_\tau}(0)$. By \autoref{prop cs quadratic estimate} we see that ${d\mathfrak F_\tau}_{|0}$ is also a CS uniformly elliptic operator asymptotic to $\mathbf D_C$. Also we have $Q_\tau(u)\in C^{k,\gamma}_{P_C,2\mu-2}$ and ${\mathfrak F_\tau}(0)\in {C^{k,\gamma}_{P_C,1}}$. Therefore by \autoref{lem main Fredholm} we can conclude that $u\in C^{\infty}_{P_C,\mu^\prime} (V_{P_C})$.
 \end{remark}

  \subsection{Proof of genericity results: Floer's $C_\epsilon$ space and Taubes' trick }\label{subsection generic moduli CS asso}
  \begin{proof}[{{{\normalfont{\textbf{Proof of \autoref{thm generic moduli cs asso}}}}}}]
We prove that $\sP_{\operatorname{cs},\cZ}^{\operatorname{reg}}$ is comeager in $\sP$ in \autoref{lem coclosed comeager}. The proof of the fact that $\bsP_{\operatorname{cs},\cZ}^{\operatorname{reg}}$ is comeager in $\bsP$ is similar. Then (i) and (ii) in \autoref{thm generic moduli cs asso} are direct consequences of (i) and (ii) in \autoref{thm moduli cs asso}.
 \end{proof}
  To prove that $\sP_{\operatorname{cs},\cZ}^{\operatorname{reg}}$ is comeager in $\sP$ in \autoref{lem coclosed comeager}, we would like to use the Sard-Smale theorem applied to the map 
 $$\mathfrak F_\cZ:W^{k+1,p}_{P_C,\mu}(V_{P_C}) \times U_\phi\times \prod_{i=1}^mU_{p_i}\times \prod_{i=1}^m\cZ_i\cap \cI_{\Sigma_i} \to W^{k,p}_{P_C,\mu-1},$$
We have chosen Sobolev spaces rather than H\"older spaces as the former is separable. The reader may observe that all the analysis in \autoref{subsection moduli CS asso} will also go through with Sobolev spaces.  A serious problem here is that $U_{\phi}\subset \sP$ is not a Banach manifold. The standard way to deal with this is to consider $\sP_k$, the space of all $C^k$ co-closed $G_2$-structures on $Y$, to enlarge $U_{\phi}$ to a Banach manifold. This comes with a drawback that the map $\mathfrak F$ will have only finitely many derivatives and extra effort is required to check exactly how many in order to state the regularity of the moduli space. To avoid all this we will instead use Floer's $C_\epsilon$ space $\sP_\epsilon\subset \sP$ of co-closed $G_2$-structures. Since the Sard--Smale theorem yields the genericity results in $\sP_\epsilon$ instead of $\sP$ we use Taubes' trick of exhausting the moduli spaces by countably many compact subsets (see \autoref{def taubes CS asso with N}) to conclude the genericity results in $\sP$. For more details on this idea, see \cite{Wendl2021}.
	
 
\begin{definition}\label{def taubes CS asso with N}For every $N,m\in \N$ and $\cZ=\prod_{i=1}^m\cZ_i$, we define $\cM_{\ocs,N,\cZ}^\phi\subset \cM_{\ocs,\cZ}^\phi$ to be the set of all CS associative submanifolds $P$ in $(Y,\phi)$ with singularities at $p_i$, cones $C_i$ with links $\Sigma_i\in \cZ_i$, rates $\mu_i\in [1+\frac1N,2]$, $G_2$-coordinate systems $\Upsilon^i:B(0,R)\to Y$ and embeddings $\Upsilon^i_{P}:(0,2\epsilon_0)\times \Sigma_i\to B(0,R)$, $i=1,\dots,m$ satisfying
	\begin{itemize}
	\item $1>R\geq4\epsilon_{0}\geq \frac 1N,\ \abs{\nabla^k\operatorname{II}_{\Sigma_i}}\leq N,{\operatorname{Vol}(\Sigma_i)}\leq N$ for all $k=0,\dots 3$,
	\item $d_{\cM^{\operatorname{hol}}}(\Sigma_i, \overline{ \cZ_i}\setminus  \cZ_i )\geq \frac 1N,$
	\item $\abs{\nabla^k\operatorname{II}_{K_P}}\leq N$ for all $k=0,\dots 3$ and ${\operatorname{Vol}(P)}\leq N$,
	\item $\abs{\nabla^k\Upsilon^i}\leq N$ and $\abs{(\nabla^\perp_{C})^k(\Psi^i_{P}-\iota)}\leq Nr^{\mu_i-k}$ for all $k=0,\dots 3$.
 \item $N\abs{x_1-x_2}\geq d_{\Sigma_i}(x_1,x_2), \forall x_1,x_2\in \Sigma_i$ and $Nd_Y(p_1,p_2)\geq d_P(p_1,p_2), \forall p_1,p_2\in P$.
	\end{itemize}
 
 We define $\sP_{\operatorname{cs},N,\cZ}^{\operatorname{reg}}\subset \sP$ to be the set of all $\phi\in \sP$ for which every $P\in \cM_{\ocs,N,\cZ}^\phi$ has the property that $\widetilde \bD_{P,\mu,\cZ}$ is surjective. 
  \end{definition}
\begin{remark}Obviously,
\begin{equation*}
	\sP_{\operatorname{cs},\cZ}^{\operatorname{reg}}=\bigcap_{N\in \N} \sP_{\operatorname{cs},N,\cZ}^{\operatorname{reg}}\subset \sP.
	\qedhere
\end{equation*}
\end{remark}
 The following lemma about convergence of CS associative submanifolds can be deduced from \cite{Shen1995}.
 \begin{lemma}\label{lem CS convergence}
	Let $\phi_n\in \sP$ be a sequence of $G_2$-structures on $Y$ converging to $\phi$ in $\sP$ with $C^\infty$ topology. Fix $N\in \N$ and $\cZ$, and let $P_n\in \cM_{\ocs,N,\cZ}^{\phi_n}$ be a sequence of CS associative submanifolds. Then there is a subsequence of $P_n$ converging to a CS associative $P\in \cM_{\ocs,N,\cZ}^{\phi}$ with weighted $C^\infty$-topology defined in \autoref{def CS weighted topo}. 	  
\end{lemma}

\begin{lemma}\label{lem coclosed N open}
	For every $N\in \N$ and $\cZ$, $\sP_{\operatorname{cs},N,\cZ}^{\operatorname{reg}}$ is open in $\sP$. 
\end{lemma}
\begin{proof}
	The complement of $\sP_{\operatorname{cs},N,\cZ}^{\operatorname{reg}}$ is closed. Indeed, it follows from \autoref{lem CS convergence} and the fact that surjectivity is an open condition.
\end{proof}
\begin{lemma}\label{lem coclosed comeager}For every $N\in \N$  and $\cZ$, $\sP_{\operatorname{cs},N,\cZ}^{\operatorname{reg}}$ is dense in $\sP$. Hence  $\sP_{\operatorname{cs},\cZ}^{\operatorname{reg}}$ is comeager in $\sP$.
\end{lemma}
To prove \autoref{lem coclosed comeager} we use Floer's $C_\epsilon$ space.
\begin{definition}Let $E\to M$ be a vector bundle over a compact Riemannian manifold $M$ with or without boundary. For each integer $k\geq 0$, we denote by $C^k(E)$ the Banach space of $C^k$-sections of $E$. 
	Let $\cE$ be the set of all sequences $(\epsilon_k)_{k=0}^\infty$ of positive real numbers with $\epsilon_k\to 0$. For each $\epsilon \in \cE $, \textbf{Floer's $C_\epsilon$ space} of sections of $E$ is defined by
	$$C_\epsilon(E):=\{s\in C^\infty(E):\norm{s}_{C_\epsilon}:=\sum_{k=0}^\infty\epsilon_k\norm{s}_{C^k}<\infty\}.$$
A pre-order $\prec$ on $\cE$ is defined by 
\begin{equation*}
\epsilon\prec \epsilon^\prime\ \text{iff}\ \limsup_{k\to\infty}\frac{\epsilon_k}{\epsilon^\prime_k}<\infty.\qedhere
\end{equation*}
 \end{definition}
	\begin{remark}
	As $M$ is compact, different $C^k$ norms for different smooth atlases of $M$ are equivalent but the norm on Floer's $C_\epsilon$ space is an infinite sum, therefore it might not be equivalent. But this not a big concern as no statement in any theorem will mention this space, it will only be used inside the proofs and where we can fix an atlas.
\end{remark}

\begin{lemma}[{\cite[Appendix B]{Wendl2016a}}]\label{lem Floer's space} The Floer's $C_\epsilon$ spaces $C_\epsilon(E)$ have the following properties:
\begin{enumerate}[(i)]
\item $C_\epsilon(E)$ is a separable Banach space and $C_\epsilon(E)\hookrightarrow C^\infty(E)$ is continuous. Moreover, $C_{\epsilon^\prime}(E)\hookrightarrow C_\epsilon(E)$ is a continuos embedding if $\epsilon\prec \epsilon^\prime$.
\item  For every countable subset $\Q$ of $C^\infty(E)$ there exists an $\epsilon_0\in \cE$ such that $\Q\subset C_\epsilon(E)$ for all $\epsilon\prec \epsilon_0$. In particular, 
\begin{equation*}
 C^\infty(E)=\bigcup_{\epsilon\in \cE } C_\epsilon(E).
 \qedhere	
 \end{equation*}
\end{enumerate}
\end{lemma}

To prove \autoref{lem coclosed comeager}, we will also require the following lemma. It essentially asserts the existence of a finite-dimensional subfamily of $\sP$ in which the CS associative submanifold $P$ is unobstructed, so that the moduli space of CS associatives in a neighbourhood of $P$ within this subfamily forms a smooth manifold. Furthermore, by Sard’s theorem, for a generic member of this subfamily, the CS associatives in this neighbourhood of $P$ are all unobstructed.

\begin{lemma}[{\citet[Proposition A.2]{Doan2017d}}]\label{lem main generic DW}
	Let $P$ be a CS associative in $(Y,\phi)$ with singularities at $p_i$ and rates $\mu_i\in (1,2]\setminus \mathcal D_{C_i}$. Then for all $w\in \ker \fL_{P,-2-\mu}\cong \coker \fL_{P,\mu}$ there exists a $3$-form $\hat \phi\in T_\phi\sP$ supported away from $p$ such that
	$$\inp {L_2(\hat \phi)}{w}_{L^2}\neq 0, \ L_2 \ \text{is defined in \autoref{lem linearization 3rd component}}. $$ 
\end{lemma}
\begin{proof}By unique continuation, there exists an open set $U$ in the interior of $K_P$ on which $w$ does not vanish identically. Let $V$ be a tubular neighbourhood of $U$ in $Y$. Choose a real valued smooth function $f$ supported in $V$ such that $df(w)\geq 0$ on $U$ and there exists a point in $U$ where $df(w)>0$. Let $\Theta\in \Omega^3(Y)$ be an extension of the $3$-form $\vol_U$ and $\iota_wd\Theta_{|V}=0$. Then there exists a $3$-form $\hat \phi\in T_\phi\sP$ supported in $V$ such that $\hat \psi=d(f\Theta)\in \Omega^4(Y)$, where $\hat \psi$ is as in \autoref{lem linearization 3rd component}. By \autoref{lem linearization 3rd component}, we have
\begin{equation*}
	\inp{L_2(\hat \phi)}{w}_{L^2}=\int_{\Gamma_{\alpha}}\iota_w(\Upsilon_{P_C}^*\hat \psi)=\int_U df(w)\vol_U>0.
\qedhere \end{equation*}
\end{proof}

\begin{proof}[{{{\normalfont{\textbf{Proof of \autoref{lem coclosed comeager}}}}}}] For each $\epsilon \in \cE $, let $\Omega_\epsilon^3(Y)$ be the Floer's $C_\epsilon$ space of smooth $3$-forms on $Y$. We define 
$$\sP_\epsilon:=\sP\cap \Omega_\epsilon^3(Y)\ \ \text{and} \ \ \cM_{\ocs,\epsilon,\cZ}:=\{(\phi,P)\in \cM_{\ocs,\cZ}: \phi\in \sP_\epsilon, P\in \cM^\phi_{\ocs,\cZ}\}.$$
 Suppose $(\phi,P)\in \cM_{\ocs,\epsilon,\cZ}$. Let $p_i$, $i=1,\dots,m$ be the singular points of $P$.  Choose $U_{\tau_0}$ and $\mu$ as in \autoref{def CS weighted topo}. Set $$U_{\tau_0,\epsilon,\cZ}:=\big(U_{\phi} \cap \sP_\epsilon\big)\times \prod_{i=1}^m U_{p_i}\times \prod_{i=1}^m\cZ_i.$$
Let $\mathfrak F_{\epsilon,\cZ}: W^{2,p}_{\mu} (V_{P_C})\times U_{\tau_0,\epsilon,\cZ}\to W^{1,p}_{P_C,\mu-1}$ be the restriction of $\mathfrak F$. We define
 $$\cM^{\operatorname{reg}} _{\ocs,\epsilon,\cZ}:=\{(\phi,P)\in \cM_{\ocs,\epsilon,\cZ}: {d\mathfrak F_{\epsilon,\cZ}}_{|\tau_0}\ \text{is surjective} \}.$$ 
  By the Implicit function theorem \cite[Theorem A.3.3]{McDuff2012} we obtain that $\cM^{\operatorname{reg}} _{\ocs,\epsilon,\cZ}$ is a separable Banach manifold. Moreover by the Sard--Smale theorem \cite[Lemma A.3.6]{McDuff2012} we see that the canonical projection map $\pi_{\epsilon,\cZ}:\cM^{\operatorname{reg}} _{\ocs,\epsilon,\cZ}\to \sP_\epsilon$ is a Fredholm map. Therefore again by \cite[Theorem A.3.3]{McDuff2012} we conclude that there exists a comeager subset $\sP_{\epsilon,\operatorname{cs},\cZ}^{\operatorname{reg}}\subset \sP_\epsilon$ having the property that for each $\phi\in \sP_{\epsilon,\operatorname{cs},\cZ}^{\operatorname{reg}}$ the linear operator $\widetilde \bD_{P,\mu,\cZ}$ is surjective for all $P$ with $(\phi,P)\in \cM^{\operatorname{reg}} _{\ocs,\epsilon,\cZ}$.
  
  With this preparation we are now ready to prove the lemma. Suppose $\phi_0\in \sP$.  We can choose $\epsilon$ sufficiently small so that $\phi_0\in \sP_\epsilon$. Since $\sP_{\epsilon,\operatorname{cs},\cZ}^{\operatorname{reg}}$ is dense in $\sP_\epsilon$, there exists a sequence $\phi_n\in \sP_{\epsilon,\operatorname{cs},\cZ}^{\operatorname{reg}}$ converging to $\phi_0$ in $C_\epsilon$-topology and hence in $C^\infty$-topology. Our claim is that $\phi_n\in \sP_{\operatorname{cs},N,\cZ}^{\operatorname{reg}} $ for all $n$ sufficiently large. If not, then there is a subsequence $P_n\in \cM_{\ocs,N,\cZ}^{\phi_n}$ such that $\widetilde \bD_{P_n,\mu_n,\cZ}$ is not surjective. By \autoref{lem CS convergence} we obtain that $P_n$ converges to a CS associative $P$ in $(Y,\phi_0)$.  Hence, by \autoref{lem main generic DW} and \autoref{lem Floer's space}, $(\phi_0,P)\in \cM^{\operatorname{reg}} _{\ocs,\epsilon,\cZ}$ for all sufficiently small $\epsilon$. Therefore  $(\phi_n,P_n)\in \cM^{\operatorname{reg}} _{\ocs,\epsilon,\cZ}$ for all $n$ sufficiently large which contradicts to the fact that  $\widetilde \bD_{P_n,\mu_n,\cZ}$ is not surjective.    
 \end{proof}

\section{Desingularizations of CS associative submanifolds}\label{section desing cs asso}

 Let $(Y,\phi)$ be a co-closed $G_2$-manifold.  We will glue a conically singular (CS) associative submanifold $P$ in $(Y,\phi)$ with singularity at one point and rescaled asymptotically conical (AC) associative submanifolds $\epsilon L$ in $\R^7$ with $\epsilon>0$ to construct closed approximate associative submanifolds ~$P_\epsilon$. We have seen in \autoref{thm closed asso} that for generic co-closed $G_2$ structures the moduli space of closed associative submanifolds is $0$-dimensional and therefore we can not expect to deform the $P_\epsilon$ to associative submanifolds $\tilde{P_\epsilon}$ ({desingularization}) in $(Y,\phi)$. In this section we will see that under some hypothesis we can do this in a $1$-parameter family of co-closed $G_2$-structures\footnote{One can generalize this by using similar analysis to multiple isolated points and glue that many asymptotically conical associative submanifolds in $\R^7$ but in that case one needs to desingularize in a higher dimensional  parameter space of co-closed $G_2$ structures, which we will not cover in this article.}.

\subsection{Approximate associative desingularizations}
Let $P$ be a CS associative submanifold of a co-closed $G_2$-manifold $(Y,\phi)$ with a singularity at ~$p$, rate $\mu\in(1,2]$ and  cone $C$ as in \autoref{def CS asso}. Let $L$ be an AC associative submanifold in $\R^7$ with the same cone $C$ and rate $\nu<1$ as in \autoref{def AC associative}. The real numbers $R, R_\infty$, $\epsilon_0$ are as in \autoref{def AC associative} and \autoref{def CS asso}. Let $P_C$, $NP_C$, $K_{P_C}$ be as in \autoref{def PC} and $L_C$, $NL_C$, $K_{L_C}$ be as in \autoref{def LC}.

Let $\bphi\in \bsP$ be a path of co-closed  $G_2$-structures and $t_0\in (0,1)$ such that $\bphi(t_0)=\phi$. Set $\phi_{t}:=\bphi(t)$. Let $U_{\tau_0}$ be as in \autoref{eq Utau0}.  Let $T>0$ be sufficiently small such that for each $t\in (t_0-T,t_0+T)$ we have $$\tau_t:=(\phi_t,p,0)\in U_{\tau_0}.$$ 

Let $\Upsilon^{\tau_t}$, $P_C^{\tau_t}$ and $\Upsilon_{P^{\tau_t}_{C}}$  be as in \autoref{def smooth family EC}.  We first glue $P_C^{\tau_t}$ and $\epsilon L_C$ to obtain a submanifold $P_{\epsilon,t,C}$ in $(Y, \phi_t)$, along with a tubular neighbourhood of it, inside which the approximate associative desingularization will be defined.

\begin{definition}[$P_{\epsilon,t,C}$:= gluing of $P_C^{\tau_t}$ and $\epsilon L_C$]For any sufficiently small $\epsilon>0$, we define a closed $3$-dimensional submanifold $P_{\epsilon,t,C}$ of $Y$ as follows. For a real number $0<q<1$ we define a real number, $\delta:=(\epsilon R_\infty)^{q}$. Then for sufficiently small $\epsilon$ we have 
 $$\epsilon R_\infty< 2\epsilon R_\infty<\delta<2\delta<\delta^{\frac 12}<2\delta^{\frac 12}<\epsilon_0< 2\epsilon_0<R.$$
We use the notation $A(a,b)$ for the annulus $\{x\in \R^7:a<\abs{x}<b\}$. For any $t\in (t_0-T,t_0+T)$, we define $$P_{\epsilon,t,C}:=P^+_{\epsilon, t, C} \bigcup  P^-_{\epsilon,t, C} $$
where  $$ P^+_{\epsilon, t, C}:= \Upsilon^{\tau_t}\big(\epsilon K_{L_C} \cup  (C \cap A(\epsilon R_\infty, 2\delta))\big)   \ \ \text{and}\ \ P^-_{\epsilon, t,  C}:=\Upsilon^{\tau_t} (C \cap A(\delta, 2\epsilon_0)) \cup K_{P_C}.$$
Here $(\Upsilon^{\tau_t})^{-1}(P^+_{\epsilon, t, C})\subset \epsilon L_C$ and $P^-_{\epsilon,t, C}\subset P_C^{\tau_t}$. 
\end{definition}
\begin{notation}
	If $t=t_0$ then we will now on omit the subscript $t$. 
\end{notation}
\begin{definition}[Tubular neighbourhood map of $P_{\epsilon,t,C}$]
The normal bundles $N(\epsilon L_C)$ and $NP_C$ can be glued to get a normal bundle 
$$NP_{\epsilon,C}=NP^+_{\epsilon,C}\cup NP^-_{\epsilon,C}.$$
The scaling $s_\epsilon:\R^7\to \R^7$, $x\mapsto \epsilon x$ induces an isomorphism ${s_\epsilon}_*:NL_C\to N(\epsilon L_C)$ defined by $${s_\epsilon}_*\nu(x):=\epsilon\nu(\epsilon^{-1}x).$$
We have a tubular neighbourhood map $$\Upsilon_{P_{\epsilon,t,C}}:=\Upsilon_{P^+_{\epsilon,t,C}}\cup \Upsilon_{P^-_{\epsilon,t,C}}:V_{P_{\epsilon,C}}\to Y,$$ where $(\Upsilon^{\tau_t})^{-1}\circ \Upsilon_{P^+_{\epsilon,t,C}} \circ \Upsilon^{\tau_t}_*$ is the restriction of $s_\epsilon\circ \Upsilon_{L_C}\circ {s_\epsilon}_*^{-1}$ and $\Upsilon_{P^-_{\epsilon,t,C}}$ is the restriction of $\Upsilon_{P^{\tau_t}_{C}}$.
\end{definition}

\begin{definition}[Approximate associative desingularization]\label{def approx asso cs}
Let $\alpha\in  C^\infty(NP_C)$ and $\beta \in C^\infty(NL_C)$ be as in \autoref{def section alpha for P} and \autoref{def section beta for L} representing $P$ and $L$, respectively.

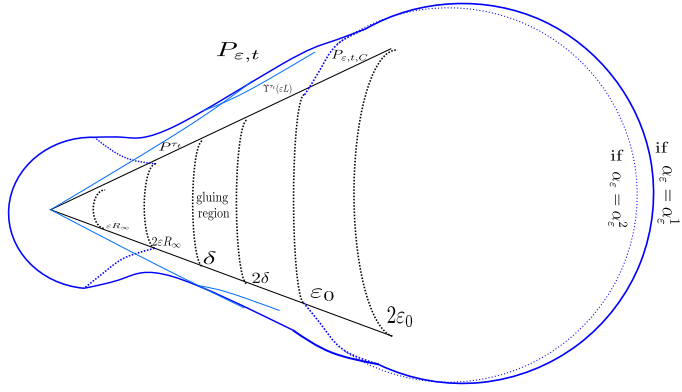
 We define the \textbf{approximate associative desingularization} $P^1_{\epsilon,t}$  by
$$ P^1_{\epsilon,t}:=\Upsilon_{P_{\epsilon,t,C}}(\alpha^1_{\epsilon}), $$
where $\alpha^1_\epsilon \in C^\infty(V_{P_{\epsilon,C}})$ is
 \begin{equation}\label{eq alpha 2 epsilon}
		\alpha^1_\epsilon:=\rho_{\delta} \Upsilon_*({s_\epsilon}_* \beta)+ (1-\rho_{\delta})\alpha. \qedhere
\end{equation}
\end{definition}
\begin{definition}[Improved approximate associative desingularization]\label{def improved approx asso cs}
Let $\alpha\in  C^\infty(NP_C)$ and $\beta \in C^\infty(NL_C)$ be as in \autoref{def approx asso cs}. Suppose there exists $\lambda_0<\nu$ such that 
\begin{enumerate}[(i)]
\item  $(\lambda_0,\nu]\cap \cD=\{\lambda_1,\dots,\lambda_l\}$ with $\lambda_1<\dots<\lambda_l,$	
\item there exist $s_0>0$ and $\alpha_i\in \ker \fL_{P, \lambda_i}$, $\beta_i\in V_{\lambda_i}$ $i=1,\dots, l$ such that 
 \[\abs{(\nabla^\perp_{C})^k(\beta-\sum_{i=1}^l\beta_i)}=O(r^{\max\{\lambda_0,2\nu-1\}-k})	\ \  \text{as} \ \ r\to \infty \ \ \text{for all}\ \  k\in \N\cup\{0\}\]
 and
  \begin{equation*}\abs{(\nabla^\perp_{C})^k(\Upsilon_{P_C}^*\alpha_i-\beta_i)}=O(r^{\lambda_i+s_0-k})	\ \  \text{as} \ \ r\to 0 \ \ \text{for all}\ \  k\in \N\cup\{0\}.\end{equation*}  
\end{enumerate}

We then define the \textbf{improved approximate associative desingularization} $P^2_{\epsilon,t}$  by
$$ P^2_{\epsilon,t}:=\Upsilon_{P_{\epsilon,t,C}}(\alpha^2_{\epsilon}), $$
where
\begin{equation}\label{eq alpha 2 epsilon}
		\alpha^2_\epsilon:=\rho_{\delta} \Upsilon_*({s_\epsilon}_* \beta)+ \Big(1-\rho_{\delta}\Big) \Big(\sum_{i=1}^l\epsilon^{1-\lambda_i}\alpha_i+\alpha \Big).\qedhere\end{equation}
\end{definition}

\begin{notation}We use the following notation:
\begin{equation*}
P_{\epsilon,t}:= P^1_{\epsilon,t}\  \text{or}\  P^2_{\epsilon,t} \qandq  \alpha_\epsilon:=\alpha^1_\epsilon \ \text{or}\ \alpha^2_\epsilon.
\end{equation*}
Set, 
\[\nu_0:=\max\{\lambda_0,2\nu-1\} \qandq  \fc_q:=\min\{(1-q)(1-\nu),q(\mu-1)\}.\]
Note that, \[\epsilon^{\fc_q} \lesssim \max\{\epsilon^{1-\nu}\delta^{\nu-1},\delta^{\mu-1}\}\lesssim \epsilon^{\fc_q}.\qedhere\]
	\end{notation}
	
Since the above approximate associative desingularization is not associative, we aim to find a genuine associative submanifold near it. In other words, our goal is to find a zero of the following nonlinear map.

 \begin{definition}\label{def nonlinear map with epsilon}
The map  $\tilde \fF_\epsilon: C^{\infty}(V_{P_{\epsilon,C}})\times (t_0-T,t_0+T)\to C^{\infty}{(NP_{\epsilon,C})}$ is defined by  
	 \begin{equation*}
	 	\inp{\tilde \fF_\epsilon(u,t)}{w}_{L^2}:=\int_{\Gamma_{u}}\iota_{w}\Upsilon_{P_{\epsilon,t,C}}^*\psi_t
	 \end{equation*}
	 where $t\in (t_0-T,t_0+T)$, $u\in C^{\infty}(V_{P_{\epsilon,C}})$ and $w\in C^\infty (NP_{\epsilon,C})$. The notation $w$ in the integrand is the extension vector field of $w\in C^\infty(NP_{\epsilon,C})$ in the tubular neighbourhood as in \autoref{notation u}.
	  \end{definition}
\begin{notation}
We  use the following notation:
\[\mathfrak F_{\epsilon,t}:=\tilde \fF_\epsilon(\cdot,t).\]
 If $t=t_0$ we omit the subscript $t$.	
\end{notation}

 \begin{definition}\label{def decomposition of F epsilon nonlinear map}
The map $\tilde \fF_{\epsilon}: C^{\infty} (V_{P_{\epsilon,C}})\times (t_0-T,t_0+T)\to C^{\infty}{(NP_{\epsilon,C})}$ can be written as 
$$\tilde \fF_{\epsilon}(u,t):=\mathfrak L_{P_\epsilon}u+(t-t_0)f_\epsilon+ \tilde Q_\epsilon(u,t)+e_\epsilon,$$ where \[e_\epsilon:=\mathfrak F_\epsilon(\alpha_\epsilon)\in C^{\infty}{(NP_{\epsilon,C})}, \quad f_\epsilon:=\frac {d\fF_{\epsilon,t}(\alpha_\epsilon)}{dt}\Bigr\rvert_{t = t_0}\in C^{\infty}{(NP_{\epsilon,C})},\]
and 
$\mathfrak L_{P_\epsilon}$ is the \textbf{linearization} of $\mathfrak F_{\epsilon}$ at $\alpha_\epsilon\in C^{\infty} (V_{P_{\epsilon,C}})$, that is, \begin{equation*}
\mathfrak L_{P_\epsilon}:={d\mathfrak F_{\epsilon}}_{|\alpha_\epsilon}:C^{\infty}{(NP_{\epsilon,C})}\to C^{\infty}{(NP_{\epsilon,C})}.\qedhere
\end{equation*}
 \end{definition}
The following proposition concerns the self-adjointness of $\mathfrak{L}_{P_\epsilon}$, which will be crucial when we attempt to invert it. The obstruction to this inversion can be expressed in terms of (approximate) kernel elements, which leads to the hypothesis in the desingularization theorem.

 \begin{prop} \label{self adjoint Linearization L_epsilon}$\mathfrak L_{P_\epsilon}$ is a formally self adjoint elliptic operator for all sufficiently small $\epsilon$. 
\end{prop}

To prove this, we need the following facts about AC associative submanifolds, which have already been observed in the case of CS associative submanifolds.

\begin{definition}\label{def linearization L on AC asso }Let $\beta \in C^\infty(NL_C)$ be as in \autoref{def section beta for L} representing the AC associative $L$. We define the differential operator $\mathfrak L_L:C_c^\infty(NL_C)\to C_c^\infty(NL_C)$ by 
$$\mathfrak L_Lu:=d\mathfrak F^{AC}_{|\beta}(u),$$
where the map  $\fF^{AC}: C^{\infty}_{L,\nu}(V_{L_{C}})\to C^{\infty}{(NL_{C})}$ is defined by  
	 \begin{equation*}
	 	\inp{\fF^{AC}(u)}{w}_{L^2}:=\int_{\Gamma_{u}}\iota_{w}\Upsilon_{L_{C}}^*\psi_e,\ \  \ u\in C^\infty C^{\infty}_{L,\nu}(V_{L_{C}}), w\in C_c^\infty(NL_C).
	\qedhere \end{equation*}
 \end{definition}
 \begin{remark}$\mathfrak L_L$ is a AC uniformly elliptic operator asymptotic to $\bD_C$. Moreover, there is a identification map $\Theta^C_L:NL_C\to NL$ similar to $\Theta^C_P$ (see \autoref{def identification of normal bundle}, \autoref{def linearization D on cs asso}) such that
  \[\bD_L=\Theta^C_L\fL_L(\Theta^C_L)^{-1}.
 \qedhere\]
 	 \end{remark}

\begin{proof}[{Proof of \autoref{self adjoint Linearization L_epsilon}}] $\mathfrak L_{P_\epsilon}$ is a formally self adjoint operator follows from same arguments as in the proof of \autoref{prop formal self adjoint}. It remains to prove that $\mathfrak L_{P_\epsilon}$ is an elliptic operator.  We denote the restrictions of $\mathfrak L_{P_\epsilon}$ over $P^\pm_{\epsilon,C}$ by $\mathfrak L_{P_\epsilon^\pm}$. By \autoref{prop cs quadratic estimate} we obtain that  
$$\mathfrak L_{P_\epsilon^-}={d\mathfrak F_{\tau_0}}_{|\alpha}+O(\epsilon^{\fc_q})=\mathfrak L_{P}+O(\epsilon^{\fc_q}).$$
Similar to \autoref{prop cs quadratic estimate} for AC associatives in $\R^7$ and the fact that $\Upsilon^*\psi=(\Upsilon^*\psi)(0)+O(r)$, we get 
$$\mathfrak L_{P_\epsilon^+}=\Upsilon_*\mathfrak L_{\epsilon L}\Upsilon_*^{-1}+O(\epsilon^{\fc_q})=\epsilon^{-1}(\Upsilon_*{s_\epsilon}_*)\mathfrak L_{L}(\Upsilon_*{s_\epsilon}_*)^{-1}+O(\epsilon^{\fc_q}).$$
Since $\mathfrak L_{P}$ and $\mathfrak L_{L}$ are elliptic, therefore $\mathfrak L_{P_\epsilon}$ is an elliptic operator for all sufficiently small $\epsilon$. 
 \end{proof}

\subsection{Weighted function spaces and estimates} 
In this subsection, we introduce the weighted function spaces on $P_{\epsilon,C}$, inspired by the corresponding function spaces on its CS and AC sides. These spaces will be used to prove the desingularization theorem. The theorem requires error estimates, linear estimates, and quadratic estimates with these weighted function spaces—all of which are established in this subsection.

\begin{definition}For each $l\in \R$, a \textbf{weight function} $w_{\epsilon,l}:P_{\epsilon,C}\to(0,\infty)$
is any smooth function on $P_{\epsilon,C}$ such that if $x=\Upsilon(r,\sigma) \in \Upsilon(B(0,R))$ then 
	$$w_{\epsilon,l}(x)=(\epsilon+r)^{-l}.$$
Let  $k\geq0$,\ $\gamma\in (0,1)$. For a continuous section $u$ of $NP_{\epsilon,C}$ we define the \textbf{weighted} $L^\infty$ \textbf{norm} and the \textbf{weighted Hölder semi-norm} respectively by 
	$$\norm{u}_{L^\infty_{\epsilon,l}}:=\norm{w_{\epsilon,l}u}_{L^\infty({NP_{\epsilon,C}})},\ \ [u]_{C_{\epsilon,l}^{0,\gamma}}:=[w_{\epsilon,l-\gamma}u]_{C^{0,\gamma}({NP_{\epsilon,C}})}.$$ 
	 For a continuous section $u$ of $NP_{\epsilon,C}$ with $k$ continuous derivatives we define the \textbf{weighted $C^k$ norm} and the \textbf{weighted Hölder norm} respectively by
	$$\norm{u}_{C^{k}_{\epsilon,l}}:=\sum_{j=0}^{k}\norm{(\nabla_{{P_{\epsilon,C}}}^\perp)^ju}_{L^\infty_{\epsilon,l-j}},\ \ \norm{u}_{C^{k,\gamma}_{\epsilon,l}}:=\norm{u}_{C^{k}_{\epsilon,l}}+[(\nabla_{{P_{\epsilon,C}}}^\perp)^ku]_{C^{0,\gamma}_{\epsilon,l-k}}.$$ 		
We define the \textbf{weighted Hölder space} $C^{k,\gamma}_{\epsilon,l}$, the \textbf{weighted $C^k$-space $C^{k}_{\epsilon,l}$} and the \textbf{weighted $L^\infty$-space $L^{\infty}_{\epsilon,l}$} to be the $C^{k,\gamma}$, $C^{k}$ and $L^{\infty}$ spaces with the weighted Hölder norm $\norm{\cdot}_{C^{k,\gamma}_{\epsilon,l}}$, weighted $C^k$-norm $\norm{\cdot}_{C^{k}_{\epsilon,l}}$ and weighted $L^\infty$-norm $\norm{\cdot}_{L^{\infty}_{\epsilon,l}}$ respectively. 
\end{definition}

\begin{prop}[{\textbf{Schauder estimate}}]\label{prop Schauder estimate} There exists a constant $C>0$ such that for any sufficiently small $\epsilon>0$ and for all $u\in C_{\epsilon, l}^{k+1,\gamma}$ we have	
$$\norm{u}_{C_{\epsilon, l}^{k+1,\gamma}}\leq C \Big(\norm{\mathfrak L_{P_\epsilon}u}_{C_{\epsilon, l-1}^{k,\gamma}}+\norm{u}_{L_{\epsilon, l}^\infty}\Big).$$
\end{prop}
\begin{proof}For $u\in C_{\epsilon, l}^{k+1,\gamma}$, define $u_\pm$ by restricting $u$ over $P^\pm_{\epsilon,C,}$. Using the  Schauder estimates in \autoref{prop elliptic regularity} and \autoref{prop cs quadratic estimate}  we obtain 
$$\norm{u_-}_{C_{\epsilon, l}^{k+1,\gamma}} \lesssim \norm{u_-}_{C_{P, l}^{k+1,\gamma}}\lesssim \norm{\mathfrak L_{P}u_-}_{C_{P, l-1}^{k,\gamma}}+\norm{u_-}_{L_{P, l}^\infty} \lesssim \norm{\mathfrak L_{P_\epsilon}u_-}_{C_{\epsilon, l-1}^{k,\gamma}}+\epsilon^{\fc_q}\norm{u_-}_{C_{\epsilon, l}^{k+1,\gamma}}+\norm{u_-}_{L_{\epsilon, l}^\infty} $$
Similarly
\begin{align*}
	\norm{u_+}_{C_{\epsilon, l}^{k+1,\gamma}} \lesssim \epsilon^{-l+1}\norm{{s_\epsilon}^{-1}_*\Upsilon_*^{-1}u_+}_{C_{L, l}^{k+1,\gamma}}
	&\lesssim \epsilon^{-l+1}\norm{\mathfrak L_{L}{s_\epsilon}^{-1}_*\Upsilon_*^{-1}u_+}_{C_{L, l-1}^{k,\gamma}}+\epsilon^{-l+1}\norm{{s_\epsilon}^{-1}_*\Upsilon_*^{-1}u_+}_{L_{L, l}^\infty}\\
	&\lesssim \epsilon^{-1}\norm{\Upsilon_*{s_\epsilon}_*\mathfrak L_{L}{s_\epsilon}^{-1}_*\Upsilon_*^{-1}u_+}_{C_{\epsilon, l-1}^{k,\gamma}}+\norm{u_+}_{L_{\epsilon, l}^\infty}\\
	&\lesssim \norm{\mathfrak L_{P_\epsilon} u_+}_{C_{\epsilon, l-1}^{k,\gamma}}+\epsilon^{\fc_q}\norm{u_+}_{C_{\epsilon, l}^{k+1,\gamma}}+\norm{u_+}_{L_{\epsilon, l}^\infty}.
	\qedhere
\end{align*}

\end{proof}

\begin{prop}[{\textbf{Error estimate}}]\label{prop error estimate with epsilon}
 The error $e_\epsilon=\mathfrak F_\epsilon(\alpha_\epsilon)$ can be estimated as follows:
\begin{equation*}
\norm{\mathfrak F_\epsilon(\alpha_\epsilon)}_{C_{\epsilon, -1}^{k,\gamma}}\lesssim
  \begin{cases*}
  \delta^{\mu}& if $\alpha_\epsilon=\alpha^1_\epsilon$ and $0<q\leq\frac {1-\nu} {\mu-\nu}$\\
	\delta^{\mu}+\epsilon^{1-\nu_0}\delta^{\nu_0}+\epsilon^{2(1-\nu)} & if $\alpha_\epsilon=\alpha^2_\epsilon$ and $\frac {\nu-\nu_0} {\nu+\mu-1+s_0-\nu_0}\leq q<1$,
\end{cases*}
  \end{equation*}
  where $\nu_0=\max\{\lambda_0,2\nu-1\}$ and $\delta=\epsilon^q$.
  \end{prop}
\begin{proof}Let $\phi_e$ be the standard $G_2$-structure on $\R^7$ and $C$ be a cone. Since $\abs{\phi-\Upsilon_*\phi_e}=O(r)$ on $B(0,R)$ therefore over $P_{\epsilon,C}\cap \Upsilon(B(0,R))$ we have
$$\abs{\fF_\epsilon(\alpha_\epsilon)}	\lesssim\abs{\fF_\epsilon(\alpha_\epsilon)-\fF_\epsilon(\alpha_\epsilon,\Upsilon_*\phi_e)}\lesssim \abs{\fF_\epsilon(\alpha_\epsilon,\Upsilon_*\phi_e)}+r.$$
\textbf{Case 1}: Suppose $\alpha_\epsilon=\alpha^1_\epsilon$. Then over $P_{\epsilon,C}\setminus \Upsilon(B(0,2\delta))$, $\alpha_\epsilon=\alpha$ and hence $\fF_\epsilon(\alpha_\epsilon) =0$.
 Over $P_{\epsilon,C}\cap\Upsilon(\overline{B(0,\delta)})$ we have $\alpha_\epsilon=\Upsilon_*({s_\epsilon}_* \beta)$ and therefore $\fF_\epsilon(\alpha_\epsilon,\Upsilon_*\phi_e)=0$. Over $P_{\epsilon,C}\cap\Upsilon(A(\delta,2\delta))$ we have
 $$\abs{\fF_\epsilon(\alpha_\epsilon,\Upsilon_*\phi_e)}\lesssim \abs{\nabla \alpha_\epsilon}\lesssim r^{\mu-1}+\epsilon^{1-\nu}r^{\nu-1}.$$
Here the last inequality uses the assumption that $0<q\leq\frac {1-\nu} {\mu-\nu}$ where $\delta=(\epsilon R_\infty)^q$. 
 From these estimates we conclude that $\norm{\mathfrak F_\epsilon(\alpha_\epsilon)}_{C_{\epsilon, -1}^{0}}\lesssim \delta^{\mu}$. A similar computation with H\"older seminorm and higher derivatives will prove that $\norm{\mathfrak F_\epsilon(\alpha_\epsilon)}_{C_{\epsilon, -1}^{k,\gamma}}\lesssim \delta^{\mu}$.
 
 \textbf{Case 2}: Suppose $\alpha_\epsilon=\alpha^2_\epsilon$. 
 Over $P_{\epsilon,C}\setminus \Upsilon(B(0,2\delta))$ we have 
 $$\abs{\fF_\epsilon(\alpha_\epsilon)}	\lesssim \Big(r^{-1}\bigl\lvert\sum_{i=1}^l\epsilon ^{1-\lambda_i} \alpha_i\bigr\rvert+\bigl\lvert\sum_{i=1}^l\epsilon ^{1-\lambda_i} \nabla \alpha_i\bigr\rvert\Big)^2\lesssim \sum_{i=1}^l(\epsilon^{-1}r)^{2(\lambda_i-1)} \lesssim (\epsilon^{-1}r)^{2(\nu-1)}.$$
 Over $P_{\epsilon,C}\cap \Upsilon(A(\delta, 2\delta))$ we have $\fF_\epsilon(\alpha_\epsilon,\Upsilon_*\phi_e)=\bD_C(\alpha_\epsilon)+Q_C(\alpha_\epsilon)$.
The term $\abs{\bD_C(\alpha_\epsilon)}$ can be estimated as follows. 
  \begin{align*}
	\abs{\bD_C(\alpha_\epsilon)}	
	&\lesssim \abs{\nabla \rho_\delta}\bigl\lvert{s_\epsilon}_*\beta-\sum_{i=1}^l\epsilon ^{1-\lambda_i} \alpha_i-\alpha\bigr\rvert+ \abs{\bD_C({s_\epsilon}_*\beta)}+ \bigl\lvert\bD_C(\sum_{i=1}^l\epsilon ^{1-\lambda_i} \alpha_i+\alpha)\bigr\rvert\\
	& \lesssim (\epsilon^{-1}r)^{\nu_0-1}+ r^{\mu-1}+(\epsilon^{-1}r)^{2(\nu-1)}+r^{\mu-1}\sum_{i=1}^l\epsilon ^{1-\lambda_i} r^{\lambda_i+s_0-1}+ r^{2(\mu-1)}.
	\end{align*}
 If $q\geq\frac {\nu-\nu_0} {\nu+\mu-1+s_0-\nu_0}$ then we have
$$\epsilon ^{1-\lambda_i} r^{\lambda_i+\mu+s_0-2}\lesssim \epsilon ^{1-\nu} r^{\nu+\mu+s_0-2}\lesssim (\epsilon^{-1}r)^{\nu_0-1}.$$
 The term $\abs{Q_C(\alpha_\epsilon)}$ can be estimated as follows.
 $$\abs{Q_C(\alpha_\epsilon)}	\lesssim (r^{-1}\abs{ \alpha_\epsilon}+\abs{ \nabla \alpha_\epsilon})^2 \lesssim (\epsilon^{-1}r)^{2(\nu-1)}+ r^{2(\mu-1)}.$$
 Hence $$\norm{\mathfrak F_\epsilon(\alpha_\epsilon)}_{C_{\epsilon, -1}^{0}}\lesssim \delta^{\mu}+\epsilon^{1-\nu_0}\delta^{\nu_0}+\epsilon^{2(1-\nu)}.$$ A similar computation with H\"older seminorm and higher derivatives will also prove that $\norm{\mathfrak F_\epsilon(\alpha_\epsilon)}_{C_{\epsilon, -1}^{k,\gamma}}\lesssim \delta^{\mu}+\epsilon^{1-\nu_0}\delta^{\nu_0}+\epsilon^{2(1-\nu)}$.
\end{proof}

\begin{prop}[\textbf{Quadratic estimate}]\label{prop quadratic estimate with epsilon} For any sufficiently small $\epsilon>0$ there exists $0<T_\epsilon<T$ such that for all $u, v\in C^{\infty}(V_{P_{\epsilon,C}})$, $\eta \in C^{\infty}(NP_{\epsilon,C})$ and $t_1,t_2\in (t_0-T_\epsilon,t_0+T_\epsilon)$ the following estimates hold. \begin{enumerate}[(i)]
\item $\abs{d\mathfrak F_{\epsilon_{|u}}(\eta)-d\mathfrak F_{\epsilon_{|v}}(\eta)}\lesssim  (w_{\epsilon,1}\abs{u-v}+w_{\epsilon,0}\abs{\nabla^\perp (u-v)})(w_{\epsilon,1}\abs{\eta}+\abs{w_{\epsilon,0}\nabla^\perp \eta})$, 
\item $\norm{Q_\epsilon(u)-Q_\epsilon(v)}_{C_{\epsilon, -1}^{k,\gamma}}\lesssim \epsilon^{-1}\norm{u-v}_{C_{\epsilon, 0}^{k+1,\gamma}}\Big(\norm{u-\alpha_\epsilon}_{C_{\epsilon, 0}^{k+1,\gamma}}+\norm{v-\alpha_\epsilon}_{C_{\epsilon, 0}^{k+1,\gamma}}\Big).$
 \item $\norm{\tilde Q_\epsilon(u,t_1)-\tilde Q_\epsilon(v,t_2)}_{C_{\epsilon, -1}^{k,\gamma}}$
 
$\lesssim \epsilon^{-1} \Big(\norm{u-v}_{C_{\epsilon, 0}^{k+1,\gamma}}+\abs{t_1-t_2}\Big)\Big(\norm{u-\alpha_\epsilon}_{C_{\epsilon, 0}^{k+1,\gamma}}+\norm{v-\alpha_\epsilon}_{C_{\epsilon, 0}^{k+1,\gamma}} + \abs{t_1-t_0}+\abs{t_2-t_0}\Big).$ 
\end{enumerate}
 
\end{prop}
\begin{proof} The proofs are exactly like \autoref{prop cs quadratic estimate}, \autoref{prop Holder cs quadratic estimate}. An important observation here is that one can express $\mathcal L_u\mathcal L_v\psi$ over $V_{P_{\epsilon,C}}$ formally as a quadratic polynomial 
$$\mathcal L_u\mathcal L_v\psi=O(f_{1,\epsilon})\cdot u\cdot v+O(f_{2,\epsilon})\cdot (u\cdot \nabla^\perp v+v\cdot \nabla^\perp u)+\psi \cdot \nabla^\perp u\cdot \nabla^\perp v,$$
where $O(f_{1,\epsilon})=O(w_{\epsilon,1})$ and $O(f_{2,\epsilon})=O(w_{\epsilon,0})$. 
Finally to see (ii) we write $Q_\epsilon(u)-Q_\epsilon(v)$ as
$$\int_0^1dQ_{\epsilon_{|tu+(1-t)v}}(u-v)dt=\int_0^1\big(d\mathfrak F_{\epsilon_{|tu+(1-t)v}}(u-v)-d\mathfrak F_{\epsilon_{|\alpha_\epsilon}}(u-v)\big)dt,$$
and (ii) follows from (i).
To prove (iii) we only observe that
$\tilde Q_\epsilon(u,t_1)-\tilde Q_\epsilon(v,t_2)$ is equal to
\begin{align*}
	&\int_0^1\big(d\tilde \fF_{\epsilon_{|(tu+(1-t)v, tt_1+(1-t)t_2)}}(u-v,0)-d\tilde \fF_{\epsilon_{|(\alpha_\epsilon,tt_1+(1-t)t_2)}}(u-v,0)\big)dt\\
	&+\int_0^1\big(d\tilde \fF_{\epsilon_{|(\alpha_\epsilon,tt_1+(1-t)t_2)}}(u-v,0)-d\tilde \fF_{\epsilon_{|(\alpha_\epsilon,t_0)}}(u-v,0)\big)dt\\
	&+\int_0^1\big(d\tilde \fF_{\epsilon_{|(tu+(1-t)v, tt_1+(1-t)t_2)}}(0,t_1-t_2)-d\tilde \fF_{\epsilon_{|(tu+(1-t)v,t_0)}}(0,t_1-t_2)\big)dt\\
	&+\int_0^1\big(d\tilde \fF_{\epsilon_{|(tu+(1-t)v,t_0)}}(0,t_1-t_2)-d\tilde \fF_{\epsilon_{|(\alpha_\epsilon,t_0)}}(0,t_1-t_2)\big)dt. \qedhere
\end{align*}
\end{proof}


The remaining estimates are the linear estimates, namely, establishing a uniform lower bound for the linearization operator on suitable weighted function spaces that depend on the scaling parameter $\epsilon$. In general, this is hopeless unless we restrict to the complement of the following approximate kernel.

 \begin{definition}\label{def matching kernel} Denote the asymptotic limit maps (see \autoref{lem main Fredholm}) $i_{P_C,0}$ and $i_{L_C,0}$ by $i_{P}$ and $i_{L}$ respectively. We define 
 	\begin{enumerate}[(i)]
 	\item the \textbf{matching kernel} $\mathcal K^{\mathfrak m}$ by
$$\mathcal K^{\mathfrak m}:=\{(u_L,u_{P})\in \ker \fL_{L,0}\times \ker \fL_{P,0}:i_{L}(u_{L})=i_{P}(u_{P})\},$$
\item the \textbf{approximate kernel} of $\fL_{P_\epsilon}$ by
$$\mathcal K^{\mathfrak m}_{\epsilon}:=\{ \rho_{\frac \delta 2} \Upsilon_*({s_\epsilon}_* u_L)+(1-\rho_{2\delta})u_P: (u_L,u_{P})\in \mathcal K^{\mathfrak m}\},$$ 
\item $\mathcal X^{k+1,\gamma}_{\epsilon}:=\{u\in C^{k+1,\gamma}_{\epsilon,0}:\inp{(\Upsilon_*{s_\epsilon}_*)^{-1}u}{u_L}_{L^2_{K_L}}=\inp{u}{u_P}_{L^2_{K_P}}=0 \ \ \forall \  (u_L,u_{P})\in \mathcal K^{\mathfrak m}\}.$\qedhere
\end{enumerate}
\end{definition}

\begin{prop}[{\textbf{Linear estimate I}}]\label{prop main linear estimate} Fix a real number $\omega>0$. For any sufficiently small $\epsilon>0$, there exists a constant $C_\omega>0$ independent of $\epsilon$ such that for  all $u\in\mathcal X^{k+1,\gamma}_{\epsilon}$,
	we have
	$$\norm{u}_{C^{k+1,\gamma}_{\epsilon,0}}\leq C_\omega\epsilon^{-\omega}\norm{\fL _{P_\epsilon}u}_{C^{k,\gamma}_{\epsilon,-1}}.$$
\end{prop}
\begin{proof}By Schauder estimate in \autoref{prop Schauder estimate} it is enough to prove that
 $$\norm{u}_{L^{\infty}_{\epsilon,0}}\leq C_\omega\epsilon^{-\omega}\norm{\fL _{P_\epsilon}u}_{C^{k,\gamma}_{\epsilon,-1}}.$$
We prove this by contradiction. If this is not true then there exist a decreasing sequence $\epsilon_n \to 0$ as $n\to\infty$ and $u_n$ in $\mathcal X^{k+1,\gamma}_{\epsilon_n}$  such that 
	$$\norm{u_n}_{L^{\infty}_{\epsilon_n,0}}=1,\ 
	 \epsilon_n^{-\omega}\norm{\fL_{P_{\epsilon_n}}u_n}_{C^{k,\gamma}_{\epsilon_n,-1}}\to 0 \ \text{as}\ n\to\infty.$$
	Denote the restrictions of $u_n$ over $P^\pm_{\epsilon_n,C}$ by $u^\pm_{n}$. We define $u_{n,P}:=u^-_{n}$ and $u_{n,L}:=\epsilon_n(\Upsilon_*{s_{\epsilon_n}}_*)^{-1}u_n^{+}$. By Schauder estimate in \autoref{prop Schauder estimate} we have $\norm{u_n}_{C^{k+1,\gamma}_{\epsilon_n,0}}$ is bounded and hence $\norm{u_{n,P}}_{C^{k+1,\gamma}_{P,0}}$, $\norm{u_{n,{L}}}_{C^{k+1,\gamma}_{L,0}}$ are also bounded.  The Arzelà-Ascoli theorem implies that there exist subsequences which we call again $u_{n,P}$ and $u_{n,L}$, and there exist $u_P$ in $C^{k+1,\frac\gamma2}_{P,0}$, $u_{L}$ in $C^{k+1,\frac\gamma2}_{L,0}$ such that $$\fL_P u_P=0, \quad \fL_{L}u_{L}=0,$$ and
	$$\norm{ u_{n,P}-u_P}_{C^{k+1,\frac\gamma2}_{P,\operatorname{loc}}}\to 0 \ \ \ \text{and} \ \ \  \norm{u_{n,{L}}-u_{L}}_{C^{k+1,\frac\gamma2}_{{L},\operatorname{loc}}}\to 0 \ \text{as}\ n\to\infty.$$
Moreover, by elliptic regularity in \autoref{prop elliptic regularity} we get $u_P\in C^{k+1,\gamma}_{P,0}$ and $u_{L}\in C^{k+1,\gamma}_{L,0}$. 
 By taking further subsequences if necessary we prove the following which leads to the contradiction:
\begin{enumerate}[(i)]
	\item $\norm{ u_{n,P}-u_{P}}_{C^{k+1,\gamma}_{P^-_{\epsilon_n,C},0}}\to 0$ as $n\to\infty$, 
 \item $\norm{ u_{n,L}-u_{L}}_{C^{k+1,\gamma}_{L_{\epsilon_n^{-1}\delta_n,C},0}}\to 0$ as  $n\to\infty$ where $L_{\epsilon_n^{-1}\delta_n,C}:=s^{-1}_{\epsilon_n}\Upsilon^{-1}P^+_{\epsilon_n,C}\subset L_C$,
 \item $u_P=0$ and $u_L=0$. Hence, $\norm{u_n}_{L^{\infty}_{\epsilon_n,0}}\leq \norm{ u^+_{n}}_{L^{\infty}_{P^+_{\epsilon_n,C}}}+\norm{ u^-_{n}}_{L^{\infty}_{P^-_{\epsilon_n,C}}}\to 0 $ as  $n\to\infty$.
\end{enumerate}

 To prove (i), we fix a real number $p$ with $1<p<\frac 1q$. For sufficiently small $s>0$ we have 
	\begin{align*}
	\norm{\fL_P((1-\rho_{\delta_n^{p}})u_n)}_{C^{k,\gamma}_{P,-s-1}}&\lesssim (\epsilon_n^{\fc_{pq}}+\epsilon_n^{\fc_{q}})\norm{u_n}_{C^{k+1,\gamma}_{0,\epsilon_n}}+\norm{\fL_{P_{\epsilon_n}}(1-\rho_{\delta_n^{p}})u_n}_{C^{k,\gamma}_{\epsilon_n,-s-1}}\\
	&\lesssim (\epsilon_n^{\fc_{pq}}+\epsilon_n^{\fc_{q}}+\delta_n^{ps})\norm{u_n}_{C^{k+1,\gamma}_{0,\epsilon_n}}+\norm{\fL_{P_{\epsilon_n}}u_{n}}_{C^{k,\gamma}_{\epsilon_n,-s-1}}\\
	&\lesssim \epsilon_n^{\fc_{pq}}+\epsilon_n^{\fc_{q}}+\delta_n^{ps}+\norm{\fL_{P_{\epsilon_n}}u_{n}}_{C^{k,\gamma}_{\epsilon_n,-1}}\lesssim \epsilon_n^{\fc_{pq}}+\epsilon_n^{\fc_{q}}+\delta_n^{ps}+\epsilon_n^{\omega}.
	\end{align*}	 
Therefore there exists $v_{n,P}\in \ker \fL_{P,-s}=\ker \fL_{P,0}$ satisfying $$\norm{(1-\rho_{\delta_n^{p}})u_n-v_{n,P}}_{C^{k+1,\gamma}_{P,-s}}\lesssim  \epsilon_n^{\fc_{pq}}+\epsilon_n^{\fc_{q}}+\delta_n^{ps}+\epsilon_n^{\omega}$$
and hence
$\norm{u_{n,P}-v_{n,P}}_{C^{k+1,\gamma}_{P^-_{\epsilon_n,C},0}}\lesssim (\epsilon_n^{\fc_{pq}}+\epsilon_n^{\fc_{q}}++\delta_n^{ps}+\epsilon_n^{\omega})\delta_n^{-s }
\to 0$  as $n\to\infty$. As, $\ker \fL_{P,0}$ is finite dimensional, the norms $\norm{\cdot}_{C^{k+1,\frac\gamma 2}_{K_P}}$ and $\norm{\cdot}_{C^{k+1,\gamma}_{P,0}}$ are equivalent. Taking further subsequence yields $\norm{ v_{n,P}-u_P}_{C^{k+1,\frac\gamma2}_{K_P}}\to 0$ and hence  $\norm{ v_{n,P}-u_P}_{C^{k+1,\gamma}_{P,0}}\to 0 \  \ \text{as}\ n\to\infty.$

To prove (ii), we define  $\tilde u_{n,L}:=\epsilon_n(\Upsilon_*{s_{\epsilon_n}}_*)^{-1}(\rho_{\delta_n^{\frac 12}}u_n)$ . For sufficiently small $s>0$ we have 
	\begin{align*}
	\norm{\fL_{L}\tilde u_{n,L}}_{C^{k,\gamma}_{L,s-1}}&\lesssim \epsilon_n^{s-1}\norm{\Upsilon_*{s_{\epsilon_n}}_*\mathfrak L_{L}{s_{\epsilon_n}}^{-1}_*\Upsilon_*^{-1}(\rho_{\delta_n^{\frac 12}}u_n)}_{C_{\epsilon, s-1}^{k,\gamma}}\\
	&\lesssim \epsilon_n^{s}(\epsilon_n^{\fc_{q}}+\delta_n^{\frac {\mu-1} 2})\norm{\rho_{\delta_n^{\frac 12}}u_n}_{C^{k+1,\gamma}_{\epsilon_n,s}}+\epsilon_n^{s}\norm{\fL_{P_{\epsilon_n}}(\rho_{\delta_n^{\frac 12}}u_n)}_{C^{k,\gamma}_{\epsilon_n,s-1}}\\
	&\lesssim \epsilon_n^{s}\delta_n^{-\frac{s}{2}}\norm{u_n}_{C^{k+1,\gamma}_{\epsilon_n,0}}+\epsilon_n^{s}\delta_n^{-\frac{s}{2}}\norm{\fL_{P_{\epsilon_n}}u_{n}}_{C^{k,\gamma}_{\epsilon_n,-1}}\lesssim \epsilon_n^{s}\delta_n^{-\frac{s}{2}}+\epsilon_n^{s}\delta_n^{-\frac{s}{2}}\epsilon_n^\omega.
	\end{align*}
	Then as in (i) there exists $v_{n,L}\in \ker \fL_{L,s}=\ker \fL_{L,0}$ satisfying
\begin{align*}
\norm{u_{n,L}-v_{n,L}}_{C^{k+1,\gamma}_{L,0}}&\lesssim (\epsilon_n^{-1}\delta_n)^s(\epsilon_n^{s}\delta_n^{-\frac{s}{2}}+\epsilon_n^{s}\delta_n^{-\frac{s}{2}}\epsilon_n^\omega)\\
&\lesssim  \delta_n^{\frac{s}2 }+\delta_n^{\frac{s}2 }\epsilon_n^\omega \to 0  \ \text{{as}}\ n\to\infty.
\end{align*}
Again as in (i) by taking further subsequence if necessary we get, $\norm{ v_{n,L}-u_L}_{C^{k+1,\frac\gamma2}_{K_L}}\to 0 $ and $\norm{ v_{n,L}-u_L}_{C^{k+1,\gamma}_{L,0}}\to 0$ as $n\to\infty$.

It remains to prove (iii). For each $\sigma\in \Sigma$ we have,
 $$i_P u_P(\sigma)=\Upsilon_*^{-1}\lim_{n\to \infty}u_{n,P}\big(\Upsilon\Big({\frac {3\delta_n}2}, \sigma\Big)\Big)= \lim_{n\to \infty}u_{n,L}\Big(\frac {3\epsilon_n^{-1}\delta_n}2, \sigma\Big)=i_Lu_L(\sigma). $$
Since $u_n\in \mathcal X^{k+1,\gamma}_{\epsilon_n}$ therefore as $n\to \infty$ we get 
$$\norm{u_P}_{L^2_{K_P}}=\inp{u_P-u_{n,P}}{u_P}_{L^2_{K_P}}\lesssim \norm{u_{n,P}-u_P}_{L^\infty_{P,0}}\norm{u_P}_{L^\infty_{P,0}}\to 0 $$ 
and hence $u_P=0$. Similarly we get $u_L=0$. 
\end{proof}

From the above linear estimate, we observe that $\mathfrak{L}_{P_\epsilon}$ may not be surjective, as it is self-adjoint, and the matching kernel is never trivial due to the scaling freedom on $L$. This is why we did not choose to perform the desingularization in a fixed $G_2$-structure, but rather along a path of $G_2$-structures. To proceed, we require a corresponding estimate for the linearization operator along this path, and we aim to obtain a uniform lower bound on its restriction to the complement of the corresponding extended matching kernel, which we define next.

\begin{definition}\label{def extended matching kernel} Given $P$ a CS associative submanifold with respect to a co-closed $G_2$-structure $\phi$ and $L$ an AC associative submanifold in $\R^7$ with the same asymptotic cone, and given $\bphi: [0,1]\to \sP$ a path of co-closed $G_2$-structures with $\bphi(t_0)=\phi$ for some $t_0\in (0,1)$, we define the \textbf{extended matching kernel}:
\begin{equation*}
\widetilde{\cK}^{\mathfrak m}:=\{(u_L,u_P,t)\in \ker \fL_{L,0}\oplus C^{k+1,\gamma}_{P,0}\oplus \R:\fL_Pu_P+tf_P=0, i_{L}u_L=i_{P}u_P\},
\end{equation*}
where $f_P$ is defined in \autoref{def 1 para moduli CS}.  Note that, the matching kernel ${\cK}^{\mathfrak m}\subset \widetilde{\cK}^{\mathfrak m}$.
\end{definition}

\begin{prop}[{\textbf{Linear estimate II}}] \label{prop main linear estimate 1 parameter}
If the extended matching kernel $\widetilde{\cK}^{\mathfrak m}$ is equal to the matching kernel $\cK^{\mathfrak m}$ and $\omega>0$ is any real number,
 then for any sufficiently small $\epsilon>0$ there exists a constant $\tilde C_\omega>0$ independent of $\epsilon$ such that for all $u\in \mathcal X^{k+1,\gamma}_\epsilon$ and $t\in \R$,we have
$$\norm{u}_{C^{k+1,\gamma}_{\epsilon,0}}+\abs{t}\leq \tilde C_\omega\epsilon^{-\omega}\norm{\fL_{P_\epsilon}u+tf_\epsilon}_{C^{k,\gamma}_{\epsilon,-1}}.$$
\end{prop}

\begin{proof}Let $f^\pm_{\epsilon}$ be the restrictions of $f_\epsilon$ over $P^\pm_{\epsilon,C}$. Then similar to the proof of \autoref{prop error estimate with epsilon} we obtain that $$\norm{ f_\epsilon^{-}-f_{P}}_{C^{k,\gamma}_{P^-_{\epsilon,C},-1}}\lesssim \epsilon^{\fc_{q}}\ \ \  \text{and} \ \ \  \norm{ f_\epsilon^{+}}_{C^{k,\gamma}_{P^+_{\epsilon,C},-1}}\lesssim \epsilon^{\fc_{q}}.$$
 
\textbf{Claim 1:} For all $u\in \mathcal X^{k+1,\gamma}_\epsilon$ with $\norm{u}_{C^{k+1,\gamma}_{\epsilon,0}}\leq 1$ we have
$$1\leq \tilde C_\omega\epsilon^{-\omega}\norm{\fL_{P_\epsilon}u-f_\epsilon}_{C^{k,\gamma}_{\epsilon,-1}}.$$
We will prove this by contradiction. Suppose it is not true. Then there exist sequences $\epsilon_n\to0$ and $u_n\in \mathcal X^{k+1,\gamma}_{\epsilon_n}$ such that 
 $\norm{u_n}_{C^{k+1,\gamma}_{\epsilon_n,0}}\leq 1$  and $\epsilon_n^{-\omega}\norm{\fL_{P_{\epsilon_n}}u_n-f_{\epsilon_n}}_{C^{k,\gamma}_{\epsilon_n,-1}}$ converges to $0$ as $n\to\infty$.
After defining $u_{n,P}$ and  $u_{n,L}$ as in \autoref{prop main linear estimate} a similar argument as in there yields a smooth section $u_P\in C^{k+1,\gamma}_{P,0}$ and $u_L\in C^{k+1,\gamma}_{L,0}$ such that $$\fL_Pu_P=f_P,\ \
 \norm{ u_{n,P}-u_P}_{C^{k+1,\frac\gamma2}_{P,\operatorname{loc}}}\to 0 \ \text{as}\  n\to\infty,$$
 and $\fL_Lu_L=0,\ \
 \norm{ u_{n,L}-u_L}_{C^{k+1,\frac\gamma2}_{L,\operatorname{loc}}}\to 0 \ \text{as}\  n\to\infty$.
Moreover, we can prove as in \autoref{prop main linear estimate} the following. 
  \begin{enumerate}[(i)]
	\item $\norm{ u_{n,P}-u_{P}}_{C^{k+1,\gamma}_{P^-_{\epsilon_n,C},0}}\to 0$ as $n\to\infty$, 
 \item $\norm{ u_{n,L}-u_{L}}_{C^{k+1,\gamma}_{L_{\epsilon_n^{-1}\delta_n,C},0}}\to 0$ as  $n\to\infty$ where $L_{\epsilon_n^{-1}\delta_n,C}:=s^{-1}_{\epsilon_n}\Upsilon^{-1}P^+_{\epsilon_n,C}\subset L_C$,
 \item $i_Pu_P=i_Lu_L$. This contradicts to the assumption.
\end{enumerate}
To prove (i), we fix a real number $p$ with $1<p<\frac 1q$. For sufficiently small $s>0$ we have 
	\begin{align*}
	&\norm{\fL_P(1-\rho_{\delta_n^{p}})(u_n-u_P)}_{C^{k,\gamma}_{P,-s-1}}\\
	&\lesssim \epsilon_n^{\fc_{pq}}+\epsilon_n^{\fc_{q}}+\norm{\fL_{P_{\epsilon_n}}(1-\rho_{\delta_n^{p}})u_n-\fL_{P}(1-\rho_{\delta_n^{p}})u_P)}_{C^{k,\gamma}_{\epsilon_n,-s-1}}\\
	&\lesssim \epsilon_n^{\fc_{pq}}+\epsilon_n^{\fc_{q}}+\delta_n^{ps}+\norm{(1-\rho_{\delta_n^{p}})(\fL_{P_{\epsilon_n}}u_{n}-f_P)}_{C^{k,\gamma}_{\epsilon_n,-s-1}}\\
	&\lesssim \epsilon_n^{\fc_{pq}}+\epsilon_n^{\fc_{q}}+\delta_n^{ps}+\norm{\fL_{P_{\epsilon_n}}u_{n}-f_{\epsilon_n}}_{C^{k,\gamma}_{\epsilon_n,-1}}+\norm{(1-\rho_{\delta_n^{p}})(f_{\epsilon_n}-f_P)}_{C^{k,\gamma}_{\epsilon_n,-1}}\\
	&\lesssim  \epsilon_n^{\fc_{pq}}+\epsilon_n^{\fc_{q}}+\delta_n^{ps}+\epsilon_n^{\omega}.
	\end{align*}
Now (i) follows as in (i) of \autoref{prop main linear estimate}. 
Since $\delta_n^{\frac s 2}\norm{\rho_{\delta_n^{\frac 12}}f_{\epsilon_n}}_{C^{k,\gamma}_{\epsilon_n,-1}}\to 0$ as $n\to \infty$, (ii) follows from (ii) of \autoref{prop main linear estimate}. Again the proof of (iii) is also similar to (iii) of \autoref{prop main linear estimate}.

\textbf{Claim 2:} For all $u\in \mathcal X^{k+1,\gamma}_\epsilon$ and $t\in \R$ with $\norm{u}_{C^{k+1,\gamma}_{\epsilon,0}}= 1$ and $\abs{t}\leq 1$ we have
$$1\leq \tilde C_\omega\epsilon^{-\omega}\norm{\fL_{P_\epsilon}u-tf_\epsilon}_{C^{k,\gamma}_{\epsilon,-1}}.$$
Suppose it is not true. Then there exist sequences $\epsilon_n\to0$ and $u_n\in \mathcal X^{k+1,\gamma}_{\epsilon_n}$, $t_n\in \R$  such that 
 $\norm{u_n}_{C^{k+1,\gamma}_{\epsilon_n,0}}=1$, $t_n\to t_\infty$  and $\epsilon_n^{-\omega}\norm{\fL_{P_{\epsilon_n}}u_n-f_{\epsilon_n}}_{C^{k,\gamma}_{\epsilon_n,-1}}$ converges to $0$ as $n\to\infty$. If $t_\infty\neq 0$ then similar arguments to Claim $1$ replacing $f_{\epsilon_n}$ by $t_nf_{\epsilon_n}$ lead to a contradiction. If $t_\infty=0$ then similar arguments in \autoref{prop main linear estimate} leads to again a contradiction.
 
Evidently, Claims $1$ and $2$ are enough to prove the proposition.
\end{proof}

\subsection{Proof of the desingularization theorem}\label{subsection main desing}

To prove the desingularization \autoref{thm main desing} we solve the nonlinear PDE, $\tilde \fF_\epsilon (u,t)=0$ with $(u,t)$ is {very close} to $(\alpha_\epsilon,0)$. Indeed, we use the following \autoref{lem Banach contraction}, an application of the Banach contraction principle \cite[Lemma 7.2.23]{Donaldson1990}.   
 \begin{lemma}\label{lem Banach contraction}
	Let $\mathcal X,\mathcal Y$ be two Banach spaces and $\cX_1\subset \cX$ be a Banach subspace. Let $V\subset \mathcal X$ be a neighbourhood of $0\in \mathcal X$. Let $x_0\in V$. Let $F:V\to \mathcal Y$ be a smooth map of the form\footnote{Note that $Q(x_0)=-L(x_0)$. }
	$$F(x)=L(x)+Q(x)+F(x_0)$$
	such that:
	\begin{itemize}
		\item $L:\mathcal X_1 \to \mathcal Y$ is an invertible operator and there exists a constant $c_L>0$ such that for all $x\in \mathcal X_1$, $\norm{x}_{\mathcal X}\leq c_L\norm{Lx}_{\mathcal Y},$
		\item $Q:V\to \mathcal Y$ is a smooth map and there exists a constant $c_Q>0$ such that for all $x_1,x_2\in V$, 
		$$\norm{Q(x_1)-Q(x_2)}_\mathcal Y\leq c_Q\norm{x_1-x_2}_\mathcal X(\norm{x_1-x_0}_\mathcal X+\norm{x_2-x_0}_\mathcal X).$$
			\end{itemize}
			
	If  $\norm{F(x_0)}_{\mathcal Y}\leq\frac 1{10 c_L^2c_Q}$ and $B(x_0,\frac 1{5 c_Lc_Q})\subset V$, then there exists a unique $x\in x_0+\mathcal X_1$ with \[\norm{x-x_0}_{\mathcal X}\leq c_L\norm{F(x_0)}_{\mathcal Y},\quad F(x)=0.\]
\end{lemma}


\begin{proof}[{\normalfont{\textbf{Proof of \autoref{thm main desing}}}}]Let $\alpha_\epsilon$ be $\alpha^1_\epsilon$ as in \autoref{def approx asso cs}. Let $T_\epsilon $ be as in \autoref{prop quadratic estimate with epsilon}.
	The map $$\tilde \fF_{\epsilon}: C^{k+1,\gamma}_{\epsilon,0} (V_{P_{\epsilon,C}})\times (t_0-T_\epsilon,t_0+T_\epsilon)\to C^{k,\gamma}_{\epsilon,-1}$$ can be written as 
$\tilde \fF_{\epsilon}(u,t_0+t)=\mathfrak L_{P_\epsilon}u+tf_\epsilon+ \tilde Q_\epsilon(u,t_0+t)+\mathfrak F_\epsilon(\alpha_\epsilon).$ 
Since the matching kernel $\mathcal K^{\mathfrak m}$ is one dimensional and the index of $\fL_{P_\epsilon}$ is zero, by \autoref{prop main linear estimate} and \autoref{prop main linear estimate 1 parameter} we obtain that 
$$\tilde L_{\epsilon}:\cX_\epsilon^{k+1,\gamma}\oplus \R\to C^{k,\gamma}_{\epsilon,-1}$$
defined by $\tilde L_\epsilon(u,t):=\fL_{{P_\epsilon}}u+tf_\epsilon$, is invertible and
$$\norm{u}_{C^{k+1,\gamma}_{\epsilon,0}}+\abs{t}\leq  C_{\tilde L_\epsilon}\norm{\tilde L_\epsilon(u,t)}_{C^{k,\gamma}_{\epsilon,-1}}\ \  \text{where}\ \   C_{\tilde L_\epsilon}=O(\epsilon^{-\omega}).$$
Again by \autoref{prop quadratic estimate with epsilon} if $t_1,t_2\in(-T_\epsilon,T_\epsilon)$, then we have
\begin{align*}
	&\norm{\tilde Q_\epsilon(u,t_0+t_1)-\tilde Q_\epsilon(v,t_0+t_2)}_{C_{\epsilon, -1}^{k,\gamma}}\\
	& \leq C_{Q_\epsilon} \Big(\norm{u-v}_{C_{\epsilon, 0}^{k+1,\gamma}}+\abs{t_1-t_2}\Big)\Big(\norm{u-\alpha_\epsilon}_{C_{\epsilon, 0}^{k+1,\gamma}}+\norm{v-\alpha_\epsilon}_{C_{\epsilon, 0}^{k+1,\gamma}} + \abs{t_1}+\abs{t_2}\Big),
\end{align*}
where $C_{Q_\epsilon}:=O(\epsilon^{-1})$.
Since $\nu<0$, by \autoref{prop error estimate with epsilon} we have, 
	$$\norm{\mathfrak F_\epsilon(\alpha_\epsilon)}C^{k,\gamma}_{\epsilon,-1}\leq \frac 1{10 C^2_{\tilde L_\epsilon}C_{Q_\epsilon}}=O(\epsilon^{1+2\omega})$$
for sufficiently small $\omega$.
Hence by \autoref{lem Banach contraction}, we have for all sufficiently small $\epsilon>0$, there exist $t(\epsilon)\in(t_0-T,t_0+T)$ and smooth closed associative submanifold $\tilde P_{\epsilon, t(\epsilon)}:=\Upsilon_{P_{\epsilon,t(\epsilon),C}}(\tilde \alpha_\epsilon)$ in $(Y,\phi_{t(\epsilon)})$, that is $\tilde \fF_{\epsilon}(\tilde \alpha_\epsilon,t(\epsilon))=0$. Moreover, if $\frac 1\mu<q\leq\frac {1-\nu} {\mu-\nu}$ then 
$$\norm{\tilde \alpha_\epsilon-\alpha_\epsilon}_{C^{k+1,\gamma}_{\epsilon,0}}+\abs{t(\epsilon)-t_0}\lesssim \epsilon^{q\mu-\omega}.$$

 
 Finally, we see that $\tilde P_{\epsilon, t(\epsilon)}\to P$ in the sense of currents as $\epsilon\to 0$. Indeed, for any $3$-form $\xi\in \Omega^3(Y)$ we have
\[
\Bigl\lvert \int_{\tilde P_{\epsilon, t(\epsilon)}}\xi-\int_{P}\xi \Bigr\rvert 
\leq \Bigl\lvert \int_{\tilde P_{\epsilon, t(\epsilon)}}\xi-\int_{P_{\epsilon, t(\epsilon)}}\xi \Bigr\rvert+ \Bigl\lvert \int_{P_{\epsilon, t(\epsilon)}}\xi-\int_{P}\xi \Bigr\rvert \to 0\ \  \text{as}\ \epsilon\to 0.
 \qedhere \]
\end{proof}
\begin{remark}\label{rmk improved approx asso desing} Observe that if there exists $\lambda_0<\nu$ satisfying (i) and (ii) of \autoref{def improved approx asso cs} then we can consider $\alpha_\epsilon=\alpha^2_\epsilon$ in the proof of \autoref{thm main desing}. Note also that  in the proof of \autoref{thm main desing} $(\tilde \alpha_\epsilon,t(\epsilon))$ satisfies the following estimates:
\begin{align}\label{eq error main desing thm} 
& \norm{\tilde \alpha_\epsilon-\alpha_\epsilon}_{C^{k+1,\gamma}_{\epsilon,0}}+\abs{t(\epsilon)-t_0} \nonumber \\
&\lesssim 
\begin{cases*}
	\epsilon^{q\mu-\omega},\ \ \text{if}\ \alpha_\epsilon=\alpha^1_\epsilon\ \text{and}\ \frac 1\mu<q\leq\frac {1-\nu} {\mu-\nu}\\
	\epsilon^{q\mu-\omega}+\epsilon^{1+(q-1)\nu_0-\omega},\ \ \text{if}\  \alpha_\epsilon=\alpha^2_\epsilon\ \text{and}\ \frac {\nu-\nu_0} {\nu+\mu-1+s_0-\nu_0}\leq q<1.
  \end{cases*}	
  \end{align}
  Indeed, these estimates follows from \autoref{lem Banach contraction}, \autoref{prop main linear estimate 1 parameter} and \autoref{prop error estimate with epsilon}.
\end{remark}

\section{Desingularizations for special cases of singularities}
In this section, we apply \autoref{thm main desing} to desingularize a CS associative submanifold with a Harvey--Lawson $T^2$-cone singularity, as well as an associative submanifold with a transverse self-intersection. In other words, we prove the two transitions discussed at the beginning of this article, although the second one is established only partially.

To apply \autoref{thm main desing}, we must verify that the hypothesis \autoref{hyp desing CS} holds in these specific cases. It turns out that there are no contributions from the CS side to the extended matching kernel, and therefore
\begin{equation}\label{eq matching special Lag}
\widetilde{\mathcal{K}}^{\mathfrak{m}} = \mathcal{K}^{\mathfrak{m}} \cong \ker \bD_{L, -1},
\end{equation}
where $L$ is the AC associative submanifold used for the desingularization. In the cases considered here, $L$ is special Lagrangian, so $\ker \bD_{L, -1}$ is purely topological, as stated in the proposition.

	\begin{prop}[{\citet[section 5, Section 5.2.3, Table 5.1]{Marshall2002}}]\label{prop lagrangian AC kernel}
Let $L$ be a AC special Lagrangian submanifold in $\C^3$. Let $\Sigma$ be the link of the asymptotic cone. Then 
$$\dim \ker \bD_{L,-1}= b^1(L)+b^0(\Sigma)-1.$$
\end{prop}	
We will see that in both of the special cases considered above, $\ker \bD_{L,-1}$ is one-dimensional, and the hypothesis \autoref{hyp desing CS} holds true.
\subsection{Desingularizations for Harvey-Lawson $T^2$-cone singularity}\label{subsection desing HL}
Let $\bphi \in \bsP_{\operatorname{cs}}^{\operatorname{reg}}$ be a generic path of co-closed $G_2$-structures. Then, by \autoref{def regular G2 str}, any CS associative submanifold $P \in \bcM_{\operatorname{cs}}^\bphi$ is unobstructed along this path; that is, the operator $\widebar{\bD}_{P,\mu,\cZ}$ is surjective. 
Moreover, any such $P$, together with an AC associative submanifold sharing the same asymptotic cone whose homogeneous kernel at rate $-1$ is two-dimensional, must satisfy the matching condition in \autoref{eq matching special Lag}. It follows that the stability index of the cone is $1$, and $P$ is an isolated point in the moduli space $\bcM_{\operatorname{cs}}^\bphi$. This is the content of the following lemma.

\begin{lemma}\label{lem generic matching kernel d=2} Let $\bphi\in \bsP_{\operatorname{cs}}^{\operatorname{reg}}$ be a path of co-closed $G_2$-structures and $P\in \bcM_{\operatorname{cs}}^\bphi$ be a CS associative with singularity at only one point. Let $L\subset \R^7$ be an AC associative with the same asymptotic cone $C$ of $P$. If $d_{-1}=2$, then 
$$\tilde \cK^\fm=\cK^\fm\cong \ker \fL_{L,-1}.$$ Moreover,  $\operatorname{s-ind}(C)=1$, $\dim \ker \fL_{P,-1}=1$ and $\ker \fL_{P,-1}=\langle v_P\rangle_\R$ such that $\inp{v_P}{f_P}_{L^2}\neq 0$.
 	\end{lemma}

\begin{proof}Since $\bphi\in \bsP_{\operatorname{cs}}^{\operatorname{reg}}$ and  $P\in \bcM_{\operatorname{cs}}^\bphi$, \autoref{thm generic moduli cs asso} (ii) implies that the stability index of the asymptotic cone $C$ satisfies $\operatorname{s-ind}(C)\leq 1$. Since $d_{-1}=2$, we have $\operatorname{s-ind}(C)= 1$, $V_\lambda=0$ for all non-zero $\lambda\in (-1,1)$, $V_0\cong \R^7$ and $C$ is a rigid cone, that is $V_1\cong T_\Sigma(G_2\cdot \Sigma)$, where $\Sigma$ is the link of $C$. Then by \autoref{thm moduli cs asso} (ii), we obtain that the deformation operator
$\bar \fL^{k,\gamma}_{P,\mu,\cZ}$ defined in \autoref{def 1 para moduli CS} is an isomorphism, where $\cZ$ is the stratum in the decomposition \autoref{eq stratification} that contains $G_2\cdot \Sigma$ as an open subset. Here $\mu > 1$ is chosen sufficiently close to $1$, as in \autoref{def CS weighted topo}, and therefore there are no indicial roots of $C$ in the interval between $1$ and $\mu$. Therefore $$\bar \fL^{k,\gamma}_{P,-1+s}:\R\oplus C^{k+1,\gamma}_{P_C,-1+s} \to C^{k,\gamma}_{P_C,-2+s}$$ 
defined by $ \bar\fL^{k,\gamma}_{P,-1+s}(t, u):=\fL_Pu+tf_P$, is an isomorphism as well, for all sufficiently small $s>0$. Indeed, the indicial roots in between $-1$ and $\mu$ only are $0$ and $1$ with the corresponding $V_0$ and $V_1$ as above and therefore  by \autoref{lem main Fredholm} the above two operators have same kernel. Since $\ind \fL_{P,-1+s}=-1$ by \autoref{prop index} (ii), we have that $\ind \bar \fL_{P,-1+s}=0$ and hence it is an isomorphism.
This further implies that $\ker \fL_{P,0} \subset \ker \fL_{P,-1+s} = 0$, and $\operatorname{coker} \fL_{P,-1+s} \cong \ker \fL_{P,-1} = \langle v_P \rangle_{\R}$ is one-dimensional with $\langle v_P, f_P \rangle_{L^2} \neq 0$. Moreover, $f_P \notin \operatorname{im} \fL_{P,-1+s} \supset \operatorname{im} \fL_{P,0}$, and hence $\widetilde{\mathcal{K}}^\mathfrak{m} = \mathcal{K}^\mathfrak{m} \cong \ker \fL_{L,-1}$.
 \end{proof}

\begin{proof}[{\normalfont{\textbf{Proof of \autoref{thm desing HL sing}}}}]
 We have seen in \autoref{eg s-ind Harvey-Lawson cone} that the Harvey-Lawson cone has $d_{-1}=2$. Therefore by \autoref{lem generic matching kernel d=2} and \autoref{prop lagrangian AC kernel} we obtain that for each Harvey-Lawson AC special Lagrangian $L^i:=L^i_1$, $i=1,2,3$ (see \autoref{eg Harvey-Lawson special Lagrangians}) the extended matching kernel and the matching kernel are one dimensional. Indeed,
 $$\dim \tilde \cK^\fm=\dim \cK^\fm=\dim \ker \fL^i_{L,-1}= b^1(S^1\times \C)+b^0(T^2)-1=1.$$
By \autoref{thm main desing} we obtain that for sufficiently small $\epsilon>0$, there exists $t^i(\epsilon)$ and smooth closed associative $\tilde P_{\epsilon,t^i(\epsilon)}$ in $(Y,\phi_{t^i(\epsilon)})$ such that $\tilde P_{\epsilon, t^i(\epsilon)}\to P$ in the sense of currents as $\epsilon\to 0$.
$\tilde P_{\epsilon, t^i(\epsilon)}$ is diffeomorphic to $P_{\epsilon, t^i(\epsilon)}$. Since $L^i$ are obtained by Dehn filling of $C_{HL}^o:=C_{HL}\setminus B(0,1)$ along  simple closed curves $\mu_i$ as mentioned in the theorem, therefore it is diffeomorphic to the Dehn filling of $P^o$ as required in the theorem (this was observed in \cite[Remark 3.6]{Doan2017d}).

It remains to prove that the leading order term of $t^i(\epsilon)-t_0$ is of the required form if $\bphi\in \bsP^\bullet$. Let $v_P\in \ker \fL_{P,-1-s}$ be as in \autoref{lem generic matching kernel d=2}. Since $\bphi\in \bsP^\bullet$, by definition we have
 \begin{equation*}
	i_{P,-1}v_P=b_1\xi_1+b_2\xi_2,\ b_1\neq 0,b_2\neq 0, b_1\neq b_2 \ \text{where}\ \xi_1, \xi_2\ \text{are in \autoref{eq AC HL} of \autoref{eg Harvey-Lawson special Lagrangians}}.   
\end{equation*}
We will prove the result for $t(\epsilon) := t^1(\epsilon)$ only, as the others follow by similar arguments.
 We use the notation $\alpha_\epsilon=\alpha^1_\epsilon=\rho_{\delta} \Upsilon_*({s_\epsilon}_* \beta)+ (1-\rho_{\delta})\alpha$ and $(\tilde \alpha_\epsilon,t(\epsilon))$ from \autoref{thm main desing}, which satisfies $$\tilde \fF_{\epsilon}(\tilde \alpha_\epsilon,t(\epsilon))=\mathfrak L_{P_\epsilon}(\tilde \alpha_\epsilon-\alpha_\epsilon)+(t(\epsilon)-t_0)f_\epsilon+  Q_\epsilon(\tilde \alpha_\epsilon-\alpha_\epsilon,t(\epsilon)-t_0)+\mathfrak F_\epsilon(\alpha_\epsilon)=0.$$ 
Define $\mathring{P}_{\epsilon}:={P_C}\setminus \Upsilon(B(0,2 \epsilon R_\infty))$ and $\mathring{P}_{\delta}:={P_C}\setminus \Upsilon(B(0,\delta))$. 
Over $\mathring{P}_{\delta}$, $$\mathfrak F_\epsilon(\alpha_\epsilon)=\fL_P(\alpha_\epsilon-\alpha)+Q_P(\alpha_\epsilon-\alpha).$$
Therefore,
\begin{equation}\label{eq main HL leading order term}
	0=\inp{\tilde \fF_{\epsilon}(\tilde \alpha_\epsilon,t(\epsilon))}{v_P}_{L^2_{\mathring{P}_{\epsilon}}}=(t(\epsilon)-t_0)\inp{f_\epsilon}{v_P}_{L^2_{\mathring{P}_{\epsilon}}}+\inp{\fL_{P}\alpha_\epsilon}{v_P}_{L^2_{\mathring{P}_{\delta}}}+e^1_\epsilon+e^2_\epsilon+e^3_\epsilon+e^4_\epsilon+e^5_\epsilon
	\end{equation}
where $e^1_\epsilon:=\inp{ \mathfrak L_{P_\epsilon}(\tilde \alpha_\epsilon-\alpha_\epsilon)}{v_P}_{L^2_{\mathring{P}_{\epsilon}}}$, 
	$e^2_\epsilon:=\inp{ Q_\epsilon(\tilde \alpha_\epsilon-\alpha_\epsilon,t(\epsilon)-t_0)}{v_P}_{L^2_{\mathring{P}_{\epsilon}}}$,
	and$$e^3_\epsilon:=\inp{Q_P(\alpha_\epsilon-\alpha)}{v_P}_{L^2_{\mathring{P}_{\delta}}}, \ \  e^4_\epsilon=-\inp{\fL_{P}(\alpha)}{v_P}_{L^2_{\mathring{P}_{\epsilon}}}+\inp{\fL_{P}(\alpha)}{v_P}_{L^2_{\mathring{P}_{\epsilon}\setminus \mathring{P}_{\delta}}},\  e^5_\epsilon=\inp{\mathfrak F_\epsilon(\alpha_\epsilon)}{v_P}_{L^2_{\mathring{P}_{\epsilon}\setminus \mathring{P}_{\delta}}}.$$	
By \autoref{prop formal self adjoint} we obtain that
\begin{equation}\label{eq main HL leading linear 1}\inp{\fL_{P}\alpha_\epsilon}{v_P}_{L^2_{\mathring{P}_{\delta}}}=\int_{\partial \mathring{P}_{\delta}}\iota_{\alpha_\epsilon}\iota_{v_P}\Upsilon_{P_C}^*\psi=\int_{\partial \mathring{P}_{\delta}}\iota_{{s_\epsilon}_*\beta}\iota_{b_1\xi_1+b_2\xi_2}\Upsilon^*\psi(0)+O(\epsilon^{2+q}).
\end{equation}	
Now,
\begin{equation}\label{eq main HL leading linear 2}
	\int_{\partial \mathring{P}_{\delta}}\iota_{{s_\epsilon}_*\beta}\iota_{b_1\xi_1+b_2\xi_2}\Upsilon^*\psi(0)=\int_{\partial B(0,\delta)\cap C}\iota_{{s_\epsilon}_*\xi_1}\iota_{b_1\xi_1+b_2\xi_2}\Upsilon^*\psi(0)+O(\epsilon^{3-q}),
\end{equation}
and
\begin{equation}\label{eq main HL leading linear 3}
	\int_{\partial B(0,\delta)\cap C}\iota_{{s_\epsilon}_*\xi_1}\iota_{b_1\xi_1+b_2\xi_2}\Upsilon^*\psi(0)= \epsilon^2\int_{\Sigma}\inp{\partial_r\times\xi_1}{b_1\xi_1+b_2\xi_2}= b_2\epsilon^2\int_{\Sigma}\inp{\partial_r\times\xi_1}{\xi_2}.
\end{equation}

We will now estimate the remaining terms in \autoref{eq main HL leading order term}. Using \autoref{eq error main desing thm} we obtain that
\begin{align*}\abs{e^1_\epsilon}&=\abs{\inp{ \mathfrak L_{P_\epsilon}(\tilde \alpha_\epsilon-\alpha_\epsilon)-\mathfrak L_{P}(\tilde \alpha_\epsilon-\alpha_\epsilon)}{v_P}_{L^2_{\mathring{P}_{\epsilon}}}}+\abs{\inp{ \mathfrak L_{P}(\tilde \alpha_\epsilon-\alpha_\epsilon)}{v_P}_{L^2_{\mathring{P}_{\epsilon}}}}\\
&\lesssim \epsilon^{2+q\mu-\omega}+\int_{\partial \mathring{P}_{\epsilon}}\abs{\iota_{\tilde \alpha_\epsilon-\alpha_\epsilon}\iota_{v_P}\Upsilon_{P_C}^*\psi}\lesssim \epsilon^{1+q\mu-\omega}
\end{align*}	
and
\begin{equation*}
	\abs{e^2_\epsilon}\lesssim \epsilon^{2q\mu-2\omega}\abs{\log \epsilon}+\epsilon^{q\mu-\omega}\abs{t(\epsilon)-t_0}+\abs{t(\epsilon)-t_0}^2 \abs{\log \epsilon}\lesssim \epsilon^{2q\mu-2\omega}\abs{\log \epsilon}.
\end{equation*}
We obtain by a further computation that 
$$\abs{e^3_\epsilon}\lesssim \epsilon^{2q\mu},\ \ \abs{e^5_\epsilon}\lesssim \epsilon^{3q}$$
and
$$ \abs{e^4_\epsilon}\lesssim \int_{\partial \mathring{P}_{\epsilon}}\abs{\iota_{\alpha}\iota_{v_P}\Upsilon_{P_C}^*\psi}+ \epsilon^{2q\mu}\lesssim\epsilon^{1+\mu}+ \epsilon^{2q\mu}. $$
Hence by choosing $\omega$ sufficiently small and $q>\frac 23$, we obtain
\begin{equation}\label{eq main HL leading linear 4}
\abs{e^1_\epsilon}+\abs{e^2_\epsilon	}+\abs{e^3_\epsilon}+\abs{e^4_\epsilon}+\abs{e^5_\epsilon}=o(\epsilon^2).
\end{equation}
Since $\inp{f_\epsilon}{v_P}_{L^2_{\mathring{P}_{\epsilon}}}=\inp{f_P}{v_P}_{L^2}+o(1)$, therefore combining \autoref{eq main HL leading order term}, \autoref{eq main HL leading linear 1}, \autoref{eq main HL leading linear 2}, \autoref{eq main HL leading linear 3} and \autoref{eq main HL leading linear 4} we obtain
\begin{equation*}
	t^1(\epsilon)=t(\epsilon)=t_0-\frac{cb_2}{\inp{f_P}{v_P}_{L^2}} \epsilon^2+o(\epsilon^2),
\end{equation*}
where $c=\int_{\Sigma}\inp{J\xi_1}{\xi_2}\neq 0.$ Similarly we can prove for $t^2(\epsilon)$ and $t^3(\epsilon)$.  Only thing to notice here is that $\xi_3=-\xi_1-\xi_2$ as in \autoref{eg Harvey-Lawson special Lagrangians}.
\end{proof}

\subsection{Desingularizations for intersecting associatives}\label{subsection desing intersecting}
 Let $\bphi$ be a path of co-closed $G_2$-structures. Let $P\in \cM_{\operatorname{cs}}^\bphi$ be an associative submanifold in $(Y,\phi_{t_0})$ with a transverse unique self intersection. In other words, $P$ is a conically singular associative with singularity at a unique point, which is modeled on a union of two transverse associative planes. To resolve this intersection, we glue in a Lawlor neck--an asymptotically conical (AC) special Lagrangian submanifold asymptotic to the planes (see \autoref{eg Lawlor neck}).  
	
\begin{proof}[{\normalfont{\textbf{Proof of \autoref{thm desing intersection}}}}]In \autoref{eg s-ind Pair of transverse associative planes} we see that the union of two transverse associative planes has $d_{-1}=0$ and $V_\lambda=\{0\}$ for all $\lambda\in (-1,0)$. We glue the intersecting associative $P$ with the AC associative $L$ where $B^{-1}\cdot L$ is the special Lagrangian Lawlor neck in $\C^3$ from \autoref{eg Lawlor neck} and $B\in G_2$ from \autoref{eq B Lawlor}. Since the $\bn$-component (perpendicular to $B\cdot \C^3$) of any element in $\ker \fL_{L,0}$ vanishes and $\bphi\in \bsP^\dagger$ (see \autoref{def dag}), the equation \autoref{eq generic intersection} implies that the extended matching kernel and the matching kernel are same and isomorphic to $\ker \fL_{L,-1}$, because $\ker \fL_{P,0}$ vanishes.
By \autoref{prop lagrangian AC kernel}, we have
 $$\dim \ker \fL_{L,-1}= b^1(S^2\times \R)+b^0(S^2\amalg S^2)-1=1.$$
By \autoref{thm main desing}, we obtain that for sufficiently small $\epsilon>0$, there exists $t(\epsilon)\in (t_0-T_\epsilon,t_0+T_\epsilon)$ and smooth closed associative $\tilde P_{\epsilon,t(\epsilon)}$ in $(Y,\phi_{t(\epsilon)})$ such that $\tilde P_{\epsilon, t(\epsilon)}\to P$ in the sense of currents as $\epsilon\to 0$.  Moreover,
 $\tilde P_{\epsilon, t(\epsilon)}$ are diffeomorphic to $P_{\epsilon, t(\epsilon)}$ and hence are diffeomorphic to the connected sums as mentioned in the theorem.
 \end{proof}

\begin{remark}[{\textbf{Discussion on leading order term}}]\label{rmk discussion on leading order term}
 It remains to prove that the leading order term of $t(\epsilon)-t_0$ is of the required form if $\bphi\in \bsP^\ddagger$ (see \autoref{remark ddag} for the reason behind this restriction).  We can consider $$\alpha_\epsilon=\alpha^2_\epsilon:=\rho_{\delta} \Upsilon_*({s_\epsilon}_* \beta)+ (1-\rho_{\delta})(\epsilon^3u_P+\alpha).$$
 Here $\alpha$ and $\beta$ represents $P$ and $L$ repectively and $u_P:=(\Theta^C_P)^{-1}\hat u_P\in \ker \fL_{P,-2}$ with $$\beta_1^\pm=B\xi^\pm, l=1, \lambda_1=\nu=-2, \lambda_0=-4, \mu=2, s_0=2, q=\frac 25.$$  

By abusing notation we denote $B\xi^\pm$ by simply $\xi^\pm$. We use the notation $(\tilde \alpha_\epsilon,t(\epsilon))$ from \autoref{thm main desing} following \autoref{rmk improved approx asso desing}, which satisfies
 \begin{equation}\tilde \fF_{\epsilon}(\tilde \alpha_\epsilon,t(\epsilon)-t_0)=\mathfrak L_{P_\epsilon}(\tilde \alpha_\epsilon-\alpha_\epsilon)+(t(\epsilon)-t_0)f_\epsilon+  Q_\epsilon(\tilde \alpha_\epsilon-\alpha_\epsilon,t(\epsilon))+\mathfrak F_\epsilon(\alpha_\epsilon)=0.
 \end{equation}
 
 Define $\mathring{P}_{\epsilon}:={P_C}\setminus \Upsilon(B(0,2 \epsilon R_\infty))$ and $\mathring{P}_{\delta}:={P_C}\setminus \Upsilon(B(0,\delta))$. 
Over $\mathring{P}_{\delta}$, $$\mathfrak F_\epsilon(\alpha_\epsilon)=\fL_P(\alpha_\epsilon-\alpha)+Q_P(\alpha_\epsilon-\alpha).$$
Therefore,
\begin{equation}\label{eq main Lawlor leading order term}
	0=\inp{\tilde \fF_{\epsilon}(\tilde \alpha_\epsilon,t(\epsilon))}{u_P}_{L^2_{\mathring{P}_{\epsilon}}}=(t(\epsilon)-t_0)\inp{f_\epsilon}{u_P}_{L^2_{\mathring{P}_{\epsilon}}}+\inp{\fL_{P}\alpha_\epsilon}{u_P}_{L^2_{\mathring{P}_{\delta}}}+e^1_\epsilon+e^2_\epsilon+e^3_\epsilon+e^4_\epsilon+e^5_\epsilon
	\end{equation}
where $e^1_\epsilon:=\inp{ \mathfrak L_{P_\epsilon}(\tilde \alpha_\epsilon-\alpha_\epsilon)}{u_P}_{L^2_{\mathring{P}_{\epsilon}}}$, 
	$e^2_\epsilon:=\inp{ Q_\epsilon(\tilde \alpha_\epsilon-\alpha_\epsilon,t(\epsilon)-t_0)}{u_P}_{L^2_{\mathring{P}_{\epsilon}}}$,
	and$$e^3_\epsilon:=\inp{Q_P(\alpha_\epsilon-\alpha)}{u_P}_{L^2_{\mathring{P}_{\delta}}}, \ \  e^4_\epsilon=-\inp{\fL_{P}(\alpha)}{u_P}_{L^2_{\mathring{P}_{\epsilon}}}+\inp{\fL_{P}(\alpha)}{u_P}_{L^2_{\mathring{P}_{\epsilon}\setminus \mathring{P}_{\delta}}},\  e^5_\epsilon=\inp{\mathfrak F_\epsilon(\alpha_\epsilon)}{u_P}_{L^2_{\mathring{P}_{\epsilon}\setminus \mathring{P}_{\delta}}}.$$	
By \autoref{prop formal self adjoint} we obtain that $\inp{\fL_{P}\alpha_\epsilon}{u_P}_{L^2_{\mathring{P}_{\delta}}}$ is
\begin{equation}\label{eq main Lawlor leading linear 1}\int_{\partial \mathring{P}_{\delta}}\iota_{\alpha_\epsilon}\iota_{u_P}\Upsilon_{P_C}^*\psi=\epsilon^{3}\int_{\partial B_{\Pi^+}(0,\delta)}\iota_{\xi^+}\iota_{u_P-\xi^+}\Upsilon^*\psi-\epsilon^{3}\int_{\partial B_{\Pi^-}(0,\delta)}\iota_{\xi^-}\iota_{u_P-\xi^-}\Upsilon^*\psi+o(\epsilon^{3}).
\end{equation}	
Now,
\begin{equation}\label{eq main Lawlor leading linear 2}
	\int_{\partial B_{\Pi^\pm}(0,\delta)}\iota_{\xi^\pm}\iota_{u_P-\xi^\pm}\Upsilon^*\psi=\int_{\partial B_{\Pi^\pm}(0,\delta)}\iota_{\xi^\pm}\iota_{u_P-\xi^\pm}\psi_e+o(\epsilon)
\end{equation}
and
\begin{equation}\label{eq main Lawlor leading linear 3}
	\int_{\partial B_{\Pi^\pm}(0,\delta)}\iota_{\xi^\pm}\iota_{u_P-\xi^\pm}\psi_e=(u_P-\xi^\pm)(0)\cdot \bn+o(\epsilon).
\end{equation}

Define $v_P=(\Theta^C_P)^{-1}\hat v_P$. Then $\fL_Pv_P=f_P$ . We observe $\inp{f_\epsilon}{u_P}_{L^2_{\mathring{P}_{\epsilon}}}=\inp{f_P}{u_P}_{L^2_{\mathring{P}_{\delta}}}+o(1)$. Therefore, similar to \autoref{eq main Lawlor leading linear 1}, \autoref{eq main Lawlor leading linear 2}, \autoref{eq main Lawlor leading linear 3}  we obtain
 \begin{align}\label{eq main Lawlor leading linear 5}
 \inp{f_P}{u_P}_{L^2_{\mathring{P}_{\delta}}}&=\int_{\partial B_{\Pi^+}(0,\delta)}\iota_{v_P}\iota_{\xi^+}\psi_e-\int_{\partial B_{\Pi^-}(0,\delta)}\iota_{v_P}\iota_{\xi^-}\psi_e+o(1) \nonumber\\
 &=\big(v_P^+(0)-v_P^-(0)\big)\cdot \bn+o(1).
\end{align}

We can further compute 
\begin{equation}\label{eq main Lawlor leading linear 4}
\abs{e^3_\epsilon}+\abs{e^4_\epsilon}+\abs{e^5_\epsilon}=o(\epsilon^3).
\end{equation}

If, in addition, one had $\abs{e^1_\epsilon} + \abs{e^2_\epsilon} = o(\epsilon^3)$, then the leading-order expansion would coincide with the expression given in \autoref{eq t(epsilon) intersection}. However, it is \textbf{not} immediately evident from our desingularization theorem why the estimate in \autoref{eq main Lawlor leading linear 4} should hold. Using \autoref{eq error main desing thm}, we can only deduce the following:
\begin{align*}
\abs{e^1_\epsilon} &= \abs{\inp{ \mathfrak{L}_{P_\epsilon}(\tilde \alpha_\epsilon - \alpha_\epsilon) - \mathfrak{L}_{P}(\tilde \alpha_\epsilon - \alpha_\epsilon)}{u_P}_{L^2(\mathring{P}_{\epsilon})}} + \abs{\inp{ \mathfrak{L}_{P}(\tilde \alpha_\epsilon - \alpha_\epsilon)}{u_P}_{L^2(\mathring{P}_{\epsilon})}} \\
&\lesssim o(\epsilon^{3}) + \left\lvert \int_{\partial B_{\Pi^+}(0, 2\epsilon R_\infty)} \iota_{\tilde \alpha_\epsilon - \alpha_\epsilon} \iota_{\xi^+} \Upsilon^* \psi - \int_{\partial B_{\Pi^-}(0, 2\epsilon R_\infty)} \iota_{\tilde \alpha_\epsilon - \alpha_\epsilon} \iota_{\xi^-} \Upsilon^* \psi \right\rvert.
\end{align*}
In fact, \autoref{eq error main desing thm} does not imply that the final expression above must be $o(\epsilon^3)$. A similar issue arises for $e^2_\epsilon$ as well.

Although not obvious, a suitable adaptation of the desingularization theorem in this special case might still yield the estimate in \autoref{eq main Lawlor leading linear 4}, and could be pursued in future work.
\end{remark}

\appendix
\section{Lemma for quadratic estimates}\label{appendix}
The following lemma provides a key pointwise estimate used in the proofs of the quadratic estimates stated in \autoref{prop cs quadratic estimate}, \autoref{prop Holder cs quadratic estimate}, and \autoref{prop quadratic estimate with epsilon}.
\begin{lemma}\label{lem quadratic}Let $M$ be a $3$-dimensional oriented submanifold of an almost $G_2$-manifold $(Y,\phi)$. Let $\Upsilon_M:V_M\subset NM\to Y$ be a tubular neighbourhood map. There exists a constant $C>0$ such that if $w \in C^\infty(NM)$ and $u, v, s\in C^\infty(V_M)$ then over $\Gamma_s:=\graph s\subset V_M$ we have $$\abs{\iota_w\mathcal L_u\mathcal L_v\psi}\lesssim \abs{w}\Big(f_{1}\abs{u}\abs{v}
+f_{2}\big(\abs{u}\abs{\nabla v}+\abs{v}\abs{\nabla u}\big)+\abs{\nabla u}\abs{\nabla v}\abs{\psi}\Big),$$
with 
\begin{align*}&f_{1}:=\abs{\psi}\abs{\nabla B}+\abs{\psi}\abs{B}^2+\abs{B}\abs{\nabla\psi}+\abs{\nabla^2\psi}+\abs{R}\abs{\psi}\\
&f_{2}:=\abs{\nabla\psi}+\abs{B}\abs{\psi},
\end{align*}
where $R$ is the Riemann curvature tensor and $B:C^\infty(NM)\times C^\infty(NM)\to C^\infty(T{V_M})$ is defined by
 $B(u,v):=\nabla_uv$. Here $u$ is understood as the extension of the vector field $u$ over $NM$ by fiberwise translation. The $4$-form $\psi$ denotes $\Upsilon_{M}^*\psi$, and $\fL$ stands for the Lie derivative.

\end{lemma}
\begin{proof}
Let $\{e_1,e_2,e_3\}$ be a local orthonormal frame for $T\Gamma_s$. For a torsion free connection we have
$$\mathcal L_v\psi(w,e_1,e_2,e_3)=\nabla_v\psi(w,e_1,e_2,e_3)+\sum_{\substack{\text{cyclic}\\\text{permutations}}}\psi(w,\nabla_{e_1}v,e_2,e_3)\ +\psi(\nabla_w v,{e_1},e_2,e_3).$$
By definition we have
 $$\mathcal L_u\mathcal L_v\psi(w,e_1,e_2,e_3)=u\big(\mathcal L_v\psi(w,e_1,e_2,e_3)\big)-\sum_{\substack{\text{cyclic}\\\text{permutations}}}\mathcal L_v\psi(w,[u, e_1],e_2,e_3).$$
Let's combine the above. We obtain
 \begin{align*}
 	 u\big(\nabla_v\psi(w,e_1,e_2,e_3)\big)&=\nabla_u\nabla_v\psi(w,e_1,e_2,e_3)\\
 	&+\sum_{\substack{\text{cyclic}\\\text{permutations}}}\nabla_v\psi(w,{\nabla_{e_1}u+[u,e_1]},e_2,e_3) +\nabla_v\psi(\nabla_wu,e_1,e_2,e_3),
	\end{align*}
\begin{align*} 
u\big(\psi(w,\nabla_{e_1}v,e_2,e_3)\big)
&= \nabla_u\psi(w,\nabla_{e_1}v,e_2,e_3)
	+\psi(w,{\nabla_u\nabla_{e_1}v},e_2,e_3)+{\psi}(w,\nabla_{e_1}v,\nabla_ue_2,e_3)\\ 
	&+{\psi}(w,\nabla_{e_1}v,e_2,{\nabla_u{e_3}})+\psi(\nabla_wu,\nabla_{e_1}v,e_2,e_3),
\end{align*}
\begin{align*}u\big(\psi(\nabla_w v,{e_1},e_2,e_3)\big)&=\nabla_u\psi(\nabla_w v,{e_1},e_2,e_3)+\sum_{\substack{\text{cyclic}\\\text{permutations}}}\psi(\nabla_w v,{\nabla_u{e_1}},e_2,e_3)\\
& +\psi({\nabla_w\nabla_uv+R(u,w)v},{e_1},e_2,e_3),	
\end{align*}
and 
\begin{align*}
&\mathcal L_v\psi(w,[u, e_1],e_2,e_3)\\
&=\psi(w,{\nabla_{[u, e_1]}v},e_2,e_3) +\psi(\nabla_wv,{[u, e_1]},e_2,e_3)\\
&+ \psi(w,{[u, e_1]},\nabla_{e_2}v,e_3)+\psi(w,{[u, e_1]},e_2,\nabla_{e_3}v)
	+\nabla_v\psi(w,{[u, e_1]},e_2,e_3).
\end{align*}	
Finally, putting together after some cancellations we obtain
\begin{align*}
\abs{\mathcal L_u\mathcal L_v\psi(w,e_1,e_2,e_3)}
&\lesssim	\abs{w}(\abs{\nabla_u\nabla_v\psi}+\abs{\nabla\psi}\abs{\nabla u}\abs{v})+\abs{\nabla\psi}\abs{ v}\abs{\nabla_wu}\\
&+\abs{w}\abs{\nabla\psi}\abs{u}\abs{\nabla v}+\abs{\psi}\abs{\nabla v}\abs{\nabla_wu}\\
&+\abs{\nabla_wv}(\abs{\nabla\psi}\abs{u}+\abs{\psi}\abs{\nabla u})+\abs{\psi}(\abs{\nabla_{w}\nabla_u v}+\abs{R}\abs{w}\abs{u}\abs{v})\\
&+\abs{w}\abs{\psi}\Big(\sum_{i=1}^3\abs{\nabla_{e_i}(\nabla_uv)}+\abs{R}\abs{u}\abs{v}+\abs{\nabla u}\abs{\nabla v}\Big).
\end{align*}
Since,
$$\abs{\nabla_u\nabla_v\psi}\lesssim \abs{u}\abs{v}(\abs{B}\abs{\nabla\psi}+\abs{\nabla^2\psi}),$$
 $$\abs{\nabla_{e_i}(\nabla_uv)}\lesssim\abs{\nabla B}\abs{u}\abs{v}+\abs{B}\abs{u}\abs{\nabla v}+\abs{B}\abs{v}\abs{\nabla u},$$
 $$\abs{\nabla_{w}(\nabla_uv)}\lesssim\abs{w}\abs{u}\abs{v}\big(\abs{\nabla B}+\abs{B}^2),$$
 we have  \[\abs{\iota_w\mathcal L_u\mathcal L_v\psi}\lesssim f_{1}\abs{w}\abs{u}\abs{v}
+f_{2}\abs{w}(\abs{u}\abs{\nabla^\perp v}+\abs{v}\abs{\nabla^\perp u})+\abs{w}\abs{\nabla^\perp u}\abs{\nabla^\perp v}\abs{\psi}.\qedhere\]
\end{proof}

\printreferences

\end{document}
